\documentclass[aos]{imsart}

\RequirePackage{amsthm,amsmath,amsfonts,amssymb}
\RequirePackage[numbers]{natbib}
\RequirePackage[colorlinks,citecolor=blue,urlcolor=blue]{hyperref}
\RequirePackage{graphicx}

\usepackage{blindtext}
\usepackage[normalem]{ulem}
\usepackage{cancel}
\usepackage{mathrsfs}

\usepackage[mathscr]{euscript}
\usepackage{amsthm}
\usepackage{amsmath}
\usepackage{amsfonts}
\usepackage{bbm}
\usepackage{xcolor}
\usepackage[numbers]{natbib}
\usepackage{comment}
\usepackage{float}
\usepackage{bbm}
\usepackage{mhequ}
\usepackage{relsize}
\usepackage{tikz}
\usepackage{mathtools}
\usetikzlibrary{arrows.meta, decorations.pathmorphing, positioning, decorations.markings}

\usepackage{xr}

\DeclareMathOperator*{\argmax}{arg\,max}
\DeclareMathOperator*{\argmin}{arg\,min}

\newcommand{\bR}{\mathbb{R}}
\newcommand{\bN}{\mathbb{N}}
\newcommand{\bM}{\mathbb{M}}
\newcommand{\bI}{\mathbb{I}}

\newcommand{\cT}{\mathcal{T}}

\newcommand{\cB}{\mathcal{B}}
\newcommand{\cC}{\mathcal{C}}
\newcommand{\cD}{\mathcal{D}}
\newcommand{\cG}{\mathcal{G}}

\newcommand{\cS}{\mathcal{S}}
\newcommand{\cJ}{\mathcal{J}}
\newcommand{\cR}{\mathcal{R}}

\newcommand{\cE}{\mathcal{E}}

\newcommand{\cP}{\mathcal{P}}

\newcommand{\cU}{\mathcal{U}}
\newcommand{\cI}{\mathcal{I}}
\newcommand{\cV}{\mathcal{V}}
\newcommand{\cL}{\mathcal{L}}

\newcommand{\sfE}{{\sf E}}
\newcommand{\Var}{{\sf Var}}
\newcommand{\Cov}{{\sf Cov}}
\newcommand{\sfQ}{{\sf Q}}
\newcommand{\sfP}{{\sf P}}

\newcommand{\sfBF}{{\sf BF}}
\newcommand{\sfp}{{\sf p}}

\newcommand{\pa}{{\rm pa}}
\newcommand{\an}{{\rm an}}
\newcommand{\de}{{\rm de}}
\newcommand{\supp}{{\rm supp}}

\newcommand{\rt}{\right}
\newcommand{\lt}{\left}

\makeatletter
\newcommand*{\indep}{%
  \mathbin{%
    \mathpalette{\@indep}{}%
  }%
}
\newcommand*{\nindep}{%
  \mathbin{
    \mathpalette{\@indep}{\not}
  }%
}
\newcommand*{\@indep}[2]{%
  \sbox0{$#1\perp\m@th$}
  \sbox2{$#1=$}
  \sbox4{$#1\vcenter{}$}
  \rlap{\copy0}
  \dimen@=\dimexpr\ht2-\ht4-.2pt\relax
  \kern\dimen@
  {#2}%
  \kern\dimen@
  \copy0 
} 
\makeatother

\newcommand{\ba}{\begin{align*}}
\newcommand{\ea}{\end{align*}}

\newcommand{\be}{\begin{equs}}
\newcommand{\ee}{\end{equs}}

\startlocaldefs

\newtheorem{theorem}{Theorem}[section]
\newtheorem{lemma}[theorem]{Lemma}

\newtheorem{corollary}{Corollary}[section]
\newtheorem{proposition}{Proposition}[section]

\theoremstyle{remark}
\newtheorem{remark}{Remark}[section]
\newtheorem{definition}[theorem]{Definition}
\newtheorem{example}{Example}[section]

\endlocaldefs

\begin{document}

\begin{frontmatter}
\title{Consistent DAG selection for Bayesian Causal Discovery under general error distributions}
\runtitle{Bayesian Causal Discovery}

\begin{aug}
\author{\fnms{Anamitra}~\snm{Chaudhuri}\ead[label=e1]{ac27@tamu.edu}},
\author{\fnms{Anirban}~\snm{Bhattacharya}\ead[label=e2]{anirbanb@stat.tamu.edu}}
\and
\author{\fnms{Yang}~\snm{Ni}\ead[label=e3]{yni@stat.tamu.edu}}
\address{Department of Statistics, Texas A\&M University, \printead{e1,e2,e3}}
\end{aug}

\begin{abstract}
We consider the problem of learning the underlying causal structure among a set of variables, which are assumed to follow a Bayesian network or, more specifically, a linear recursive structural equation model (SEM)  with the associated errors being independent and allowed to be non-Gaussian. A Bayesian hierarchical model is proposed to identify the true data-generating directed acyclic graph (DAG) structure where the nodes and edges represent the variables and the direct causal effects, respectively. Moreover, incorporating the information of non-Gaussian errors, we characterize the distribution equivalence class of the true DAG, which specifies the best possible extent to which the DAG can be identified 
based on purely observational data.
Furthermore, under the consideration that the errors are distributed as some scale mixture of Gaussian, where the mixing distribution is unspecified, and mild distributional assumptions, we establish that by employing a non-standard DAG prior, the posterior probability of the distribution equivalence class of the true DAG converges to unity as the sample size grows. This shows that the proposed method achieves the posterior DAG selection consistency, which is further illustrated with examples and simulation studies. 
\end{abstract}



\end{frontmatter}

\section{Introduction}

Learning causal structure in complex systems is a fundamental challenge across a broad range of disciplines, from traditional scientific fields to modern engineering and technology. 
Unlike conventional statistical methods that focus merely on correlation, the field of causal discovery primarily considers the problem of discovering the directionality and strength of causal relationships between variables, often from observational data. Thus, it has become a critical tool for researchers aiming to predict the effects of interventions on the systems, especially where controlled experimentation may be expensive, unethical, or even infeasible. Such necessities arise not only in various areas of natural science, such as epidemiology \cite{Petersen2024epi}, public health \cite{shen2020challenges},  genomics \cite{choi2023model}, neuroscience \cite{zhou2023functional}, and climate and environmental science \cite{runge2019inferring}, but also in numerous domains in social science, such as psychology \citep{ni2025causal}, philosophy \cite{glymour2019review}, and economics \cite{imbens2004nonparametric}. Moreover, with recent advances in science and technology and the increase in size and complexity of data generation processes, causal discovery has acquired significant relevance in the fields of machine learning \cite{berhanrd2012caus_antcaus} and artificial intelligence 
\cite{xia2021causal, zevcevic2021relating} 
through various emerging areas such as causal representation learning \cite{scholkopf2021toward, zhang2024causal}, causal transfer learning \cite{zhang2017transfer}, causal algorithmic fairness \cite{zhang2018fairness}, and causal reinforcement learning \cite{bareinboim2015bandits}.

This work focuses on learning causal structures from purely observational data within the framework of causal Bayesian networks, which are widely used to represent causal relationships among variables through directed acyclic graphs (DAGs).
This is, in general, a nontrivial and difficult task due to the vast number of potential DAG structures and multiple DAGs representing the same set of conditional independence relationships. In fact, DAGs are generally identifiable only up to their corresponding Markov equivalence class, in which all DAGs encode the same conditional independencies \cite{heckerman1995learning}.  

Numerous methods have been proposed in the past (see reviews such as \cite{drton2017structure}) to 
estimate the Markov equivalence class, which can be broadly classified as constraint-based, score-based, and hybrid methods. Constraint-based approaches such as the widely used PC algorithm \cite{spirtes2001causation} and its high-dimensional variants \cite{kalisch2007estimating, maathuis2009estimating, harris2013pc}, the FCI algorithm \cite{spirtes2001anytime}, and the RFCI algorithm \cite{colombo2012learning} aim to infer the underlying conditional independencies based on hypothesis testing. Score-based methods aim to maximize certain scoring criteria over the space of models, viz., the DAGs, their equivalence classes, or their causal orderings, generally through some search procedure. One of the most notable is the GES algorithm \cite{chickering2002optimal} that performs a two-stage greedy search over the space of equivalence classes to obtain the best-scored one. An alternative popular approach along this direction is Bayesian structure learning, which utilizes Markov chain Monte Carlo algorithm to search over the model space, enabling posterior inference on relevant quantities through model averaging; see, for example, the series of works \cite{madigan1996bayesian, geiger2002parameter, friedman2003being, castelletti2018learning}. Hybrid methods such as \cite{tsamardinos2006max, schmidt2007learning, alonso2013scaling} combine these two approaches by deploying a score-based search algorithm over a restricted space estimated via conditional independence tests. One common thread of the aforementioned methods is that they all aim to infer Markov equivalence classes, which may contain DAGs with significantly different causal interpretations \cite{wang2020high} and can be quite large \cite{andersson1997characterization}.

Recent studies have discovered that under additional distributional assumptions, the exact DAG structure, rather than 
the associated Markov equivalence class, can be recovered solely from observational data. To be specific, in the case of continuous variables, it is a popular choice to represent the causal structure using a structural equation model (SEM). In their seminal work, Shimizu et al. \cite{shimizu2006linear} proposed the linear non-Gaussian acyclic model, abbreviated as LiNGAM, where the functional form of the SEM is linear and the errors are non-Gaussian. They show that the underlying DAG can be uniquely identified under their model by establishing the equivalence between LiNGAM and independent component analysis (ICA)\cite{comon1994independent}. 
In a similar vein, the unique identification of the underlying DAG is possible if the functional form of the SEM is non-linear with some mild regularity assumptions on the function and noise \cite{hoyer2008nonlinear, peters2011identifiability, peters2014causal} or if the functional form is linear and the errors have equal variances \cite{peters2014identifiability, chen2019causal, loh2014high, rothenhausler2018causal}.

Historically, Bayesian DAG structure learning methods have been primarily focused on developing efficient computational algorithms for Gaussian DAG models 
\cite{giudici2003improving, grzegorczyk2008improving, niinimaki2011partial, su2016improving, goudie2016gibbs, kuipers2017partition}. Only recently, a few studies \cite{cao2019posterior, lee2019minimax, zhou2023complexity} established the consistency of such Bayesian approaches. However, for non-Gaussian DAG models, there are significantly 
fewer works \cite{hoyer2009bayesian, shimizu2014bayesian}. Although these works already showed via extensive simulations that, in general, their performance is significantly better than the existing non-Bayesian methods under a vast range of non-Gaussian distributions, a rigorous Bayesian DAG selection method with some desired statistical property, e.g., consistency, is lacking in this context, as it
comes with several inherent challenges such as modeling the errors with some appropriate non-Gaussian distribution, analytical intractability of the marginal likelihood, and asymptotic analysis of Bayes factors under model misspecification. 

In this work, we address this research gap by considering a linear acyclic SEM, where the associated errors are independent and not necessarily Gaussian, with the objective of proposing a Bayesian method that consistently recovers the true underlying DAG or its equivalence class under mild assumptions. 
More specifically, we develop a method that not only takes advantage of non-Gaussianity for finer identifiability but also is more general than LiNGAM in that we allow for the possible presence of Gaussian errors. In order to precisely characterize this as well as relax the restriction of all true errors being non-Gaussian, we assume that the errors in the data-generating process are distributed as a scale mixture of Gaussian with some unknown mixing distribution.
This is generally considered to be a popular and appropriate choice \cite{box2011bayesian, west1984outlier, andrews1974scale} for the error distributions not only because of symmetry around the origin, but also due to its comprehensive representation that encompasses a large family of distributions including well-known distributions such as Laplace, Student's t, Cauchy, the family of stable and exponential power distributions \cite{west1987scale}, and the scale mixture thereof. 
These prominently include polynomial-tailed distributions, 
thereby relaxing the restriction of log-concavity in the existing literature \cite{wang2020high}.
Under this consideration, we propose a Bayesian hierarchical model where the variables are generated by an SEM with the errors being modeled via {\it Laplace distributions} with unknown scale parameters, thereby rendering it a misspecified (working) model. Nevertheless, this misspecification is intentional as it offers significant advantages for identifiability and asymptotic analysis, exploiting the advantages of the Laplace distribution despite the inevitable analytical challenges arising from the associated likelihood function. Specifically, we address the intractability of the marginal likelihoods by establishing Laplace approximations \cite{tierney1986accurate}, which is non-trivial in this context, particularly due to the non-smoothness of the log-likelihood functions. Importantly, the deterministic quantities appearing in the approximation result possess convenient expressions due to favorable properties of the Laplace distribution, facilitating the establishment of our identifiability theory by seamlessly transitioning between the probabilistic and graph-theoretic aspects of the working model, which is less evident with other parametric or semiparametric non-Gaussian error models. 
Regarding identifiability, we characterize the {\it distribution equivalence class}, new notions of {\it risk equivalence class} and {\it minimal risk equivalence class}, and their relationships, specifying the best possible extent to which the true underlying DAG can be identified based on purely observational data. These characterizations are presented in a suite of new identifiability results that capture subtle interactions between our working model and the postulated ground truth. 
Furthermore, we propose a non-standard prior over the families of DAGs, which imposes a penalty on the number of edges and ensures that, only under the finite second moment assumption of the errors, the posterior probability of the distribution equivalence class tends to unity as the sample size grows. In this way, we establish that the proposed method achieves the desired posterior DAG selection consistency over a broad semiparametric class of ground truths, and 
finally illustrate the theoretical results with concrete examples and simulation studies. Our DAG identifiability and selection consistency results encompass linear Gaussian DAGs,  linear non-Gaussian DAGs, and linear DAGs with both Gaussian and non-Gaussian errors. The DAG selection consistency (Bayesian or frequentist) in the latter two cases is novel in the literature to the best of our knowledge. From a broader perspective, we add to the growing literature of Bayesian model selection consistency in graphical and non-Gaussian models, and under model misspecification \cite{rossell2018tractable,geng2019probabilistic,niu2020bayesian,rossell2021approximate}. 


The remainder of this paper is organized as follows. In Section \ref{sec:prob_form} we formally describe the acyclic structural equation model that we consider as our causal model in this paper, and state our main objective. Furthermore, we introduce the proposed Bayesian hierarchical model in Section \ref{sec:prop_meth}, and state our identifiability results in Section \ref{sec:identif}. Next, we establish the asymptotic properties of our method, namely the Bayes factor consistency and posterior consistency, in Section \ref{sec:asym_prop}. The results of simulation studies are presented in Section \ref{sec:sim_stud}, and finally, we conclude and discuss potential extensions of this work in Section \ref{sec:cnclsn}. All proofs are presented in the Appendix, 
where we also include auxiliary results of independent theoretical interest.

\section{Problem formulation}\label{sec:prob_form}

\paragraph*{Preliminaries}
We denote the set of real numbers by $\bR$ and the set of natural numbers by $\bN := \{1, 2, \dots\}$, and for any $n \in \bN$, we denote $[n]:= \{1, 2, \dots, n\}$. A DAG is denoted by a tuple $\gamma=(V, E)$ where $V=[p]$ is the set of $p$ nodes and $E\subset V\times V$ is the set of directed edges, i.e., $(k, j) \in E$ if there is a directed edge from node $k$ to node $j$ in $\gamma$, which will be denoted by $(k\to j)\in \gamma$ throughout the rest of the paper for simplicity.
The family of all DAGs with $p$ nodes is denoted by $\Gamma^p$.
We call node $k$ a \textit{parent} of node $j$ in $\gamma$ if $(k \to j) \in \gamma$, and the set of parents of node $j$ is denoted by $\pa^\gamma(j)$. The total number of edges in $\gamma$ is denoted by $|\gamma|$. We call any DAG $\gamma' \in \Gamma^p$ with edge set $E'$ a \emph{supergraph} of $\gamma$, denoted by $\gamma'\supseteq \gamma$ with a slight abuse of notation, if $E' \supseteq E$, i.e., every directed edge in $\gamma$ is present in $\gamma'$, and collect them within the class
$\cS^\gamma := \{\gamma' \in \Gamma^p : \gamma'\supseteq \gamma\}$.
The set of conditional independence relationships encoded by $\gamma$ (via the notion of d-separation) is denoted by $\bI(\gamma)$, and any $\gamma' \in \Gamma^p$ is said to be \textit{Markov equivalent} to $\gamma$ if $\bI(\gamma) = \bI(\gamma')$.
Finally, the family of all permutations of $[p]$ is denoted by $\cT_p$. We use $\mbox{Laplace}(0, 1)$ to denote the standard Laplace (or Double-Exponential) distribution with density function $(2x)^{-1} e^{-|x|}$ for $x \in \bR$. 

\subsection{Structural causal model}

Consider $p$ random variables $X_j, j \in [p]$.
We assume that they are generated by a linear recursive SEM governed by a data-generating true DAG $\gamma^* \in \Gamma^p$ with nodes $[p]$ representing the set of random variables and edges $E^*$ representing their direct causal relationships -- for every $j, k \in [p]$, $(k \to j) \in \gamma^*$ if $X_k$ has a \textit{direct linear (causal) effect} on $X_j$.
Consequently, there exists a permutation of the variables $\sigma^* \in \cT_p$, which we refer to as the \textit{causal order} of the variables, such that 
$(k \to j) \in \gamma^*$ only if $\sigma^*(k) < \sigma^*(j)$. Therefore, the parents of a node always have lower causal orders than the node itself. 

Letting $\pa^*(j) \equiv \pa^{\gamma^*}(j)$ denote the parent set of node $j$ in $\gamma^*$ for every $j \in [p]$, the SEM assumes that 
$X_{j}$ is some (unknown) linear function of $X_{k}, k \in \pa^*(j)$, plus an (unobserved) independent random error variable $\epsilon_{j}$,
\begin{align}\label{eq:model}
\qquad X_{j} = \sum_{k \in {\rm pa}^*(j)} \beta^*_{jk} X_{k} + \epsilon_{j}\quad \text{with} \quad \epsilon_j \;\; {\buildrel\rm ind\over\sim} \;\; \sfP^*_j,
\end{align}
where the (unknown) {\it non-zero} SEM coefficient $\beta^*_{jk} \in \bR$ quantifies the direct causal effect of $X_k$ on $X_j$.
We consider $n$ independent and identically distributed (iid) observations of the random vector $X = (X_1, X_2, \dots X_p)$, denoted by $X^{(i)} = (X_{1}^{(i)}, X_{2}^{(i)}, \dots, X_{p}^{(i)}),
\ i \in [n]$, and let $D_n := \{X^{(i)} : i \in [n]\}$ denote the complete dataset.

\paragraph*{Error distribution}
It is well known that observational data alone may not distinguish DAGs from each other as they are generally identifiable only up to the Markov equivalence class \cite{heckerman1995learning}. 
For instance, if the implied joint distribution of the SEM is Gaussian, then Markov equivalence implies distribution equivalence \cite{geiger2002parameter}, and thus, neither conditional independence tests nor likelihood-based scores can differentiate between Markov equivalent DAGs. On the other hand, 
as shown in \cite{shimizu2006linear}, LiNGAM, i.e., model \eqref{eq:model} with $\sfP_j^*$ being non-Gaussian for every $j \in [p]$, allows for unique identification of $\gamma^*$.
However, to the best of our knowledge, there is no existing theory or method rigorously studying the case where an arbitrary subset of the errors is Gaussian and the rest are non-Gaussian. In this article, we do not impose any restriction on the number of non-Gaussian errors, unlike LiNGAM. In such situations, one may expect the extent of identifiability to lie in between unique DAG identifiability and identifiability up to Markov equivalence classes, and we rigorously characterize this phenomenon. 

For each error distribution $\sfP^*_j$, we assume it follows some scale mixture of Gaussian, where the mixing distribution is unknown. Such scale mixtures are a popular and flexible class for representing error distributions; see Remark \ref{rem:sc_mix} below. 
Formally, the distributions of errors \ $\sfP^*_j, \ j \in [p]$ can be expressed as
\begin{align}\label{eq:error_dist}
\epsilon_j \ | \ \lambda_j \;\; \sim \;\; \text{N}(0, \lambda_j^2) \quad \text{with} \quad \lambda_j \;\; {\buildrel\rm ind\over\sim} \;\; \sfQ_j^*,
\end{align}
where each $\sfQ_j^*$ is a probability distribution on $(0, \infty)$. 
Under the above representation, $\epsilon_j$ is Gaussian if and only if $\lambda_j$ is a degenerate random variable, i.e., $\sfQ_j^*$ is a point mass. 
We denote by $n\cG^*$ the set of nodes in $\gamma^*$ corresponding to the non-Gaussian errors, that is,
\begin{align}\label{eq:nG*}
n\cG^* := \{j \in [p] : \;\; \epsilon_j \text{ in \eqref{eq:error_dist} is non-Gaussian, i.e.,} \;\;\; \lambda_j \;\; \text{is non-degenerate}\}.
\end{align}
\begin{remark}[Generality of the scale mixture of Gaussians]\label{rem:sc_mix}
The scale mixture of Gaussians is a widely recognized choice \cite{west1987scale} for error distributions due to several reasons.
First, the distributions represented by \eqref{eq:error_dist} are continuous, unimodal, and symmetric with respect to $0$, which is a generally desirable property for error distributions. Second, the scale mixture representation in \eqref{eq:error_dist} is highly flexible and encompasses a wide class of well-known distributions such as contaminated Gaussian, Laplace, Student's t, Cauchy, and Logistic, \cite{andrews1974scale, west1984outlier} and the scale mixtures of those well-known distributions. More generally, it has been shown \cite{west1987scale} that the class of Gaussian scale mixtures includes the distributions from the symmetric stable family as well as the exponential power family \cite{box2011bayesian}. In particular, the scale mixture family includes polynomial-tailed distributions, thereby relaxing the log-concavity assumption in the related works, e.g., \cite{wang2020high}.
\end{remark}
From \eqref{eq:error_dist}, $\sfP^*_j$ admits a probability density ${\sfp^*_j}$ with respect to the Lebesgue measure, 
\begin{align*}
\sfp^*_j(x) = \int_0^\infty \frac{1}{\lambda} \phi\lt(\frac{x}{\lambda}\rt) d\sfQ^*_j(\lambda), \quad x \in \bR, \ j \in [p],
\end{align*}
where $\phi(\cdot)$ denotes the density of a standard Gaussian distribution. 
Furthermore, due to the independence of the errors, $\sfP^*$, the joint probability distribution of the errors, is given by $\sfP^* = \otimes_{j \in [p]} \sfP^*_j$.
Then the joint probability distribution $\sfP_X^*$ of $X$ induced by $\sfP^*$ admits a joint density $\sfp_X^*$ given by
\begin{align}\label{eq:true_dens}
\sfp_X^*(x) = \prod_{j \in [p]} \sfp^*_j\lt(x_j - \sum_{k \in {\rm pa}^*(j)} \beta^*_{jk} x_{k}\rt), \qquad x = (x_1, x_2, \dots, x_p) \in \bR^p, 
\end{align}
which is known as the Bayesian network factorization.
We illustrate the above in a concrete example. 

\begin{example}\label{ex:4node}
Consider $p = 4$ with $\gamma^*$ being the DAG as shown in Figure \ref{fig:4node}, 
 \begin{figure}[H]
\centering
\begin{tikzpicture}[->, >=stealth, thick]
\node (3) [circle, draw, fill=black, text=white, minimum size=1cm] at (1,1.732) {3}; 
\node (2) [circle, draw, fill=black, text=white, minimum size=1cm] at (0,0) {2}; 
\node (1) [circle, draw, minimum size=1cm] at (2,0) {1}; 
\node (4) [circle, draw, fill=black, text=white, minimum size=1cm] at (4,0) {4};
\draw (3) -- (2);
\draw (3) -- (1);
\end{tikzpicture}        
\caption{DAG $\gamma^*$ with the nodes in $n\cG^*$ marked in black.}
\label{fig:4node}
\end{figure}
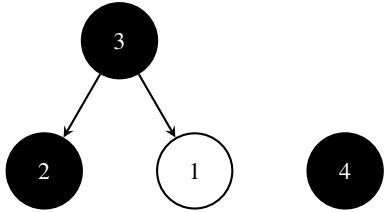
\noindent and let the associated data-generating SEM in \eqref{eq:model} take the following specific form: 
\begin{align*}
X_3 &= \epsilon_3,\\
X_2 &= 1.5 X_3 + \epsilon_2,\\
X_1 &= -3.2 X_3 + \epsilon_1,\\
X_4 &= \epsilon_4,
\end{align*} 
where the error distributions are: 
\begin{align*}
\epsilon_1 \sim \text{N}(0, 2.8), \quad \epsilon_2 \sim \text{Laplace}(0, 1), \quad \epsilon_3 \sim \text{t}_2, \quad \text{and} \quad \epsilon_4 \sim \frac{3}{4} \ \text{N}(0, 1) + \frac{1}{4} \ \text{N}(0, 4),
\end{align*}
with $\text{t}_2$ being the Student's t-distribution with degrees of freedom 2.
That is, in view of \eqref{eq:error_dist}, the distributions $\sfQ_1^*, \sfQ_2^*$, $\sfQ_3^*$ and $\sfQ_4^*$ are such that
\begin{align*}
\lambda_1^2 = 2.8 \quad \text{w.p.} \;\; 1,
\qquad 
\lambda_2^2 \sim \text{Exp}(2), \qquad
\lambda_3^2 \sim \text{Inv.G}(1, 1), \quad \text{and} \quad
\lambda_4^2 =
\begin{cases}
1 \quad \text{w.p.} \;\; \frac{3}{4}\\
4 \quad \text{w.p.} \;\; \frac{1}{4}
\end{cases},
\end{align*}
where $\text{Inv.G}$ is the inverse-gamma distribution,
and according to \eqref{eq:nG*}, we have $n\cG^* = \{2, 3, 4\}$.
\end{example}

In the rest of the paper, we follow the convention in Figure \ref{fig:4node}, marking in black the nodes corresponding to the non-Gaussian errors.


\subsection{Consistent DAG selection}

The goal of causal discovery is to identify the true underlying DAG $\gamma^*$ based on purely observational data $D_n$. However, the data-generating distribution $\sfP_X^*$ may be equivalently represented by  DAGs apart from $\gamma^*$. As a consequence, the exact recovery of $\gamma^*$ is infeasible without further assumptions, and it is generally possible to identify $\gamma^*$ only up to a certain class of DAGs, namely the \textit{distribution equivalence class}. In order to elucidate this, we first present the notion of distribution equivalence as follows.

\paragraph*{Distribution equivalence}

Regarding $\sfP_X^*$, the information about its underlying causal structure and non-Gaussianity is specified by $\gamma^*$ and $n\cG^*$, respectively, through \eqref{eq:model} and \eqref{eq:nG*}. Thus, we formally encode such a specification by the tuple $(\gamma^*, n\cG^*)$, and $\sfP_X^*$ is said to be \textit{represented} by $(\gamma^*, n\cG^*)$. 
More generally, for any DAG $\gamma \in \Gamma^p$ and $n\cG \subseteq [p]$, we denote by $\cP(\gamma, n\cG)$ the family of distributions of $X$ that are represented by $(\gamma, n\cG)$, that is,
\begin{align*}
\cP(\gamma, n\cG) := \{\sfP_X \, : \; &\sfP_X \; \text{is the distribution of $X$ under which $X_j, j \in [p]$ are generated}\\
&\text{by some linear recursive SEM represented by $\gamma$ such that}\\
&\text{the error corresponding to node $j$ is non-Gaussian if and only if $j\in n\cG$}\}.
\end{align*}
For instance, following \eqref{eq:model} and \eqref{eq:nG*}, clearly $\sfP^*_X \in \cP(\gamma^*, n\cG^*)$. We now define below the concept of distribution equivalence.

\begin{definition}[Distribution equivalence]
For any $\gamma, \gamma' \in \Gamma^p$ and $n\cG, n\cG' \subseteq [p]$, the tuples $(\gamma, n\cG)$ and $(\gamma', n\cG')$ are called \textit{distribution equivalent} if $\cP(\gamma, n\cG) = \cP(\gamma', n\cG')$, that is,
every $\sfP_X \in \cP(\gamma, n\cG)$ can be alternatively represented by $(\gamma', n\cG')$, and vice versa.
\end{definition}

Subsequently, we define the \textit{distribution equivalence class}  $\cE(\gamma^*, n\cG^*)$ of $(\gamma^*, n\cG^*)$,
\begin{align}\label{eq:dist_equiv}
\begin{split}
\cE(\gamma^*, n\cG^*) := \{\gamma \in \Gamma^p : \cP(\gamma, n\cG) = \cP(\gamma^*, n\cG^*)\;\text{for some}\;n\cG \subseteq [p]\}. 
 \end{split}
\end{align}
Therefore, following the above definition, the underlying distribution $\sfP_X^*$ can be represented by any DAG $\gamma \in \cE(\gamma^*, n\cG^*)$ in addition to $\gamma^*$, and hence, it is impossible to distinguish between $\gamma^*$ and $\gamma$ by their distributions. This indicates that $\cE(\gamma^*, n\cG^*)$ is the best possible extent of identification to achieve, and thereby, we call a DAG selection method to be \textit{consistent} if the estimated DAG tends to be only inside of $\cE(\gamma^*, n\cG^*)$ as the sample size grows.

\paragraph*{Objectives of this paper}

Our goal is to develop a Bayesian hierarchical model that achieves \textit{posterior DAG selection consistency}, that is,
\begin{align}\label{eq:goal_pmsc}
\text{posterior probability of $\cE(\gamma^*, n\cG^*)$} \;\; \to \;\; 1 \quad \text{in $\sfP^*$-probability \; as $n \to \infty$.}
\end{align}
In the above, we assume the number of nodes $p$ to be fixed, and focus on establishing selection consistency over the nonparametric class of Gaussian scale-mixture errors $\sfP^*$. 
One specific instance of this consistency result is when $n\cG^*=[p]$, i.e., all errors are non-Gaussian.
To the best of our knowledge, such consistency result is novel even in the frequentist non-Gaussian DAG literature.

In addition, en route to the establishment of the consistency result, we provide a new characterization of the distribution equivalence class $\cE(\gamma^*, n\cG^*)$ through a novel mixed graph representation, which generalizes the existing identifiability results of Gaussian DAGs (all errors are Gaussian) and LiNGAM (all errors are non-Gaussian) to the case of arbitrary presence of Gaussian and non-Gaussian errors.

\section{Proposed method}\label{sec:prop_meth}

In this section, we propose a family of Bayesian hierarchical models for selecting DAGs in $\Gamma^p$, and provide a sketch of proof of posterior DAG selection consistency, which is rigorously established in Section \ref{sec:asym_prop}. 

\subsection{Bayesian hierarchical model}

For a given DAG $\gamma \in \Gamma^p$, we consider that the observations $X^{(i)}, \ i \in [n]$ are iid and follow the \textit{Laplace-error} SEM $\bM_\gamma$ with real SEM coefficients $b_{jk}^\gamma, k \in \pa^\gamma(j), \, j \in [p]$, and positive scale parameters $\theta_j^\gamma, j \in [p]$ along with their corresponding prior distributions,
\begin{align}\label{eq:lap_model_alt}
\begin{split}
\bM_\gamma: \quad \qquad \qquad \quad \quad  \text{SEM:} \qquad 
&X_j = \sum_{k \in {\rm pa}^\gamma(j)} b_{jk}^\gamma X_{k} + e_j^\gamma, \quad j \in [p], \\ 
& e_j^\gamma/\theta_j^\gamma \;\; \overset{\rm iid}{\sim} \;\; \text{Laplace} \, (0, 1),\\
\text{coefficient prior:} \qquad &b_{j}^\gamma \;\; \overset{\rm ind}{\sim}  \;\; \pi_{b, j}^\gamma(\cdot),\\
\text{scale prior:} \qquad &\theta_j^\gamma \;\;\overset{\rm iid}{\sim} \;\; \pi_\theta^\gamma(\cdot).\\
\end{split}
\end{align}
where $b_j^\gamma := (b_{jk}^\gamma : k \in \pa^\gamma(j)), \, j \in [p]$.
We treat $\bM_\gamma$ as our \textit{working model} and emphasize here that the true data-generating errors need not be distributed as Laplace; see Remark \ref{rem:mis_specif} below for further discussions on this point.

We collect the coefficients as $b^\gamma := (b_{jk}^\gamma : j \in [p], k \in \pa^\gamma(j))$ and the scale parameters as $\theta^\gamma := (\theta_j^\gamma : j \in [p])$, and in particular, when $\gamma = \gamma^*$, we denote them as $b^* := (b^*_{jk} : j \in [p], \, k \in \pa^*(j))$ and $\theta^* := (\theta^*_j : j \in [p])$, respectively. 
Thus, the joint density of $X$ under the working model $\bM_\gamma$ in \eqref{eq:lap_model_alt} is given by 
\begin{align}\label{eq:pX}
f^\gamma(x | b^\gamma, \theta^\gamma, \gamma) &= \prod_{j \in [p]} \frac{1}{2\theta_j^\gamma} \exp\Bigg(-\frac{1}{\theta_j^\gamma} \bigg|x_j - \sum_{k \in {\rm pa}^\gamma(j)} b_{jk}^\gamma x_{k}\bigg|\Bigg),
\end{align}
where $x = (x_1, x_2, \dots, x_p) \in \bR^p$, and
in particular, when $\gamma = \gamma^*$, we denote the above as $f^*(x |b^*, \theta^*, \gamma^*)$.
Subsequently, by \eqref{eq:pX}, the likelihood function of data 
$D_n$ is given by 
\begin{align}\label{eq:likeli_L}
\cL\lt(D_n\rt | b^\gamma, \theta^\gamma, \gamma) = \Big(2^p\prod_{j \in [p]} \theta_j^\gamma\Big)^{-n} \exp\Bigg(-\sum_{j \in [p]}\frac{1}{\theta_j^\gamma} \sum_{i \in [n]}\bigg|X_j^{(i)} - \sum_{k \in {\rm pa}^\gamma(j)} b_{jk}^\gamma X_{k}^{(i)}\bigg|\Bigg).
\end{align}
Marginalizing over the parameters, we obtain the marginal likelihood for DAG $\gamma$,
\begin{align}\label{eq:marg_m}
m\lt(D_n | \gamma\rt)
= \int \cL\lt(D_n\rt | b^\gamma, \theta^\gamma, \gamma) \prod_{j \in [p]} \pi_\theta^\gamma(\theta_j^\gamma) \, \pi_{b, j}^\gamma(b_j^\gamma) \, d\theta_j^\gamma \, db_{j}^\gamma \, .
\end{align}
The marginal likelihood or \emph{evidence} is a crucial quantity for Bayesian model selection. Specifically, given a generic DAG prior $\gamma \; \sim \; \pi_g(\cdot)$, the posterior probability of $\gamma$ given data $D_n$ is proportional to the product of the marginal likelihood and the DAG prior,
\begin{align}\label{eq:post_prob}
\pi(\gamma | D_n) \;\; \propto \;\; m\lt(D_n | \gamma\rt) \times \pi_g(\gamma). 
\end{align} 
The Bayes factor and the posterior odds in favor of $\gamma$ over any $\gamma' \in \Gamma^p$ are denoted by $\sfBF_n(\gamma, \gamma')$ and $\Pi_n(\gamma, \gamma')$, respectively, i.e., 
\begin{align}\label{eq:def_BF}
\sfBF_n(\gamma, \gamma') := \frac{m\lt(D_n | \gamma \; \rt)}{m\lt(D_n | \gamma'\rt)}, \qquad \Pi_n(\gamma, \gamma') := 
\frac{\pi(\gamma \; |  D_n)}{\pi(\gamma' | D_n)} = \sfBF_n(\gamma, \gamma') \times \frac{\pi_g(\gamma \; )}{\pi_g(\gamma')}.
\end{align}
A natural choice of $\pi_g(\cdot)$ is the uniform prior, i.e., $\pi_g(\cdot) \; \propto \; 1$; however, somewhat surprisingly,  a non-trivial DAG prior is required to ensure the desired model selection consistency \eqref{eq:goal_pmsc} under certain scenarios, as we show later in Section \ref{sec:asym_prop}.
\begin{remark}[Model misspecification]\label{rem:mis_specif}
It is important to note that, in general, the working model $\bM_\gamma$ in \eqref{eq:lap_model_alt} is misspecified, even when $\gamma = \gamma^*$.
To see this, recall the setup in Example \ref{ex:4node}, and consider the SEM of $\bM_{\gamma^*}$,
\begin{align*}
X_3 &= e_3^*,\\
X_2 &= b_{23}^* X_3 + e_2^*,\\
X_1 &= b_{13}^* X_3 + e_1^*,\\
X_4 &= e_4^*,
\end{align*} 
where $(e_j^*/\theta_j^*) \overset{\rm iid}{\sim} \;\; \text{Laplace} \, (0, 1)$ for $j \in [4]$.
Due to the misspecification in error distributions, we have, for almost every $x \in \bR^4$, 
\begin{align*}
\sfp_X^*(x) \neq f^*(x |b^*, \theta^*, \gamma^*) \qquad \text{for every} \;\; b^*, \theta^*.
\end{align*}
The same holds for any $\gamma$.  
We emphasize that this misspecification is intentional, that is, we only treat the Laplace-error model as a {\it working} model and {\it do not} assume or require the {\it true errors} to be Laplace -- all we assume is that they lie in the family of the scale mixture of Gaussian  \eqref{eq:error_dist}. The misspecification necessitates careful considerations in our theoretical study, but also brings important advantages in terms of identifiability and asymptotics, exploiting specific properties of the Laplace distribution. We remark here that the Gaussian family, which is perhaps the most common choice for an error distribution, leads to identifiability issues in the present setting. On the other hand, more expressive semi-parametric models carry their own challenges. These points are further elucidated in Remarks  \ref{rem:why_not_Gauss} and \ref{rem:why_lap} to provide insights behind our choice of Laplace errors.
\end{remark}

\begin{remark}[Prior choice \& intractability of the marginal likelihood]\label{rem:prior}
For the SEM coefficients $b_j^\gamma, j \in [p]$,
we consider typical choices 
such as the independent Gaussian (ridge) priors and Zellner's g-prior, that is,
\begin{align*}
\pi_{b, j}^\gamma(\cdot)  \;\; \equiv \;\; \text{N}(\boldsymbol{0}, \Sigma_j), \quad \text{with} \quad \Sigma_j = \tau_j^2 I_{|\pa^\gamma(j)|} \quad \text{or} \quad \Sigma_j=g \, (D_{n, j}^{\gamma \, T}D_{n, j}^\gamma)^{-1}, \;\; g > 0,
\end{align*} 
where $D_{n, j}^\gamma \in \bR^{n \times |\pa^\gamma(j)|}$ denotes the data matrix consisting of the observations of random variables $X_k, k \in \pa^\gamma(j)$. Alternatively, one may use non-local priors \citep{johnson2012bayesian,altomare2013objective}. For the scale parameters, we similarly consider standard choices for $\pi_\theta^\gamma(\cdot)$, for instance, the inverse-Gamma priors,
i.e., for some $\alpha, \beta > 0$,
\begin{align*} 
\pi_\theta^\gamma(\theta_j^\gamma) \propto (\theta_j^\gamma)^{-\alpha - 1}\exp(-{\beta}/{\theta_j^\gamma}), \quad \theta_j^\gamma \in (0, \infty).  
\end{align*}
However, the marginal likelihood in \eqref{eq:marg_m} is not analytically tractable for any of these prior choices due to the Laplace likelihood, which constitutes a major challenge in establishing posterior DAG selection consistency. We circumvent this issue by developing a Laplace approximation to the marginal likelihood in Theorem \ref{thm:lap_app}.
\end{remark}

\subsection{A brief sketch on DAG selection consistency under model misspecification}\label{sec:sketch}

Since our working model is generally misspecified even when $\gamma = \gamma^*$, the posterior distribution of $(b^\gamma, \theta^\gamma)$ asymptotically targets the \textit{pseudo-true} parameters \cite{white1982maximum, kleijn2012bernstein}, given by 
\begin{align}\label{eq:target}
(\tilde{b}^\gamma, \tilde{\theta}^\gamma) := \argmin_{(b^\gamma, \theta^\gamma)} \; H^\gamma(b^\gamma, \theta^\gamma),
\end{align}
where $H^\gamma(b^\gamma, \theta^\gamma)$ is the negative expected log density under the working model \eqref{eq:lap_model_alt}, i.e.,
\begin{align}\label{def:H_fun}
H^\gamma(b^\gamma, \theta^\gamma) := - \; \sfE_* \lt[\log \; f^\gamma(X | b^\gamma, \theta^\gamma, \gamma)\rt],
\end{align} 
with the expectation $\sfE_*[\cdot]$ taken over $X$ under the data-generating true distribution $\sfP_X^*$. 
In particular, when $\gamma = \gamma^*$, we denote the above function by $H^*(b^*, \theta^*)$ and its minimizer by $(\tilde{b}^*, \tilde{\theta}^*)$. We state some important properties of $H^\gamma(\cdot)$ in the following Lemma. 
\begin{lemma}\label{lem:H_gamma_basic}
Assume $\sfE_*[|\lambda_j|] < \infty$ for every $j \in [p]$. Then, $H^\gamma(b^\gamma, \theta^\gamma)$ is finite for each $\gamma \in \Gamma_p$, and $(b^\gamma, \theta^\gamma) \in \mathbb{R}^{|\gamma|} \times (0, \infty)^p$. Moreover, the minimization problem in \eqref{eq:target} possesses a unique solution given by 
\begin{align*}
\tilde{b}_j^\gamma = \argmin_{b_j^\gamma} \; \sfE_*\bigg[\bigg|X_j - \sum_{k \in {\rm pa}(j)} b_{jk}^\gamma X_{k}\bigg|\bigg], \quad 
\tilde{\theta}_j^\gamma = \bigg(\sfE_*\bigg[\bigg|X_j - \sum_{k \in {\rm pa}(j)} \tilde{b}_{jk}^\gamma X_{k}\bigg|\bigg]\bigg)^{-1}, \;\;\; j \in [p]. 
\end{align*}
In particular, when $\gamma = \gamma^*$, 
\begin{align*}
\tilde{b}^*_{jk} = \beta^*_{jk} \qquad \text{for every} \;\; k \in \pa^*(j).
\end{align*} 
Let $h_\gamma$ denote the minimized value of $H^\gamma(\cdot)$. Then, 
\begin{align}\label{eq:h_gamma_id}
h_\gamma :\, = \min_{(b^\gamma, \theta^\gamma)} H^\gamma(b^\gamma, \theta^\gamma) \, = \, p (1 + \log 2) 
- \sum_{j \in [p]} \log \tilde{\theta}_j^\gamma.
\end{align}
\end{lemma}
\begin{proof}
The proof can be found in Appendix \ref{app:identif}, see Lemma \ref{lem:minH} and Lemma \ref{lem:minH*}.
\end{proof}
In the rest of the paper, we assume the condition $\sfE_*[|\lambda_j|] < \infty$ for every $j \in [p]$, which guarantees finiteness of the function $H^\gamma(\cdot)$, for every $\gamma \in \Gamma^p$.
The uniqueness of the minimizer exploits log-concavity of $f^\gamma$ under a reparameterization, see Lemma \ref{lem:log-concave}. An important upshot of Lemma \ref{lem:H_gamma_basic} is that when $\gamma = \gamma^*$, the pseudo-true parameters $\tilde{b}_{jk}^*$ target the corresponding true SEM coefficients $\beta_{jk}^*$, even though the error model is misspecified.

We call $h_\gamma$ the population risk (or simply {\it risk}) associated with $\gamma$, and in particular, when $\gamma = \gamma^*$, we denote it by $h_* \equiv h_{\gamma^*}$. To connect the marginal likelihood $m\lt(D_n | \gamma\rt)$ with the risk $h_\gamma$, we develop a version of the Laplace approximation \cite{tierney1986accurate, tierney1989fully} under model misspecification in Theorem \ref{thm:lap_app} to obtain
\begin{align}\label{eq:lap_missp}
\log m\lt(D_n | \gamma\rt) = -n h_\gamma (1 + O_p(n^{-1/2})) - \frac{p + |\gamma|}{2}\log n + O_p(1).
\end{align}
The derivation of the above approximation is non-trivial due to non-differentiability of the likelihood function \eqref{eq:likeli_L}; refer to the discussion around Theorem \ref{thm:lap_app}. 
Following \eqref{eq:def_BF} and using \eqref{eq:lap_missp}, we then have
\begin{align} \label{eq:approx_BF}
\log \sfBF_n(\gamma^*, \gamma) 
& = n \lt(h_\gamma - h_{*}\rt) - \frac{|\gamma^*| - |\gamma|}{2}\log n  + R_n, 
\end{align}
where $R_n$ is a remainder term which is at most $O_p(\sqrt{n})$. The leading contribution to the log-Bayes factor between $\gamma^*$ and $\gamma$ therefore comes from the risk difference $(h_\gamma - h_{*})$, and thus, we undertake a careful study of the properties of $h_\gamma$ in the next section. 
Next, in order to establish the posterior DAG selection consistency, first we 
show that for every $\gamma \in \Gamma^p$,
\begin{align}\label{ineq:H>H*}
(h_\gamma - h_*) \geq 0,
\end{align}
and furthermore, there exists a family of DAGs, say $\cE^* \subseteq \Gamma^p$ such that
\begin{align}\label{char:E*}
\begin{split}
\text{both} \quad 
&h_\gamma = h_*,
\quad \text{and} \quad |\gamma| = |\gamma^*|, \qquad \text{if} \;\; \gamma \in \cE^*, \;\; \text{and}\\
\text{either} \quad 
&h_\gamma > h_*,
\quad \text{or} \;\;\; \quad |\gamma| > |\gamma^*|, \qquad \text{otherwise.}
\end{split}
\end{align}
Then, we show that the remainder term $R_n$ in \eqref{eq:approx_BF} is $O_p(1)$ if $\gamma \supseteq \gamma^*$, see Lemma \ref{lem:wilks_missp}, and $O_p(\sqrt{n})$ otherwise.
Subsequently, we propose appropriate DAG priors 
to ensure that the posterior odds $\Pi_n(\gamma^*, \gamma)$ diverges to $\infty$ if $\gamma \not\in \cE^*$. 
Moreover, in Theorem \ref{thm:E=DE}, we will not only show the existence of $\cE^*$ but also establish that it coincides with the distribution equivalence class $\cE(\gamma^*, n\cG^*)$, which facilitates deriving the desired posterior consistency results.
To do so, we exploit the key advantage that the pseudo-true parameters are theoretically tractable under the Laplace-error model as shown in Lemma \ref{lem:H_gamma_basic}.

\begin{remark}[Gaussianity, tractability, and non-identifiability]\label{rem:why_not_Gauss}
To obtain an analytically tractable marginal likelihood, it is appealing to model the errors $e_j, j \in [p]$ by Gaussian distributions. However, it does not lead to identifiability of  $\gamma^*$, 
as we demonstrate in the following. Indeed, if we consider some $\gamma$ that is Markov equivalent to $\gamma^*$, 
and in \eqref{eq:lap_model_alt}, let 
\begin{align*}
(e_j^\gamma/\theta_j^\gamma) \overset{\rm iid}{\sim} \text{N} (0, 1),
\end{align*} 
then because under Gaussianity, Markov equivalence implies distribution equivalence \cite{geiger2002parameter}, for every $b^*, \theta^*$, there exist some $b^\gamma, \theta^\gamma$ such that
\begin{align*}
f^\gamma(x |b^\gamma, \theta^\gamma, \gamma) = f^*(x |b^*, \theta^*, \gamma^*) \qquad \text{for every} \;\; x \in \bR^p.
\end{align*}
In other words, $\bM_\gamma$ is  equivalent to $\bM_{\gamma^*}$, 
 resulting in non-identifiability between $\gamma^*$ and $\gamma$. 
However, if $\epsilon_j, j \in [p]$ are all non-Gaussian, as in LiNGAM \cite{shimizu2006linear}, then $\gamma^*$ must be uniquely identifiable. Therefore, for identifiability beyond Markov equivalence classes, it is necessary to use some non-Gaussian error distributions at the cost of losing tractability of the marginal likelihood.
\end{remark}

\begin{remark}[Laplace vs other parametric and semiparametric error distributions]\label{rem:why_lap}
Modeling the errors with the Laplace distribution with unknown scale parameters offers several advantages over other parametric families of non-Gaussian distributions. 
As Lemma \ref{lem:H_gamma_basic} shows, under the Laplace-error model, the functions $H^\gamma(\cdot)$ in \eqref{def:H_fun} assume tractable forms for all $\gamma$, and moreover, exploiting log-concavity, they admit unique
population targets $(\tilde{b}^\gamma, \tilde{\theta}^\gamma)$ having analytically tractable expressions. 
These expressions are convenient for deriving subsequent identifiability theory, since they allow us to smoothly connect between the probabilistic and graph-theoretic properties as established in Theorem \ref{thm:H<H}; see Section \ref{subsec:pf_sketch_iden} for a brief proof sketch regarding this point. Such analytical simplicity is not immediately obvious if we consider, for example, Cauchy, t, or many other parametric non-Gaussian error distribution families. Furthermore, in spite of non-smoothness of the log-likelihoods and intractability of the marginal likelihoods, it is possible to establish Laplace approximations; see Theorem \ref{thm:lap_app}. 

Outside parametric families, a potentially attractive choice is to employ semiparametric mixture distributions with large support on the space of symmetric unimodal distributions. For example, one may consider scale mixture of Gaussians $\int \eta^{-1} \phi\lt(x/\eta\rt) d\sfP(\eta)$ or mixtures of uniforms $\int (2 \theta)^{-1} \mathbf{1}_{[-\theta, \theta]}(x) d \sfP(\theta)$ with the mixing distribution $\sfP$ assigned a Dirichlet process (DP) prior or its many variants. While such flexible error distributions appear routinely in nonparametric Bayesian modeling \cite[Chapter 5]{ghosal2017fundamentals}, and more sporadically in the structure learning context \cite{hoyer2009bayesian, shimizu2014bayesian}, their rigorous performance characterization in model selection contexts is comparatively limited \cite{kundu2014bayes}. In addition to analytic intractability of the marginal likelihood due to the presence of infinite-dimensional nuisance parameters associated with the mixing distribution $\sfP$, the validity of the Laplace approximation becomes less immediate. Moreover, the function $H^\gamma(\cdot)$ loses its tractability and as a consequence, it becomes more complicated to understand the target of estimation in \eqref{def:H_fun}.

\end{remark}

\section{Identifiability}\label{sec:identif}
In this section, we develop our theory of identifying the underlying DAG $\gamma^*$ up to the distribution equivalence class $\cE(\gamma^*, n\cG^*)$, which is a prerequisite for our posterior DAG selection consistency theory. In light of \eqref{char:E*}, we introduce the following class 
\begin{align}\label{def:E*}
\cE^* := \Big\{\gamma \in \Gamma^p : \;
h_\gamma = h_*
\quad \text{and} \quad |\gamma| = |\gamma^*|\Big\},
\end{align}
which consists of DAGs $\gamma$ with $|\gamma| = |\gamma^*|$ that achieve the same population risk as $\gamma^*$. Intuitively, this implies the posterior $\pi(\cdot \mid D_n)$ should concentrate on the set $\cE^*$, since the number of model parameters under $M_\gamma$ and $M_{\gamma^*}$ are the same for any $\gamma \in \cE^*$.  
Interestingly, we show below that $\cE^*$ coincides with the distribution equivalence class $\cE(\gamma^*, n\cG^*)$. Observe that $\cE(\gamma^*, n\cG^*)$ is purely a property of the true underlying data generation process, whereas $\cE^*$ arises via interaction between the working model and the true distribution through \eqref{def:H_fun}. Therefore, such a result signifies the ability of our postulated working model to accurately recover the distribution equivalence class. 
Furthermore, we characterize this class by deriving necessary and sufficient conditions for a DAG to belong to $\cE^*$.

\subsection{Risk, Markov, and distribution equivalence}
To begin with, we consider the \textit{risk function} $\gamma \mapsto h_\gamma$ 
defined in \eqref{eq:h_gamma_id}.  
We introduce below the notion of \textit{risk equivalence} and the \textit{risk equivalence class} of DAGs.

\begin{definition}[Risk equivalence class]
Two DAGs $\gamma_1, \gamma_2 \in \Gamma^p$ are said to be \textit{risk equivalent} if
$h_{\gamma_1} = h_{\gamma_2}$.
 For any $\gamma \in \Gamma^p$, its \textit{risk equivalence class} is defined as
\begin{align*}
\{\gamma' \in \Gamma^p : \; h_\gamma = h_{\gamma'} \,
\}.
\end{align*}
\end{definition}

Next, in order to establish the identifiability of $\gamma^*$ at least up to a certain class, our primary step is to establish \eqref{ineq:H>H*}, that is, the risk function $h_\gamma$ 
is indeed minimized at $\gamma^*$ and further characterize the set of its minimizers, which is the risk equivalence class of $\gamma^*$.
In this regard, we first define $\bar{\cE}^* \supseteq \cE^*$ as the \emph{risk equivalence class} of $\gamma^*$, i.e.,
\begin{align}\label{def:barE*}
\bar{\cE}^* := \Big\{\gamma \in \Gamma^p : 
h_\gamma = h_*
\Big\}.
\end{align}

Clearly, $\gamma^* \in \bar{\cE}^*$. As shown in Lemma \ref{lem:supset_h}, more generally, $\gamma \in \bar{\cE}^*$ if $\gamma$ is a supergraph of $\gamma^*$, i.e.,
\begin{align*}
    \cS^* := \lt\{\gamma \in \Gamma^p : \gamma\supseteq \gamma^* \rt\} \subseteq \bar{\cE}^*.
\end{align*} 
Moreover, there may be more elements in $\bar{\cE}^*$, i.e., $\bar{\cE}^* \setminus \cS^* \neq \emptyset$, and when and \emph{only when} it is the case, we consider the additional assumption of \textit{faithfulness} \cite{spirtes2001causation}, formally defined below.

\begin{definition}[Faithfulness \cite{spirtes2001causation}]\label{def:faith}
Let $\bI(\sfP_X^*)$ denote the set of conditional independence relationships under $\sfP_X^*$. Then $\sfP_X^*$ is called faithful to $\gamma^*$ if \ $\bI(\sfP_X^*) \subseteq \bI(\gamma^*)$.
\end{definition}
Therefore, 
we assume that, 
\begin{align}\label{assum:faith}
\text{in the case of} \;\; \bar{\cE}^* \setminus \cS^* \neq \emptyset, \quad \sfP_X^* \;\; \text{is faithful to} \;\; \gamma^*.
\end{align}
We emphasize that the assumption of faithfulness is not needed when $\bar{\cE}^* = \cS^*$, as we validate this shortly in Corollary \ref{cor:E=DE=gam}.
In the next theorem, we show that, under the above assumption, the risk function $h_\gamma$ 
is minimized over $\bar{\cE}^*$ and subsequently characterize $\bar{\cE}^*$ in Corollary \ref{cor:barE*} under the assumption \eqref{assum:faith}.  

\begin{theorem}[Minimized risk]\label{thm:H<H}
For every $\gamma' \in \Gamma^p$, we have
\begin{align*}
h_* \leq h_{\gamma'},
\end{align*}
where the equality holds if and only if \, $\gamma' \supseteq \gamma$ \ for which \ $\sfP_X^* \in \cP(\gamma, n\cG)$ \ for some $n\cG \subseteq [p]$.

Under the assumption \eqref{assum:faith}, the last part of the condition is equivalent to $\sfP_X^* \in \cP(\gamma, n\cG^*)$, which in turn holds, if and only if 
$\gamma$ satisfies the following conditions:
\begin{enumerate}
\item[(1)] 
for every $j \in  n\cG^*$, $\pa^\gamma(j) = \pa^*(j)$, and
\item[(2)] for every $j \notin  n\cG^*$, 
$\pa^\gamma(j)$ is such that there exists non-zero \ $\beta_{jk}^\gamma, k \in \pa^\gamma(j)$ \ for which
\begin{align}\label{eq:Gaus_barE}
\eta_j^\gamma := \Big(X_j - \sum_{k \in \pa^\gamma(j)} \beta_{jk}^\gamma X_k\Big) 
\end{align}
is some linear combination of the Gaussian errors \ $\epsilon_j, j \notin n\cG^*$, and $\eta_j^\gamma, \, j \notin n\cG^*$ are pairwise independent.
\end{enumerate}
\end{theorem}

\begin{proof}
The proof can be found in Appendix \ref{pf:thm:H<H}.
\end{proof}

The above result implies that 
the risk is minimized by some $\gamma' \in \Gamma^p$ if $\sfP_X^*$ can be \textit{represented} by $(\gamma, n\cG^*)$ where $\gamma'\supseteq\gamma$, that is, under $\sfP_X^*$, the variables can be alternatively generated by an SEM 
under $\gamma$ 
whose nodes with non-Gaussian errors are indicated by $n\cG^*$. In fact, it is not difficult to observe from the conditions in Theorem \ref{thm:H<H} (see also Lemma \ref{lem:R=C=emp} and Lemma \ref{lem:R,C=non_emp}) that the structural equations corresponding to the nodes in $n\cG^*$ must be identical to those in \eqref{eq:model}, and for the rest, they follow from \eqref{eq:Gaus_barE}, that is,
\begin{align*}
X_j &= \sum_{k \in {\rm pa}^*(j)} \beta^*_{jk} X_{k} + \epsilon_{j}, \qquad\ \text{for every} \;\; j \in n\cG^*,\\
X_j &= \sum_{k \in {\rm pa}^\gamma(j)} \beta^\gamma_{jk} X_{k} + \eta_{j}^\gamma, \qquad \text{for every} \;\; j \notin n\cG^*.
\end{align*}
Thus, in light of Theorem \ref{thm:H<H}, we define the \emph{minimal risk equivalence class} of $\gamma^*$,
\begin{align}\label{def:epsR}
    \bar{\cE}_R^* 
:=  \{\gamma \in \Gamma^p : \gamma \;\; \text{satisfies conditions} \; (1) \; \text{and} \; (2) \; \text{in Theorem \ref{thm:H<H}}\} \, \subseteq \, \bar{\cE}^*.
\end{align}
It is minimal in the sense that any risk equivalent DAG must be a supergraph of some DAG in this class.
We refer to Figure \ref{fig:illst_DE} for a pictorial representation of the aforementioned classes of DAGs. Furthemore, following \eqref{def:E*}, \eqref{def:barE*} and \eqref{def:epsR}, it is clear that
$\cE^* = \big\{\gamma \in \bar{\cE}^* : \; |\gamma| = |\gamma^*|\big\}$, and $\gamma^* \in \bar{\cE}^*_R \cap \cE^*$. More interestingly, in the following result we show that when, in particular, $\bar{\cE}^* = \cS^*$, both $\bar{\cE}^*_R$ and $\cE^*$ along with $\cE(\gamma^*, n\cG^*)$ reduce to $\{\gamma^*\}$, resulting in the \textit{unique identification} of $\gamma^*$, without any additional assumption on $\sfP_X^*$, as indicated earlier.

\begin{corollary}[Unique identifiability]\label{cor:E=DE=gam}
If \; $\bar{\cE}^* = \cS^*$, \; then \; $\bar{\cE}^*_R =  \cE^* = \cE(\gamma^*, n\cG^*) =  \{\gamma^*\}$.
\end{corollary}

\begin{proof}
    The proof can be found in Appendix \ref{pf:cor:E=DE=gam}.
\end{proof}

\begin{corollary}[Characterization of risk equivalence class]\label{cor:barE*}
Under the assumption \eqref{assum:faith}, we have
\begin{align*}
\bar{\cE}^* \; = \; \lt\{\gamma' \in \Gamma^p : \, \gamma' \supseteq \gamma \quad \text{for which} \quad \sfP_X^* \in \cP(\gamma, n\cG^*)\rt\} \; = \;  \mathop{\scalebox{1.5}{$\cup$}}_{\gamma \in \bar{\cE}_R^*} \; \cS^\gamma.  
\end{align*}
Moreover, 
$\bar{\cE}^* = \cS^*$ if and only if  $\bar{\cE}^*_R = \{\gamma^*\}$.
\end{corollary}

\begin{proof}
The proof immediately follows from the definition of $\bar{\cE}^*$ in \eqref{def:barE*}, the conditions for equality stated in Theorem \ref{thm:H<H}, and Corollary \ref{cor:E=DE=gam}.
\end{proof}
  
It is important to particularly identify under what conditions both $\cE^*$ and $\cE(\gamma^*, n\cG^*)$ reduce to the smallest possible size, i.e., $\cE^* = \cE(\gamma^*, n\cG^*) = \{\gamma^*\}$, thereby ensuring the unique identifiability of $\gamma^*$. We show three such conditions each
leading to $\bar{\cE}^* = \cS^*$, which, by Corollary \ref{cor:E=DE=gam}, in turn implies that both $\cE^*$ and $\cE(\gamma^*, n\cG^*)$ contain only $\gamma^*$. 

\begin{proposition}[Sufficient conditions for unique identifiability]\label{prop:spl_cases}
We have $\bar{\cE}^* = \cS^*$ if any of the following conditions holds:
\begin{enumerate}
\item[(a)] there is at most one Gaussian error, i.e., $|n\cG^*| \geq (p-1)$,
\item[(b)] all error variances are equal, or
\item[(c)] 
the assumption 
\eqref{assum:faith} holds, and $\bar{\cE}^*_R = \{\gamma^*\}$, e.g., when variances of all Gaussian errors are equal.
\end{enumerate}   
\end{proposition}
\begin{proof}
The proof can be found in Appendix \ref{pf:prop:spl_cases}.
\end{proof}
 
\begin{remark}[LiNGAM]
LiNGAM \cite{shimizu2006linear} assumed that all errors are non-Gaussian, i.e., $|n\cG^*| = p$, which can also be slightly relaxed to condition (a) in Proposition \ref{prop:spl_cases} by following the identifiability properties of ICA. In this work, we also achieve this identifiability result, although with an alternative proof technique that is crucial in our context; see Appendices \ref{app:working} and \ref{app:identif} for further detail. 
\end{remark}

\begin{remark}[Equal error variance]
It has been shown in numerous works \cite{peters2014identifiability, chen2019causal, loh2014high, rothenhausler2018causal} that under the assumption of all error variances being equal, unique identification is possible, which is also formalized in condition (b) in Proposition \ref{prop:spl_cases}. Moreover, in condition (c), we show that this can be partially relaxed in the present context by requiring the equality of variances only for the nodes with Gaussian errors, under the assumption \eqref{assum:faith}.
\end{remark}

\begin{remark}[Sufficiency and non-necessity]\label{rem:barE=gam*}
Although sufficient, neither of the restrictions regarding the number of non-Gaussian errors or error variances stated in Proposition \ref{prop:spl_cases} is necessary for having $\bar{\cE}^*_R =  \{\gamma^*\}$, as demonstrated in the following example.

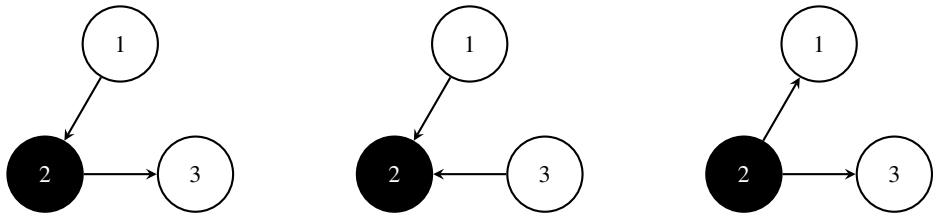
\begin{figure}[H]
\centering
   \begin{minipage}{0.32\textwidth}
        \centering
        \begin{tikzpicture}[->, >=stealth, thick]
    \node (1) [circle, draw, minimum size=1cm] at (1,1.732) {1}; 
    \node (2) [circle, draw, fill=black, text=white, minimum size=1cm] at (0,0) {2}; 
    \node (3) [circle, draw, minimum size=1cm] at (2,0) {3}; 
    \draw (1) -- (2);
    \draw (2) -- (3);
\end{tikzpicture}        
    \end{minipage}
    \begin{minipage}{0.32\textwidth}
        \centering
        \begin{tikzpicture}[->, >=stealth, thick]
    \node (1) [circle, draw, minimum size=1cm] at (1,1.732) {1}; 
    \node (2) [circle, draw, fill=black, text=white, minimum size=1cm] at (0,0) {2}; 
    \node (3) [circle, draw, minimum size=1cm] at (2,0) {3}; 
    \draw (1) -- (2);
    \draw (3) -- (2);
\end{tikzpicture}        
    \end{minipage}
    \begin{minipage}{0.32\textwidth}
        \centering
        \begin{tikzpicture}[->, >=stealth, thick]
    \node (1) [circle, draw, minimum size=1cm] at (1,1.732) {1}; 
    \node (2) [circle, draw, fill=black, text=white, minimum size=1cm] at (0,0) {2}; 
    \node (3) [circle, draw, minimum size=1cm] at (2,0) {3}; 
    \draw (2) -- (1);
    \draw (2) -- (3);
\end{tikzpicture}        
    \end{minipage}
    \caption{Examples to illustrate non-necessity of the conditions in Proposition \ref{prop:spl_cases}.}
    \label{fig:eps_2_nonG}
    \end{figure}
\noindent    Consider $\gamma^*$ to be any of the three DAGs in Figure \ref{fig:eps_2_nonG}, with $\epsilon_2$ being the only non-Gaussian error (i.e., $n\cG^* = \{2\}$) and no restriction on the error variances. Clearly, neither condition (a) or (b) in Proposition \ref{prop:spl_cases} holds. But since no other DAG satisfies the conditions in Theorem \ref{thm:H<H}, we have $\bar{\cE}^*_R = \{\gamma^*\}$.
\end{remark}

\begin{remark}[Faithfulness]\label{rem:no_faith}
It has been shown that faithfulness of $\sfP_X^*$ is not required for unique identification of $\gamma^*$ under the scenarios of all errors being non-Gaussian, as in LiNGAM \cite{shimizu2006linear}, or all error variances being equal, as in \cite{peters2014identifiability, chen2019causal, loh2014high, rothenhausler2018causal}. Indeed, these scenarios are specifically included as conditions (a) and (b) in Proposition \ref{prop:spl_cases}, where the assumption of faithfulness is not needed.
In condition (c) we further demonstrate that under the assumption \eqref{assum:faith}, the unique identification is feasible 
even under a more general case, when $\bar{\cE}^*_R =  \{\gamma^*\}$, i.e., there is no other DAG that can represent $\sfP_X^*$. 
Indeed, this not only encompasses the aforementioned scenarios but also includes other interesting cases such as the example illustrated in Remark \ref{rem:barE=gam*}.
\end{remark}

Generally, $\bar{\cE}^*_R$ may contain DAGs other than $\gamma^*$ and may or may not coincide with $\cE^*$, as shown in the following two concrete examples. 
 \begin{example}\label{ex:barE-E} 
Consider $\gamma^*$ to be the DAG in Figure \ref{fig:risk_equiv}(a) with the following SEM:
\begin{align*}
X_1 &= \epsilon_1,\\
X_2 &= \beta^*_{21}X_1 + \epsilon_2,\\
X_3 &= \beta^*_{32}X_2 + \epsilon_3,
\end{align*}
where $\epsilon_1$ is the only non-Gaussian error, i.e., $n\cG^* = \{1\}$, and $\epsilon_2, \epsilon_3 \overset{\text{iid}}{\sim} \text{N}(0, 1)$. 

\begin{figure}[H]
  \centering
    \begin{minipage}{0.32\textwidth}
        \centering
        \begin{tikzpicture}[->, >=stealth, thick]
            \node (1) [circle, draw, fill=black, text=white, minimum size=1cm] at (1,1.732) {1};
            \node (2) [circle, draw, minimum size=1cm] at (0,0) {2};
            \node (3) [circle, draw, minimum size=1cm] at (2,0) {3};
            \draw (1) -- (2);
            \draw (2) -- (3);
        \end{tikzpicture}
        \par (a) 
    \end{minipage}
    \begin{minipage}{0.32\textwidth}
        \centering
        \begin{tikzpicture}[->, >=stealth, thick]
            \node (1) [circle, draw, fill=black, text=white, minimum size=1cm] at (1,1.732) {1};
            \node (2) [circle, draw, minimum size=1cm] at (0,0) {2};
            \node (3) [circle, draw, minimum size=1cm] at (2,0) {3};
            \draw (1) -- (3);
            \draw (1) -- (2);
            \draw (3) -- (2);
        \end{tikzpicture}
        \par (b)
    \end{minipage}
    \begin{minipage}{0.32\textwidth}
        \centering
        \begin{tikzpicture}[->, >=stealth, thick]
            \node (1) [circle, draw, fill=black, text=white, minimum size=1cm] at (1,1.732) {1};
            \node (2) [circle, draw, minimum size=1cm] at (0,0) {2};
            \node (3) [circle, draw, minimum size=1cm] at (2,0) {3};
            \draw (1) -- (3);
            \draw (1) -- (2);
            \draw (2) -- (3);
        \end{tikzpicture}
        \par (c)
    \end{minipage}
    \caption{The DAGs in {\rm (a)}, {\rm (b)} and {\rm (c)} are $\gamma^*$, $\gamma$ and $\gamma'$, respectively, in Example \ref{ex:barE-E}.}
    \label{fig:risk_equiv}
\end{figure}
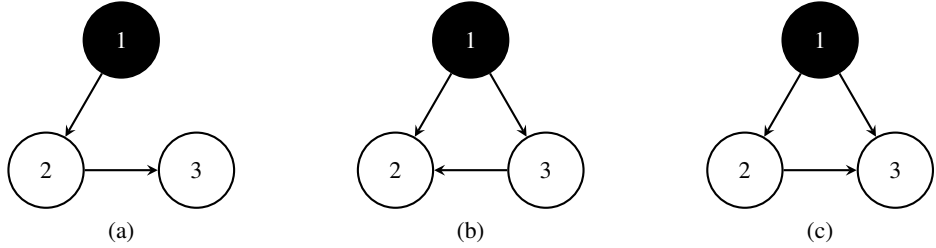

Let $\gamma$ be the DAG in Figure \ref{fig:risk_equiv}(b). Then, the variables can be alternatively generated by the following SEM based on $\gamma$, also with only node $1$ having a non-Gaussian error: 
\begin{align*}
X_1 &= \epsilon_1,\\
X_3 &= \beta^\gamma_{31}X_1 + \eta^\gamma_3,\\
X_2 &= \beta^\gamma_{23}X_3 + \beta^\gamma_{21}X_1 + \eta^\gamma_2,
\end{align*}
where the SEM coefficients are
\begin{align*}
\beta^\gamma_{31} = \beta^*_{32}\beta^*_{21}, \qquad \;\; \beta^\gamma_{23} = \frac{\beta^*_{32}}{1+\beta^{*2}_{32}} \qquad \text{and} \qquad \beta^\gamma_{21} = \frac{\beta^*_{21}}{1 + \beta^{*2}_{32}},
\end{align*}
and the error variables $\eta_2^\gamma$ and $\eta_3^\gamma$ are
\begin{align*}
\eta_2^\gamma = \frac{1}{1+\beta^{*2}_{32}}\epsilon_2 - \frac{\beta^*_{32}}{1+\beta^{*2}_{32}}\epsilon_3 \qquad \text{and} \qquad \eta_3^\gamma = \beta^*_{32}\epsilon_2 + \epsilon_3, 
\end{align*}
which are Gaussian and independent. This implies that $\gamma$ satisfies that conditions in Theorem \ref{thm:H<H}, or equivalently, $\sfP_X^* \in \cP(\gamma, n\cG^*)$ and in fact, there is no other DAG that satisfies these conditions. 
Furthermore, if $\gamma'$ denotes the DAG in Figure \ref{fig:risk_equiv}(c), then clearly $\gamma' \supset \gamma^*$, implying $\cS^* = \{\gamma^*, \gamma'\}$, and also, $\cS^\gamma = \{\gamma\}$. 
Therefore, 
following \eqref{def:E*}, \eqref{def:barE*}, \eqref{def:epsR}, and Corollary \ref{cor:barE*},
\begin{align*}
    \bar{\cE}^*_R = \{\gamma^*, \gamma\}, \;\; 
    \text{and under the assumption \eqref{assum:faith}}, \;\;
    \bar{\cE}^* = \{\gamma^*, \gamma, \gamma'\}, \;\; \text{and} \;\;
    \cE^* = \{\gamma^*\}. 
\end{align*}
\end{example}

\begin{example}\label{ex:barE=E} 
Consider $\gamma^*$ to be the DAG in Figure \ref{fig:risk_equiv_E=RE}(a) with the following SEM:
\begin{align*}
X_1 &= \epsilon_1,\\
X_2 &= \beta^*_{21}X_1 + \epsilon_2,\\
X_3 &= \beta^*_{32}X_2 + \beta^*_{31}X_1 + \epsilon_3,
\end{align*}
where $\epsilon_1$ is the only non-Gaussian error, i.e., $n\cG^* = \{1\}$, and $\epsilon_2, \epsilon_3 \overset{\text{iid}}{\sim} \text{N}(0, 1)$. 

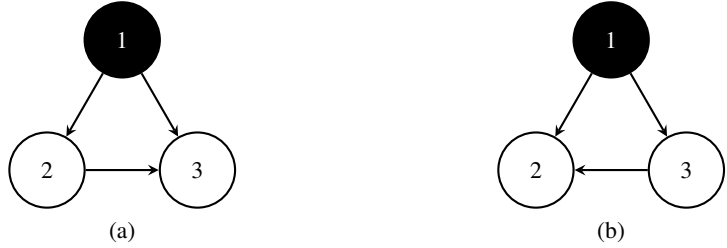
\begin{figure}[H]
  \centering
    \begin{minipage}{0.45\textwidth}
        \centering
        \begin{tikzpicture}[->, >=stealth, thick]
            \node (1) [circle, draw, fill=black, text=white, minimum size=1cm] at (1,1.732) {1};
            \node (2) [circle, draw, minimum size=1cm] at (0,0) {2};
            \node (3) [circle, draw, minimum size=1cm] at (2,0) {3};
            \draw (1) -- (2);
            \draw (2) -- (3);
            \draw (1) -- (3);
        \end{tikzpicture}
        \par (a) 
    \end{minipage}
    \begin{minipage}{0.45\textwidth}
        \centering
        \begin{tikzpicture}[->, >=stealth, thick]
            \node (1) [circle, draw, fill=black, text=white, minimum size=1cm] at (1,1.732) {1};
            \node (2) [circle, draw, minimum size=1cm] at (0,0) {2};
            \node (3) [circle, draw, minimum size=1cm] at (2,0) {3};
            \draw (1) -- (3);
            \draw (1) -- (2);
            \draw (3) -- (2);
        \end{tikzpicture}
        \par (b)
    \end{minipage}
    \caption{The DAGs in {\rm (a)} and {\rm (b)} are $\gamma^*$ and $\gamma$, respectively, in Example \ref{ex:barE=E}.}
    \label{fig:risk_equiv_E=RE}
\end{figure}

Let $\gamma$ be the DAG in Figure \ref{fig:risk_equiv_E=RE}(b). Then, the variables can be alternatively generated by the following SEM based on $\gamma$, also with only node $1$ having a non-Gaussian error: 
\begin{align*}
X_1 &= \epsilon_1,\\
X_3 &= \beta^\gamma_{31}X_1 + \eta^\gamma_3,\\
X_2 &= \beta^\gamma_{23}X_3 + \beta^\gamma_{21}X_1 + \eta^\gamma_2,
\end{align*}
where the SEM coefficients are
\begin{align*}
\beta^\gamma_{31} = \beta^*_{32}\beta^*_{21} + \beta^*_{31}, \qquad \;\; \beta^\gamma_{23} = \frac{\beta^*_{32}}{1+\beta^{*2}_{32}} \qquad \text{and} \qquad \beta^\gamma_{21} = \frac{\beta^*_{21} - \beta^*_{32}\beta^*_{31}}{1 + \beta^{*2}_{32}},
\end{align*}
and the error variables $\eta_2^\gamma$ and $\eta_3^\gamma$ are
\begin{align*}
\eta_2^\gamma = \frac{1}{1+\beta^{*2}_{32}}\epsilon_2 - \frac{\beta^*_{32}}{1+\beta^{*2}_{32}}\epsilon_3 \qquad \text{and} \qquad \eta_3^\gamma = \beta^*_{32}\epsilon_2 + \epsilon_3, 
\end{align*}
which are Gaussian and independent. This implies that $\gamma$ satisfies the conditions in Theorem \ref{thm:H<H}, or equivalently, $\sfP_X^* \in \cP(\gamma, n\cG^*)$, and in fact, there is no other DAG that satisfies these conditions. 
Furthermore, it is clear that $\cS^* = \{\gamma^*\}$ and $\cS^\gamma = \{\gamma\}$, and 
therefore, 
following \eqref{def:E*}, \eqref{def:barE*}, \eqref{def:epsR}, and Corollary \ref{cor:barE*},
\begin{align*}
    \bar{\cE}^*_R = \{\gamma^*, \gamma\}, \;\; 
    \text{and under the assumption \eqref{assum:faith}}, \qquad
    \bar{\cE}^* = \cE^* = \{\gamma^*, \gamma\}. 
\end{align*}
\end{example}

Therefore, for the general case $\bar{\cE}^*_R \supseteq \{\gamma^*\}$, it still remains to confirm the equality between $\cE^*$ and $\cE(\gamma^*, n\cG^*)$.
In this regard, in the next theorem, we establish two important results regarding the class $\cE^*$. First, we establish more generally that under the assumption \eqref{assum:faith}, the family $\cE^*$ can be characterized as the set of DAGs that are not only risk equivalent but also Markov equivalent to $\gamma^*$, and second, it coincides with the distribution equivalence class $\cE(\gamma^*, n\cG^*)$ defined in \eqref{eq:dist_equiv}, thereby fulfilling the objective of this section.



\begin{theorem}[Risk, Markov, and distribution equivalence]\label{thm:E=DE}
Suppose that \eqref{assum:faith} holds.
Then we have
\begin{align*}
\cE^* = \Big\{\gamma \in \Gamma^p : \, 
h_\gamma = h_*
\quad \text{and} \quad \bI(\gamma) = \bI(\gamma^*)\Big\} = \cE(\gamma^*, n\cG^*).
\end{align*}
\end{theorem}

\begin{proof}
The proof can be found in Appendix \ref{pf:thm:E=DE}.
\end{proof}
The following corollary of Theorem \ref{thm:E=DE} immediately establishes \eqref{char:E*}, which, as discussed in Section \ref{sec:sketch}, facilitates the development of the consistency theory in Section \ref{sec:asym_prop}. 
For every $\gamma \in \Gamma^p$, we denote by $\delta_\gamma$ and $\psi_\gamma$ the difference in the value of the risk function and the difference in the number of edges, respectively, between $\gamma$ and $\gamma^*$, i.e.,
\begin{align}\label{def:del_psi}
\delta_\gamma := h_\gamma - h_*
\qquad \text{and} \qquad
\psi_\gamma := |\gamma| - |\gamma^*|.
\end{align}
\begin{corollary}\label{cor:ineq_risk_and_param}
Fix any $\gamma \in \Gamma^p$. Then we have
\begin{align*}
    \delta_\gamma \geq 0, \quad \, \text{with equality being achieved if and only if} \;\; \gamma \in \bar{\cE}^*.
\end{align*}
Moreover, if the assumption \eqref{assum:faith} holds, then we have $\cE^* \subseteq \bar{\cE}^*_R$, and in the case of $\gamma \in \bar{\cE}^*$,
\begin{align*}
    \psi_\gamma \geq 0, \quad \text{with equality being achieved if and only if} \;\; \gamma \in \cE^*.
\end{align*}
Thus, 
$\max\{\delta_\gamma, \psi_\gamma\} \geq 0$, where equality holds if and only if $\gamma \in \cE^*$.
\end{corollary}
\begin{proof}
The proof can be found in Appendix \ref{pf:cor:ineq_risk_and_param}.
\end{proof}

Furthermore, we 
also depict the above results in Figure \ref{fig:illst_DE}, under the assumption \eqref{assum:faith}.
\begin{figure}[!htb]
        \centering
\begin{tikzpicture}
  \draw[fill=red!50, draw=purple, thick] (0,0) ellipse (6.3cm and 3.8cm);

  \draw[fill=orange!50, draw=red, thick] (0,0) ellipse (4.9cm and 3cm);
  
  \draw[fill=yellow!50, draw=orange, thick] (0,0) ellipse (3.2cm and 2cm);
  
  \draw[fill=green!30, draw=green!90!black, thick] (0,0) ellipse (1.8cm and 1.2cm);

  \node at (0, 0.8) {\textbf{$\cE^* = \cE(\gamma^*, n\cG^*)$}};
  \node at (1, 0.25) {$\delta_\gamma = 0$};
  \node at (1, -0.25) {$\psi_\gamma = 0$};
  \node at (2.4, 0.25) {$\delta_\gamma = 0$};
  \node at (2.4, -0.25) {$\psi_\gamma > 0$};
  \node at (5.6, 0) {$\delta_\gamma > 0$};
  \node at (2.2, 1.25) {\textbf{$\bar{\cE}^*_R$}};
  \node at (3.7, 1.8) {\textbf{$\bar{\cE}^*$}};
  \node at (3.8, 0.5) {$\delta_\gamma = 0$};
  \node at (3.8, -0.5) {$\psi_\gamma > 0$};
  \node at (4.05, 0) {$\gamma \supset \gamma' \in \bar{\cE}^*_R$};
  \node 
  at (-0.1, 0.1) {\textbf{$\gamma^*$}};
  \node at (0, 0) {\textbf{$\cdot$}};
  \node at (4.9, 2.2) {\textbf{$\Gamma^p$}};
\end{tikzpicture}
\caption{The classes of DAGs \, $\Gamma^p, \bar{\cE}^*, \bar{\cE}^*_R$ and $\cE^*$ are represented by the ovals marked with red, orange, yellow, and green, respectively. Therefore, the class $\bar{\cE}^*_R \setminus \cE^*$ is represented by the region purely marked with yellow, and due to \eqref{assum:faith}, the green region equivalently represents the class $\cE(\gamma^*, n\cG^*)$. Each region is characterized by $\delta_\gamma$ and $\psi_\gamma$.
}
    \label{fig:illst_DE}
    \end{figure}
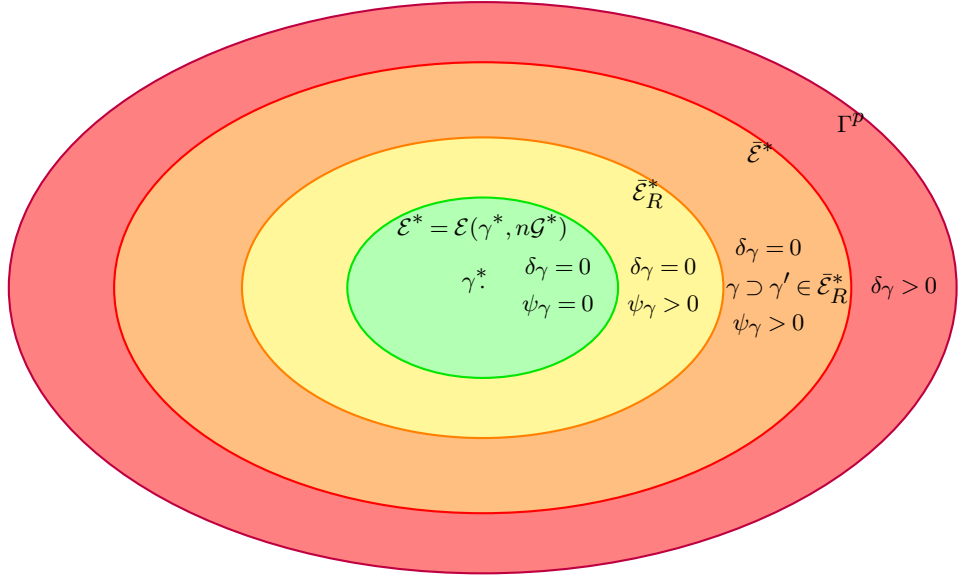

\subsection{Proof sketch of the identifiability results}\label{subsec:pf_sketch_iden}
We provide a brief proof sketch of the identifiability theory, specifically Theorem \ref{thm:H<H} and Theorem \ref{thm:E=DE}, to demonstrate how the Laplace-error working model allows us to seamlessly connect between the probabilistic and graph-theoretic properties. First, without loss of generality, suppose the true causal order $\sigma^*$ is such that $\sigma^*(j) = j$ for every $j \in [p]$, and fix any arbitrary $\gamma \in \Gamma^p$ with the corresponding causal order $\sigma$, and the model parameters $b^\gamma$ and $\theta^\gamma$. Let $e_\sigma^\gamma$ be the random vector whose elements are $e_j^\gamma, j \in [p]$ and ordered according to $\sigma$, i.e., 
\begin{align}\label{eq:e_sig_inv}
    e_\sigma^\gamma := (e_{\sigma^{-1}(1)}^\gamma, e_{\sigma^{-1}(2)}^\gamma, \dots, e_{\sigma^{-1}(p)}^\gamma), \quad \text{implying that} \;\; e^\gamma_{\sigma^{-1}(j)} = a_j^T\epsilon, \;\; j \in [p],
\end{align} 
where $\epsilon = (\epsilon_1, \dots, \epsilon_p)$ collects the error variables from the data generating process \eqref{eq:model}, and the elements of $a_j \in \bR^p$ are some functions of $b_{jk}^\gamma$'s and $\beta^*_{jk}$'s; 
see the discussion around \eqref{eq:defA} in Appendix \ref{app:working}. 
Therefore, if we let $A = ((a_{jk})) \in \bR^{p \times p}$ be such that for every $j \in [p]$, its $j^{\rm th}$ row is $a_j^T$, then clearly $e^\gamma_\sigma = A\epsilon$, and
also, as we show in Lemma \ref{lem:detA}, $\det(A) = 1$ . Now, to prove that $h_\gamma \geq h_*$, it suffices to show, in view of Lemma \ref{lem:H_gamma_basic}, that
\begin{align}\label{ineq:eps_e}
\prod_{j \in [p]} \sfE_*[|\epsilon_j|] \; \leq \; \prod_{j \in [p]}\sfE_*[|e^\gamma_j|].
\end{align}
Due to \eqref{eq:error_dist}, i.e., the errors being scale mixture of Gaussian, and thereby exploiting the closed-form expression of their first absolute moment, it is equivalent to having, as shown in Lemma \ref{lem:exp_mod}, that
\begin{align*}
\prod_{j \in [p]} \sfE_*[\lambda_j] \leq \prod_{j \in [p]} \sfE_*\Big[\big(\sum_{k \in [p]} a_{jk}^2\lambda_k^2\big)^{1/2}\Big]. 
\end{align*}
We prove the above by first constructing an appropriate square matrix, on which we apply Hadamard's inequality \cite{hadamard1893determinants} (Lemma \ref{lem:Hadamard}), and then employing the fact that $\det(A) = 1$, see Lemma \ref{lem:ineq_a_lam}.
Consequently, the equality in the above follows from the conditions of equality in the Hadamard's inequality, which are in turn shown to be equivalent to satisfying either of the following, for every $i, j \in [p]$: 
\begin{enumerate}
\item[(1)] for every $k \in [p]$, $a_{ik}a_{jk} = 0$, or 
\item[(2)] for every $k \in [p]$, such that $a_{ik}a_{jk} \neq 0$, 
$\lambda_k$ is almost surely degenerate, satisfying 
\begin{align*}
\sum_{k \in [p]}
a_{ik}a_{jk} \lambda^2_k \;\; \overset{\rm a.s.}{=} 0, \quad \text{i.e.,} \quad (a_i \circ \lambda)^T(a_j \circ \lambda) \overset{\rm a.s.}{=} 0.
\end{align*}
\end{enumerate}
Next, based on the two conditions above, we extract the structural form of $A$ by using important results from linear algebra, see Lemma \ref{lem:Aprop}. Furthermore, by using Darmois-Skitovic Theorem \cite{darmois1953analyse, skitovitch1953property} (Lemma \ref{lem:DarSki}), the assumption \eqref{assum:faith} in appropriate scenarios, and a series of intermediate lemmas, we derive that, the error terms $e_j^\gamma, j \in [p]$ must be pairwise independent along with the following structure: 
\begin{align*}
e_j^\gamma =
\begin{cases}
\epsilon_j \qquad &\text{if} \;\; j \in n\cG^*\\
\text{some linear combination of} \;\; \epsilon_j, j \notin n\cG^* \qquad &\text{otherwise}
\end{cases}.
\end{align*}
We show that the former case leads us to the parental preservation, i.e., condition (1) in Theorem \ref{thm:H<H}, and the latter one to condition (2) in the same. Moreover, a limiting argument over the SEM coefficients, as shown in Lemma \ref{lem:supset_h}, ensures that if the equality in \eqref{ineq:eps_e} holds, then it also holds for any $\gamma' \supseteq \gamma$, thereby establishing Theorem \ref{thm:H<H}. 

Finally, due to the independence of the errors $e_j^\gamma, j \in [p]$, and assumption \eqref{assum:faith}, we prove that if the equality in \eqref{ineq:eps_e} holds, then $\bI(\gamma) \subseteq \bI(\gamma^*)$, see Lemma \ref{lem:IgamsubIgam*}. Given that, it follows from the probabilistic properties of the graphical models, as shown in Lemma \ref{lem:|gam|and|gam*|}, that $|\gamma^*| \leq |\gamma|$, where equality holds if and only if $\bI(\gamma) = \bI(\gamma^*)$. This is crucial for characterizing $\cE^*$ in terms of Markov equivalence and thereby establishing its equality with the distribution equivalence class $\cE(\gamma^*, n\cG^*)$ in Theorem \ref{thm:E=DE}.

\subsection{Characterization of the distribution equivalence class}
Now that Theorem \ref{thm:E=DE} has established the equality between $\cE^*$ and $\cE(\gamma^*, n\cG^*)$, it only remains to justify the asymptotic approximation \eqref{eq:lap_missp} to advance towards posterior DAG selection consistency. 
In addition, 
when 
$\cE^*$ may include more DAGs other than $\gamma^*$, it is of interest to graphically characterize the elements in $\cE^*$ as they will be indistinguishable from $\gamma^*$. Following Theorem \ref{thm:E=DE}, this is equivalent to characterizing the distribution equivalence class $\cE(\gamma^*, n\cG^*)$ by developing some graphical criteria. For this reason, we first introduce some graph-theoretic notations that will be useful in the rest of this paper.

\paragraph*{Graph theoretic notations}

Consider a DAG $\gamma \in \Gamma^p$. We say that there is a \textit{path} between node $k$ to node $j$ if there is a sequence $k = k_0, k_1, \dots, k_q = j$ such that either $(k_{\ell-1} \to k_\ell) \in \gamma$ or $(k_\ell \to k_{\ell-1}) \in \gamma$ for every $\ell \in [q]$. In particular, we say that the path is \textit{directed from} node $k$ to node $j$ if $(k_{\ell-1} \to k_\ell) \in \gamma$ for every $\ell \in [q]$, and node $k$ is called an \textit{ancestor} of node $j$, and node $j$ is called a \textit{descendant} of node $k$ in $\gamma$. We denote by $\an^\gamma(j)$ and $\de^\gamma(j)$ the set of ancestors and the set of descendants of node $j$ in $\gamma$, respectively, and further define $\bar{\an}^\gamma(j) := \an^\gamma(j) \cup \{j\}$, and $\bar{\de}^\gamma(j) := \de^\gamma(j) \cup \{j\}$. In particular, when $\gamma = \gamma^*$, we denote them by $\an^*(j)$, $\bar{\an}^*(j)$, $\de^*(j)$ and $\bar{\de}^*(j)$.

\begin{theorem}[Parental preservation]\label{thm:char_E}
We have
\begin{align*}
\cE(\gamma^*, n\cG^*) &= \{\gamma \in \Gamma^p :  \pa^\gamma(j) = \pa^*(j) \quad \text{for every} \;\; j \in n\cG^* \quad \text{and} \quad \bI(\gamma) = \bI(\gamma^*)\}.
\end{align*}
\end{theorem}
\begin{proof}
The proof can be found in Appendix \ref{pf:thm:char_E}.
\end{proof}

Theorem \ref{thm:char_E} implies that in order for any DAG $\gamma \in \Gamma^p$ to be distribution equivalent to $(\gamma^*, n\cG^*)$, it not only needs to be Markov equivalent to $\gamma^*$ but also requires every node in $n\cG^*$ to have the same parents as in $\gamma^*$. We call the second property the \textit{parental preservation}, which emerges solely because of the presence of non-Gaussian errors. 

\begin{remark}
If we assume all errors are Gaussian, i.e., $n\cG^* = \emptyset$, or equivalently, $\sfP_X^*$ is multivariate Gaussian, then by Theorem \ref{thm:char_E}, $\cE(\gamma^*, n\cG^*)$ clearly reduces to the Markov equivalence class of $\gamma^*$. This leads us to the well-known result in literature (e.g., \cite{geiger2002parameter}) that for Gaussian DAG models, Markov equivalence is equivalent to distribution equivalence.
\end{remark}

In the following corollary of Theorem \ref{thm:char_E}, we present an interesting property of any distributional equivalent DAG that arises as a consequence of parental preservation and Markov equivalence.

 \begin{corollary}
 [Ancestral restriction]\label{cor:anc_res}
 For any $\gamma \in \cE(\gamma^*, n\cG^*)$ and any $k, \ell \in [p]$, if there exists any $j \in n\cG^*$ such that $k \in \bar{\an}^*(j)$ and $\ell \in \de^*(j)$, that is, 
 $k$ is an ancestor of $\ell$ through $j$ in $\gamma^*$,
 then $\ell \notin \an^\gamma(k)$, that is,
 $\ell$ cannot be an ancestor of $k$ in $\gamma$.
 \end{corollary}
\begin{proof}
The proof can be found in Appendix \ref{pf:cor:anc_res}.
\end{proof}

In other words, for any $k, \ell \in [p]$, if there exists any $j \in n\cG^*$ such that there is a directed path from $k$ to $\ell$ through $j$, then there will be no directed path from $\ell$ to $k$ in any $\gamma \in \cE(\gamma^*, n\cG^*)$. We call this property the \textit{ancestral restriction}:
\textit{no true descendant of $j$ is allowed in $\gamma$ to be an ancestor of any true ancestor of $j$}; see Figure \ref{fig:anc_res} for an illustration.

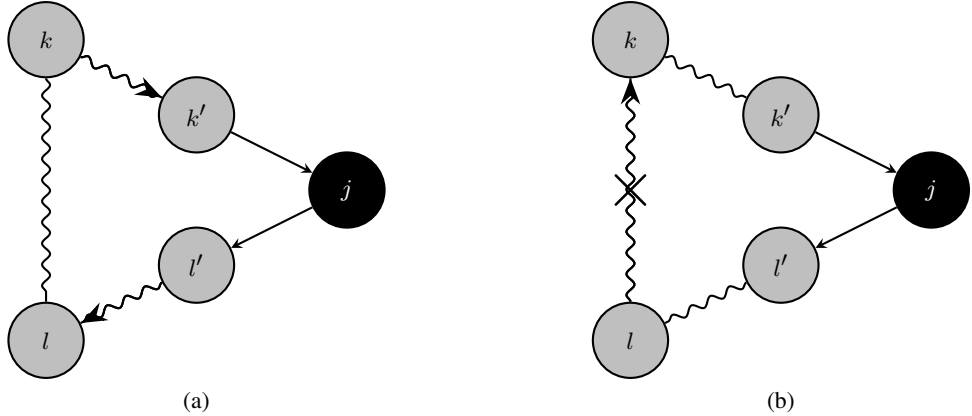
\begin{figure}
\centering
\begin{minipage}{0.45\textwidth}
        \centering
\begin{tikzpicture}
    \node[draw, circle, fill=lightgray, minimum size=1cm, thick] (k1) at (0,5) {$k$}; 
    \node[draw, circle, fill=lightgray, minimum size=1cm, thick] (kp1) at (2,4) {$k'$};
    \node[draw, circle, fill=black, text=white, minimum size=1cm, thick] (j1) at (4,3) {$j$}; 
    \node[draw, circle, fill=lightgray, minimum size=1cm, thick] (lp1) at (2,2) {$l'$};
    \node[draw, circle, fill=lightgray, minimum size=1cm, thick] (l1) at (0,1) {$l$};

    \draw[decorate, decoration={snake, amplitude=1.5pt, segment length=8pt}, thick, postaction={midway, draw, ->, >={Stealth[scale=1.5]}}] (k1) -- (kp1);
    \draw[->, >=stealth, thick] (kp1) -- (j1);
    \draw[->, >=stealth, thick] (j1) -- (lp1);
    \draw[decorate, decoration={snake, amplitude=1.5pt, segment length=8pt}, thick, postaction={midway, draw, ->, >={Stealth[scale=1.5]}}] (lp1) -- (l1);
    \draw[decorate, decoration={snake, amplitude=1.5pt, segment length=8pt}, thick] (k1) -- (l1);
    
     \end{tikzpicture}
        \\ (a) 
    \end{minipage}
    \hfill
\begin{minipage}{0.45\textwidth}
        \centering
        \begin{tikzpicture}

    \node[draw, circle, fill=lightgray, minimum size=1cm, thick] (k2) at (8,5) {$k$}; 
    \node[draw, circle, fill=lightgray, minimum size=1cm, thick] (kp2) at (10,4) {$k'$};
    \node[draw, circle, fill=black, text=white, minimum size=1cm, thick] (j2) at (12,3) {$j$}; 
    \node[draw, circle, fill=lightgray, minimum size=1cm, thick] (lp2) at (10,2) {$l'$};
    \node[draw, circle, fill=lightgray, minimum size=1cm, thick] (l2) at (8,1) {$l$};

    \draw[decorate, decoration={snake, amplitude=1.5pt, segment length=8pt}, thick] (k2) -- (kp2);
    \draw[->, >=stealth, thick] (kp2) -- (j2);
    \draw[->, >=stealth, thick] (j2) -- (lp2);
    \draw[decorate, decoration={snake, amplitude=1.5pt, segment length=8pt}, thick] (lp2) -- (l2); 
    \draw[decorate, decoration={snake, amplitude=1.5pt, segment length=8pt}, thick, postaction={midway, draw, ->, >={Stealth[scale=1.5]}}] (l2) -- (k2);
    
    
    \node at (8,3) {\fontsize{200pt}{240pt}\selectfont $\times$}; 

\end{tikzpicture}
        \\ (b) 
    \end{minipage}
    \caption{The DAG in {\rm (a)} is $\gamma^*$ and in {\rm (b)} is some DAG $\gamma$ that is distribution equivalent to $\gamma^*$. Some nodes are labeled in gray as the corresponding errors are not specified to be non-Gaussian. The wiggly edges denote the existence of paths between the connecting nodes, and if directed, they indicate a directed path. In {\rm (b)}, there is no directed path possible from node $\ell$ to node $k$ due to the ancestral restriction (Corollary \ref{cor:anc_res}).}
    \label{fig:anc_res}
\end{figure}

Furthermore, it is possible to graphically represent the class $\cE(\gamma^*, n\cG^*)$ in a unique fashion. For this, first, we state the notion of \textit{completed partially directed acyclic graph} (CPDAG) which is a graphical representation that uniquely encodes the Markov equivalence class of a DAG, see \cite{spirtes2001causation}. 

\begin{definition}[CPDAG]
The CPDAG of $\gamma \in \Gamma^p$, denoted by ${\rm CPDAG}(\gamma)$, is the mixed graph with the same skeleton of $\gamma$ such that for any 
edge $(j \to k) \in \gamma$, we have 
\begin{align*}
&(j \to k) \in {\rm CPDAG}(\gamma) \;\; \text{if and only if} \;\; (j \to k) \in \gamma' \;\; \text{for every} \;\; \gamma' \;\; \text{such that} \;\; \bI(\gamma') = \bI(\gamma);
\end{align*}
otherwise, we omit the direction to represent it as an undirected edge $(j \mathrel{-\mkern-10mu-} k)$.
\end{definition}

Now, based on ${\rm CPDAG}(\gamma^*)$, we present the following corollary to characterize the set of edges in $\gamma^*$ that must retain their direction in every distribution equivalent DAG.

\begin{corollary}[Graphical criteria for distribution equivalence class]\label{cor:enc_DAG}
For 
any $(j \to k) \in \gamma^*$, we have $(j \to k) \in \gamma$ for every $\gamma \in \cE(\gamma^*, n\cG^*)$ if and only if any of the following condition holds:
\begin{enumerate}
\item[(1)] $(j \to k) \in {\rm CPDAG}(\gamma^*)$, or
\item[(2)] either $j \in n\cG^*$ or $k \in n\cG^*$, or
\item[(3)] if $(j \to k)$ were reversed, then it would either create a new v-structure, produce a cycle, or violate the parental preservation in $\gamma$.
\end{enumerate}
\end{corollary}
\begin{proof}
The proof can be found in Appendix \ref{pf:cor:enc_DAG}.
\end{proof}

Following Corollary \ref{cor:enc_DAG}, we uniquely encode $\cE(\gamma^*, n\cG^*)$ by a mixed graph that we define as ${\rm resCPDAG}(\gamma^*; n\cG^*)$, obtained by imposing further \textit{restrictions} on ${\rm CPDAG}(\gamma^*)$ through the steps below.
\begin{enumerate}
\item[(A)] Extract ${\rm CPDAG}(\gamma^*)$.
\item[(B)] For every undirected edge $(j \mathrel{-\mkern-10mu-} k) \in {\rm CPDAG}(\gamma^*)$, such that either $j \in n\cG^*$ or $k \in n\cG^*$, restore its direction as per $\gamma^*$.
\item[(C)] Finally, orient any additional undirected edges according to Meek's rules \cite{meek1995causal} while ensuring parental preservation, for example, to satisfy the necessary ancestral restrictions. 
\end{enumerate}
These steps are in accordance with the conditions depicted in Corollary \ref{cor:enc_DAG}. Specifically, step (A) guarantees Markov equivalence reflected in condition (1), step (B) ensures parental preservation inscribed in condition (2), and step (C) corresponds to condition (3). A similar algorithm appeared in \cite{hoyer2008causal}, which attempts to derive distribution equivalence patterns under arbitrary error distributions. 
Finally, consider the following illustration. 

\begin{example}[Restricted CPDAG]\label{ex:rcpdag}
Consider $\gamma^*$ to be the DAG in Figure \ref{fig:resCP}(a). Then the distribution equivalence class is encoded by ${\rm resCPDAG}(\gamma^*; n\cG^*)$, as shown in Figure \ref{fig:resCP}(c).

\begin{figure}
\centering

\begin{minipage}{0.45\textwidth}
        \centering
        \begin{tikzpicture}[
        ->, 
        >=stealth, 
        every node/.style={circle, draw, minimum size=1cm}, thick,
        scale=1
    ]
    \node[fill=black, text=white] (A1) at (-3, 4) {1};
    \node[fill=black, text=white] (A2) at (-1, 4) {2};
    \node[fill=black, text=white] (A3) at (1, 4) {3};
    \node (A4) at (3, 4) {4};
    
    \node[fill=black, text=white] (A5) at (-2, 2) {5};
    \node (A6) at (0, 2) {6};
    \node (A7) at (2, 2) {7};
    \node (A8) at (4, 2) {8};
    
    \node (A12) at (-3, 0) {12};
    \node (A11) at (-1, 0) {11};
    \node[fill=black, text=white] (A10) at (1, 0) {10};
    \node (A9) at (3, 0) {9};

    \draw (A8) -> (A4);
    \draw (A4) -> (A3);
    \draw (A3) -> (A2);
    \draw (A2) -> (A1);
    \draw (A8) -> (A7);
    \draw (A2) -> (A6);
    \draw (A1) -> (A5);
    \draw (A2) -> (A5);
    \draw (A8) -> (A9);
    \draw (A6) -> (A7);
    \draw (A7) -> (A9);
    \draw (A6) -> (A9);
    \draw (A5) -> (A12);
    \draw (A5) -> (A11);
    \draw (A11) -> (A10);
    \draw (A9) -> (A10);
    \draw (A12) -> (A11);
    \draw[bend right=45] (A8) to (A6);
    
    \end{tikzpicture}
        \\ (a) 
    \end{minipage}
    \vspace{1cm}
    
    \begin{minipage}{0.45\textwidth}
        \centering
        \begin{tikzpicture}[
        ->,
        >=stealth, 
        every node/.style={circle, draw, fill=white, minimum size=1cm}, thick,
        scale=1
    ]
    \node [fill=black, text=white] (A1) at (-3, 4) {1};
    \node [fill=black, text=white] (A2) at (-1, 4) {2};
    \node [fill=black, text=white] (A3) at (1, 4) {3};
    \node (A4) at (3, 4) {4};

    \node [fill=black, text=white] (A5) at (-2, 2) {5};
    \node (A6) at (0, 2) {6};
    \node (A7) at (2, 2) {7};
    \node (A8) at (4, 2) {8};

    \node (A12) at (-3, 0) {12};
    \node (A11) at (-1, 0) {11};
    \node [fill=black, text=white] (A10) at (1, 0) {10};
    \node (A9) at (3, 0) {9};

    \draw [-] (A8) -- (A4);
    \draw [-] (A4) -- (A3);
    \draw [-] (A3) -- (A2);
    \draw [-] (A2) -- (A1);
    \draw [-] (A8) -- (A7);
    \draw (A2) -> (A6);
    \draw [-] (A1) -- (A5);
    \draw [-] (A2) -- (A5);
    \draw [-] (A8) -> (A9);
    \draw [-] (A6) -- (A7);
    \draw [-] (A7) -- (A9);
    \draw [-] (A6) -- (A9);
    \draw [-] (A5) -- (A12);
    \draw [-] (A5) -- (A11);
    \draw (A11) -> (A10);
    \draw (A9) -> (A10);
    \draw [-] (A12) -- (A11);
    \draw[bend right=45] (A8) to (A6);

    \end{tikzpicture}
        \\ (b) 
    \end{minipage}
\vspace{1cm}
    
    \begin{minipage}{0.45\textwidth}
        \centering
        \begin{tikzpicture}[
        ->, 
        >=stealth, 
        every node/.style={circle, draw, minimum size=1cm}, thick,
        scale=1
    ]
    \node[fill=black, text=white] (B1) at (6, 4) {1};
    \node[fill=black, text=white] (B2) at (8, 4) {2};
    \node[fill=black, text=white] (B3) at (10, 4) {3};
    \node (B4) at (12, 4) {4};
    
    \node[fill=black, text=white] (B5) at (7, 2) {5};
    \node (B6) at (9, 2) {6};
    \node (B7) at (11, 2) {7};
    \node (B8) at (13, 2) {8};
    
    \node (B12) at (6, 0) {12};
    \node (B11) at (8, 0) {11};
    \node[fill=black, text=white] (B10) at (10, 0) {10};
    \node (B9) at (12, 0) {9};

    \draw [-] (B8) -- (B4);
    \draw [red, ultra thick, densely dashed] (B4) -> (B3);
    \draw [red, ultra thick, densely dashed] (B3) -> (B2);
    \draw [red, ultra thick, densely dashed] (B2) -> (B1);
    \draw [blue, ultra thick, loosely dotted] (B8) -> (B7);
    \draw (B2) -> (B6);
    \draw [red, ultra thick, densely dashed] (B1) -> (B5);
    \draw [red, ultra thick, densely dashed] (B2) -> (B5);
    \draw [blue, ultra thick, loosely dotted] (B8) -> (B9);
    \draw [-] (B6) -- (B7);
    \draw [-] (B7) -- (B9);
    \draw [-] (B6) -- (B9);
    \draw [red, ultra thick, densely dashed] (B5) -> (B12);
    \draw [red, ultra thick, densely dashed] (B5) -> (B11);
    \draw (B11) -> (B10);
    \draw (B9) -> (B10);
    \draw [-] (B12) -- (B11);
    \draw[bend right=45] (B8) to (B6);
    
    \end{tikzpicture}
        \\ (c) 
    \end{minipage}

\caption{
The DAG in {\rm (a)} is $\gamma^*$ in Example \ref{ex:rcpdag}. Its CPDAG is shown in {\rm (b)}, which corresponds to step {\rm (A)}. The ${\rm resCPDAG}(\gamma^*; n\cG^*)$ is shown in {\rm (c)}, where the highlighted directed edges correspond to the additional restrictions beyond those imposed by the CPDAG: the dashed red ones are due to step {\rm (B)}, and the blue dotted ones are due to step {\rm (C)}.}
\label{fig:resCP}
\end{figure}
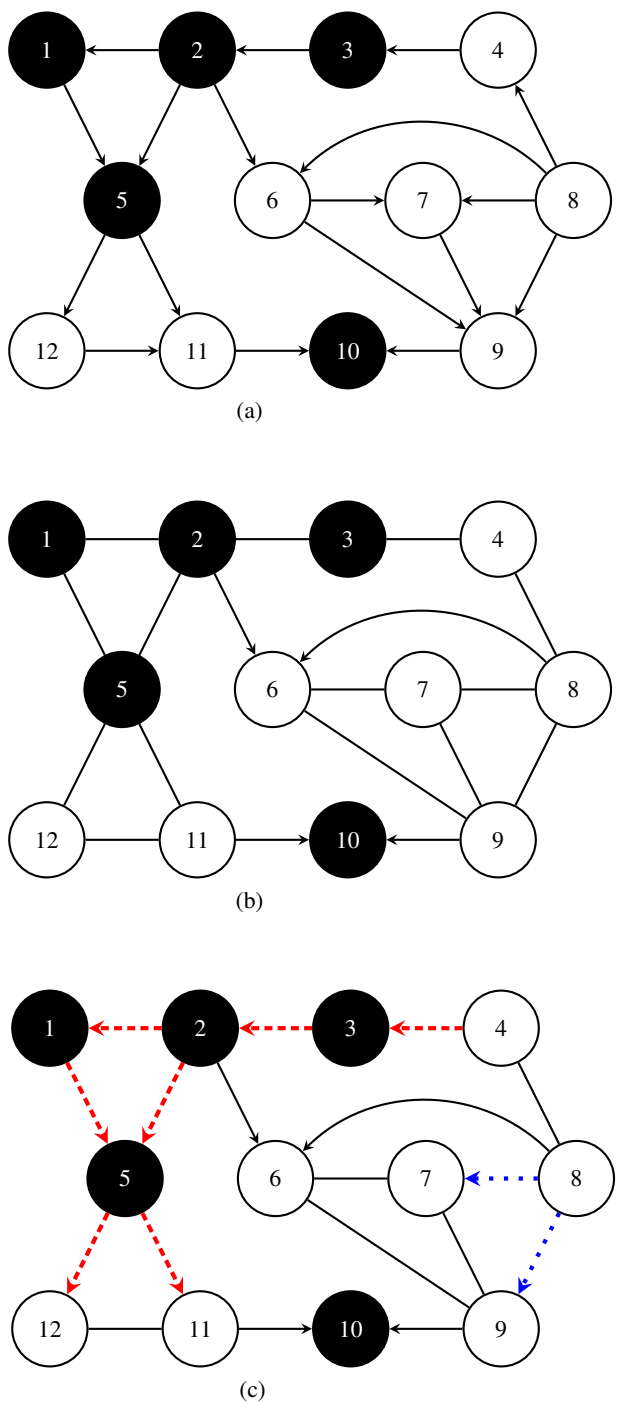

\end{example}

\section{Posterior DAG selection consistency}\label{sec:asym_prop}

In this section, we establish the posterior DAG selection consistency of the proposed method. Specifically, we prove that under the assumption of finite second moment of the mixing variables and, in certain cases, the assumption of faithfulness, the posterior probability of the distribution equivalent class $\cE(\gamma^*, n\cG^*)$ converges to unity as the sample size grows.

 \paragraph*{Distributional assumptions}
 
Formally, we assume that the mixing variables, i.e., the scale parameters in \eqref{eq:error_dist}, have (unknown) finite second moments,
\begin{align}\label{assum:fin_sec}
\sfE_*[\lambda_j^2] < \infty \quad \text{for every} \;\; j \in [p].
\end{align}
This implies that the errors $\epsilon_j, j \in [p]$ also have finite second moments.
 
 \begin{remark}[Error distribution generality]
 This assumption encompasses most of the well-known choices that we mentioned earlier such as contaminated Gaussian, Laplace, Logistic, Student's t, etc. More importantly, it includes various heavy-tailed distributions, for example, Student's t with degree of freedom larger than $2$ and generalized hyperbolic distribution.
 \end{remark}
 
 
 \paragraph*{Laplace approximation}
 
 The consistency property of our method is established based on some approximation results, as we have described briefly in Section \ref{sec:sketch}. These results are formally curated in the following theorem that provides us with a strong foundation for our asymptotic theory, and in fact, could also be of independent interest to the readers. To be specific, we derive a version of the Laplace approximation for the logarithm of the marginal likelihood in terms of the corresponding risk value and the number of associated parameters. Here and elsewhere in this section, we assume by default {\em local priors} such as the g-prior or ridge prior (see Remark \ref{rem:prior}) on $b_j^\gamma$.  
 \begin{theorem}[Laplace approximation]\label{thm:lap_app}
 Suppose that \eqref{assum:fin_sec} holds. Then for every $\gamma \in \Gamma^p$, we have, as $n \to \infty$,
\begin{align*}
\log m\lt(D_n | \gamma\rt) 
&=\max_{(b^\gamma, \theta^\gamma)} \, \log \cL(D_n | b^\gamma, \theta^\gamma, \gamma) - \frac{p + |\gamma|}{2} \, \log n + c_\gamma + o_p(1)\\
&= - \, n \, h_\gamma 
(1 + O_p(n^{-1/2})) \; - \; \frac{p+|\gamma|}{2} \log n + c_\gamma + o_p(1),
\end{align*}
where $c_\gamma$ is some positive constant (free of $n$) depending on $\gamma$, and the $O_p$ and $o_p$ statements are under $\sfP^*$.
 \end{theorem} 
 
 \begin{proof}
  The proof can be found in Appendix \ref{pf:thm:lap_app}.
 \end{proof}
 
Theorem \ref{thm:lap_app} allows us to bypass the analytical intractability of the marginal likelihoods in calculating the Bayes factors and posterior odds. 
The first equality in Theorem \ref{thm:lap_app} gives the more familiar version of the Laplace approximation, where the logarithm of the marginal likelihood is related to the maximized log-likelihood. The second equality further connects it to the negative expected log-likelihood $H^\gamma(\cdot)$ evaluated at the pseudo-true value $(\tilde{b}^\gamma, \tilde{\theta}^\gamma)$, and thereby incurs an additional stochastic term -- this connection is achieved using Lemma \ref{lem:conv_Ln} which exploits a representation of the maximized log-likelihood function in terms of the MLEs of the scale parameters, thereby avoiding more involved empirical process based arguments. The standard Bayesian penalty for model complexity shows up in terms of the number of model parameters times $\log n/2$. 

Establishing this result encounters challenges due to model misspecification and non-differentiability of the likelihood function. In well-specified models with thrice differentiable log-likelihood functions plus additional standard regularity conditions, the Laplace approximation follows from a quadratic expansion of the log-likelihood function around the maximum likelihood estimator; see  \cite[Remark 1.4.5]{GR2003}. 
Although there are existing works that obtain relevant asymptotic results under such model misspecification \cite{kleijn2012bernstein, bhattacharya2020nonasymptotic}, they are not applicable in our setup for various reasons such as non-differentiability of the associated likelihood function \eqref{eq:likeli_L} preventing necessary Taylor expansions and lack of stronger probability conditions regarding the tail behavior of the errors. In order to circumvent this challenge, we establish an alternative Taylor-like decomposition of the log-likelihood function exploiting 
log-concavity of the likelihood function in the spirit of \cite{hjort2011asymptotics}, and {\it only} use finiteness of the second moment to obtain the desired asymptotic approximations; see Appendix \ref{app:lap} for a cascade of results leading to Theorem \ref{thm:lap_app}.

  \paragraph*{Posterior DAG selection consistency}

Following the identifiability theory derived in Section \ref{sec:identif}, and using the Laplace approximation in Theorem \ref{thm:lap_app}, we now establish the main results of this section that the proposed method achieves posterior DAG selection consistency, as follows. First, we consider the case when $\cE^* = \cS^*$, or in other words, 
every risk equivalent DAG must be a superset of $\gamma^*$, 
and show that, in this case, any typical non-informative DAG prior would be sufficient to achieve the desired consistency.

 \begin{theorem}\label{thm:post_cons_exac}
     Suppose that \eqref{assum:fin_sec} holds, and $\bar{\cE}^* = \cS^*$, for example, when any of the conditions in Proposition \ref{prop:spl_cases} is true. Consider any DAG prior $\pi_g(\cdot)$ such that there exists $C > 0$ satisfying $\pi_g(\gamma)/\pi_g(\gamma') \leq C$ for every $\gamma, \gamma' \in \Gamma^p$. Then we have
    \begin{align*}
        \pi(\gamma^* | D_n) \to 1, \quad \text{in} \;\; \sfP^*\text{-probability}.
    \end{align*}
     \end{theorem}
     \begin{proof}
         The proof can be found in Appendix \ref{pf:thm:post_cons_exac}.
     \end{proof}
The condition on the prior $\pi_g(\cdot)$ is very mild, and is satisfied by any DAG prior which is strictly positive over $\Gamma^p$ and free of $n$. A key ingredient in the proof of Theorem \ref{thm:post_cons_exac} is Lemma \ref{lem:wilks_missp}, which establishes that the remainder term $R_n$ in \eqref{eq:approx_BF} is $O_p(1)$ whenever $\gamma \in \cS^*$. In the well-specified setting, this is a ramification of classical Wilk's phenomenon \cite{wilks1938large}, which however does not directly apply to the present setting due to model misspecification and non-differentiability of the likelihood function. 

Now, we focus on the more general case, when $\cS^* \subseteq \bar{\cE}^*$ (especially when the containment is strict), or equivalently, in view of Corollary \ref{cor:ineq_risk_and_param}, $\cE^* \subseteq \bar{\cE}^*_R$, i.e., there may exist some 
DAG $\gamma$ outside $\cE^*$ that can \textit{represent} $\sfP_X^*$, but with \textit{more edges} than $\gamma^*$, see Corollary \ref{cor:ineq_risk_and_param}. Thus, in this case, it becomes imperative to exclude such DAGs to recover $\cE^*$, and for that, we consider a \textit{complexity} prior, as stated in the next theorem, which penalizes DAGs with more edges (complexity) appropriately, and in turn facilitates the desired consistency.

\begin{theorem}\label{thm:post_cons_gen}
    Suppose that \eqref{assum:faith} and \eqref{assum:fin_sec} hold. 
    Consider the DAG prior $\pi_g(\cdot)$ such that for any arbitrary constant $\alpha \in (1/2, 1)$, 
    \begin{align*}
        \pi_g(\gamma) \propto \exp\lt(-n^\alpha d_n |\gamma|\rt), \qquad \text{for every} \;\; \gamma \in \Gamma^p,
    \end{align*}
    where $d_n$ is any (stochastically) bounded positive sequence, possibly data-dependent.
    Then 
    \begin{align*}
        {\rm Pr}(\gamma \in \cE(\gamma^*, n\cG^*) | D_n) \; {\to} \; 1, \quad \text{in} \;\; \sfP^*\text{-probability}.
    \end{align*}
\end{theorem}

\begin{proof}
  The proof can be found in Appendix \ref{pf:thm:post_cons_gen}.
  \end{proof}

The above theorem establishes that in the long run, the proposed method correctly identifies the distribution equivalence class $\cE(\gamma^*, n\cG^*)$ by specifically showing that the posterior probability that the DAG must belong to this class tends, in probability, to one as the sample size grows. We briefly comment on the role of the complexity prior $\pi_g(\cdot)$ in Theorem \ref{thm:post_cons_gen}. Unlike Theorem \ref{thm:post_cons_exac}, we now have $\bar{\cE}^* \setminus \cS^* \ne \emptyset$, i.e., there exist risk-equivalent DAGs that are not supergraphs of $\gamma^*$ (refer to Example \ref{ex:barE-E} and Figure \ref{fig:risk_equiv} for a concrete example). For any $\gamma \in \bar{\cE}^* \setminus \cS^*$, all we can claim about the remainder term $R_n$ in 
\eqref{eq:approx_BF} is that it is $O_p(\sqrt{n})$ (in contrast, recall from the discussion after Theorem \ref{thm:post_cons_exac} that $R_n = O_p(1)$ for $\gamma \in \cS^*$), which therefore becomes the leading contribution to the log-Bayes factor in \eqref{eq:approx_BF} due to risk-equivalence, $h_\gamma = h_*$. To differentiate such $\gamma$ from $\gamma^*$, we exploit the fact that such $\gamma$ must involve more edges, i.e., $\psi_\gamma > 0$ (see Corollary \ref{cor:ineq_risk_and_param}). The condition $\alpha > 1/2$ in the complexity prior adequately {\it penalizes} these additional edges to overcome the stochastic contribution from $R_n$, whereas the condition $\alpha < 1$ ensures that for all $\gamma \notin \bar{\cE}^*$, the term $n(h_\gamma - h_*)$ remains the leading contribution to $\log \Pi_n(\gamma^*, \gamma)$.

Note that, for nested models, non-local priors \citep{johnson2012bayesian} are known to discard spurious parameters at a faster rate compared to local priors. However, since the main purpose of the proposed complexity prior is to distinguish between $\gamma\in\bar {\cE}^* \setminus \cS^*$ and $\gamma^*$, which are non-nested models, it is not immediate whether standard non-local priors can achieve the same. We leave this as an avenue for future work.    


\begin{remark}[Difference in convergence rate]
The in-probability convergences stated in Theorems \ref{thm:post_cons_exac} and \ref{thm:post_cons_gen} are primarily attributed to
the divergence of the posterior odds $\Pi_n(\gamma^*, \gamma), \; \gamma \notin \cE^*$; see Lemma \ref{lem:PrE}.
We obtain some interesting facts about the rate of divergence of the posterior odds. 
Specifically, in Appendix \ref{pf:thm:post_cons_exac} we show that in the context of Theorem \ref{thm:post_cons_exac},
$\Pi_n(\gamma^*, \gamma)$ diverges to infinity in a \textit{polynomial} rate when $\gamma$ is a superset of $\gamma^*$, i.e., 
$\gamma \in \bar{\cE}^* \setminus \cE^*$, whereas the divergence is \textit{exponentially} fast when $\gamma$ is not risk equivalent, i.e., $\gamma \notin \bar{\cE}^*$.
Formally, we derive that, when $n$ is large,
\begin{align*}
\Pi_n(\gamma^*, \gamma) = 
\begin{cases}
        n^{\psi_\gamma/2}e^{O_p(1)} \qquad &\text{if} \;\; \gamma \in \bar{\cE}^* \setminus \cE^*\\
         \exp\lt(n(\delta_\gamma + O_p(n^{-1/2}))\rt)   
         \qquad &\text{otherwise}
        \end{cases},
\end{align*}
where $\psi_\gamma$ and $\delta_\gamma$ are the differences in the numbers
of edges and the risks, respectively, defined in \eqref{def:del_psi}. This polynomial versus exponential rates of divergence of the Bayes factor has been observed more generally 
\cite{johnson2010use}. Careful usage of non-local priors \cite{johnson2012bayesian,rossell2021approximate} on the coefficients $b_j^\gamma$ may improve the polynomial rate in Theorem \ref{thm:post_cons_exac} to a faster polynomial or even exponential rate.  

Furthermore, in Appendix \ref{pf:thm:post_cons_gen} we derive that under the complexity prior, the polynomial divergence rate above becomes exponential but with an exponent of order $n^\alpha$, or more specifically, of the form
$\exp(n^\alpha(d_n \psi_\gamma + O_p(n^{1/2-\alpha})))$.  
In this way, the difference in the risk values and the number of edges are reflected in the rates of divergences of the posterior odds. 
Thus, when $\bar{\cE}^* \setminus \cE^* \neq \emptyset$, applying the above we obtain the following rate (see Lemma \ref{lem:PrE}), 
\begin{align*}
    1 - {\rm Pr}(\gamma \in \cE(\gamma^*, n\cG^*) | D_n)
    \;\, \text{is of order} \;
\begin{cases}
n^{-\psi/2}e^{O_p(1)} &\text{in Theorem \ref{thm:post_cons_exac}}\\
\exp(-n^\alpha(d_n\psi + O_p(n^{1/2-\alpha}))) 
 &\text{in Theorem \ref{thm:post_cons_gen}}
\end{cases}  
\end{align*}
where $\psi := \min_{\gamma \in \bar{\cE}^* \setminus \cE^*}\psi_\gamma$ 
is positive due to Corollary \ref{cor:ineq_risk_and_param}. 
When $\cE^* = \bar{\cE}^*$, for example, when $\gamma^*$ is a complete graph, the above rate is exponential, specifically of the form $\exp\lt(-n(\delta + O_p(n^{-1/2}))\rt)$, where $\delta := \min_{\gamma \notin \bar{\cE}^*}\delta_\gamma$ is positive again due to Corollary \ref{cor:ineq_risk_and_param}. 
\end{remark}

\begin{remark}[Practical choice of $d_n$]\label{rem:d_n}
Although deterministic choices of $d_n$ already guarantee the posterior consistency in Theorem \ref{thm:post_cons_gen}, we allow it to be stochastic mainly for improved finite sample model selection performance. 
Specifically, we first derive in Appendix \ref{pf:thm:post_cons_gen} that,
\begin{align*}
\log \Pi_n(\gamma^*, \gamma) 
&= n\, \delta_\gamma + n^{\alpha}d_n \psi_\gamma + \frac{\psi_\gamma}{2} \, \log n + O_p(n^{1/2}),  \qquad \text{for every} \;\; \gamma \in \Gamma^p, 
\end{align*}
and then apply Corollary \ref{cor:ineq_risk_and_param}. However, in the case when $\delta_\gamma > 0$ and $\psi_\gamma < 0$ for some $\gamma \notin \bar{\cE^*}$ (Figure \ref{fig:illst_DE}), the rate of divergence depends on the magnitude of $\delta_\gamma$. To be precise, if $\delta_\gamma$ is very close to $0$, and $d_n$ is chosen as a relatively large constant, then the above divergence is quite slow, thereby affecting the overall rate of in-probability convergence. Therefore, a favorable choice of $d_n$ is some constant $d^*$ such that $\delta_\gamma + d^*\,\psi_\gamma > 0$ for every $\gamma \notin \bar{\cE}^*$; refer to Section \ref{sec:sim_stud} where we approximate such $d^*$ based on data, and by this means, recommend a data-dependent choice of $d_n$ to implement it in the simulation studies.
\end{remark}

\section{Simulation studies}\label{sec:sim_stud}

In this section, we present the results of three simulation studies to illustrate our theoretical results established in Section \ref{sec:identif} and Section \ref{sec:asym_prop}. 
Specifically, in the first study, we consider a setup where $\bar{\cE}^* = \cS^*$, i.e., unique identifiability of the true underlying DAG is feasible, and thereby show posterior consistency with the uniform DAG prior, whereas in the second and third studies, we consider $\cS^* \subset \bar{\cE}^*$, and thus, it is necessary for us to implement the complexity prior to achieve the desired consistency. In each study, we consider $p = 3$, fix some underlying $\gamma^*$, and as mentioned in Remark \ref{rem:prior}, for every $\gamma \in \Gamma^p$ and $j \in [p]$, consider the following priors:
\begin{align*}
    \pi_{b, j}^\gamma(\cdot) \;\; \equiv \;\; \text{N}(\boldsymbol{0}, 100 \,  I_{|\pa^\gamma(j)|}) \qquad \text{and} \qquad \pi_\theta^\gamma (\cdot) \;\; \equiv \;\;  \text{Inv.G}(1, 1).
\end{align*} 
Since the marginal likelihoods $m\lt(D_n | \gamma\rt), \; \gamma \in \Gamma^p$ in \eqref{eq:marg_m} are analytically intractable, we consider \textit{importance sampling} to compute each of them numerically with $10^4$ Monte Carlo iterations. To be specific, for the \textit{importance distributions} of $b_{j}^\gamma, \theta_j^\gamma, \, j \in [p]$, we consider that
\begin{align*}
    b_j^\gamma \;\; \overset{{\rm ind}}{\sim} \;\; \text{multivar.t}_\nu(\hat{b}_{j, n}^\gamma, \hat{\Sigma}_{j,n}^\gamma) \qquad \text{and} \qquad \theta_j^\gamma \;\; \overset{{\rm iid}}{\sim} \;\; 
    \text{Lognormal}\lt(\log\hat{\theta}_{j,n}^\gamma - c_n/2, \, c_n\rt),
\end{align*}
where the parameters are specified by 
\begin{align*}
&c_n = \log(1+1/n), \qquad \nu = 5, \qquad
&&\hat{b}_{j, n}^\gamma = \min_{b_j} \sum_{i \in [n]} \lt|X_j^{(i)} - b_j^TX_{\pa(j)}^{(i)}\rt|,\\
&\hat{\theta}_{j, n}^\gamma = \frac{1}{n}\sum_{i \in [n]} \lt|X_j^{(i)} - \hat{b}_{j, n}^TX_{\pa(j)}^{(i)}\rt|,
&&\hat{\Sigma}^\gamma_{j, n} = \frac{\nu-2}{\nu}\times \frac{\hat{\theta}_{j, n}^\gamma}{2} (D_{n, j}^{\gamma \, T}D_{n, j}^\gamma)^{-1}.
\end{align*}
In order to portray the asymptotic properties, or more specifically, the asymptotic behavior of the posterior probability of the distribution equivalence class, we consider a range of sample sizes $n \in \{100 \times 2^k : k = 4, 5, 6, 7\}$, and for each sample size $n$, we consider $100$ replications of data $D_n$ simulated from the underlying distribution. For each replication, we compute ${\rm Pr}(\gamma \in \cE(\gamma^*, n\cG^*) | D_n)$. 

\paragraph*{First study}
Consider $\gamma^*$ to be the DAG in Figure \ref{fig:sim_stud1} with the following SEM:
\begin{align}\label{eq:sim_stud_SEM1}
\begin{split}
X_1 &= \epsilon_1,\\
X_2 &= 2.5X_1 + \epsilon_2,\\
X_3 &= 1.8X_2 + \epsilon_3,
\end{split}
\end{align}
where  
\begin{align*}
 \epsilon_1\ | \ \lambda_1 \sim \text{N}(0, \lambda_1^2), \quad \text{where} \quad \lambda_1 \sim  \text{Unif}\,[0.2, 0.4], \quad \epsilon_2 \sim \text{N}(0, 0.25), \;\; \text{and} \;\; \epsilon_3 \sim \text{t}_3. 
 \end{align*}
Hence, $\epsilon_1, \epsilon_3$ are non-Gaussian, i.e., $n\cG^* = \{1, 3\}$.
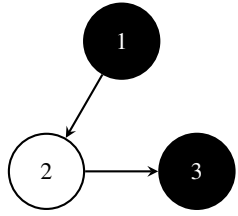
\begin{figure}[H]
  \centering
        \begin{tikzpicture}[->, >=stealth, thick]
            \node (1) [circle, draw, fill=black, text=white, minimum size=1cm] at (1,1.732) {1};
            \node (2) [circle, draw, minimum size=1cm] at (0,0) {2};
            \node (3) [circle, draw, fill=black, text=white, minimum size=1cm] at (2,0) {3};
            \draw (1) -- (2);
            \draw (2) -- (3);
        \end{tikzpicture}
    \caption{DAG $\gamma^*$ in the first study.}
    \label{fig:sim_stud1}
\end{figure}
\noindent By proposition \ref{prop:spl_cases}, we have $\bar{\cE}^* = \cS^*$, i.e., $\cE(\gamma^*, n\cG^*) = \cE^* = \bar{\cE}^*_R = \{\gamma^*\}$, and therefore, following Theorem \ref{thm:post_cons_exac} the uniform DAG prior $\pi_g(\cdot) \propto 1$ is sufficient to lead us to the desired posterior consistency. Indeed, as shown in Figure \ref{fig:study1}(a), the boxplots of ${\rm Pr}(\gamma \in \cE(\gamma^*, n\cG^*) | D_n)$ approaches $1$, demonstrating in-probability convergence of $\pi(\gamma^*|D_n)$ as established in Theorem \ref{thm:post_cons_exac}; see also the histogram of ${\rm Pr}(\gamma \in \cE(\gamma^*, n\cG^*) | D_n)$ 
 for $n = 100 \times2^7$ in Figure \ref{fig:study1}(b), which clearly concentrates at 1.

\begin{figure}[ht]
  \centering
  \begin{minipage}[b]{0.42\textwidth}
    \centering
    \includegraphics[width=\textwidth]{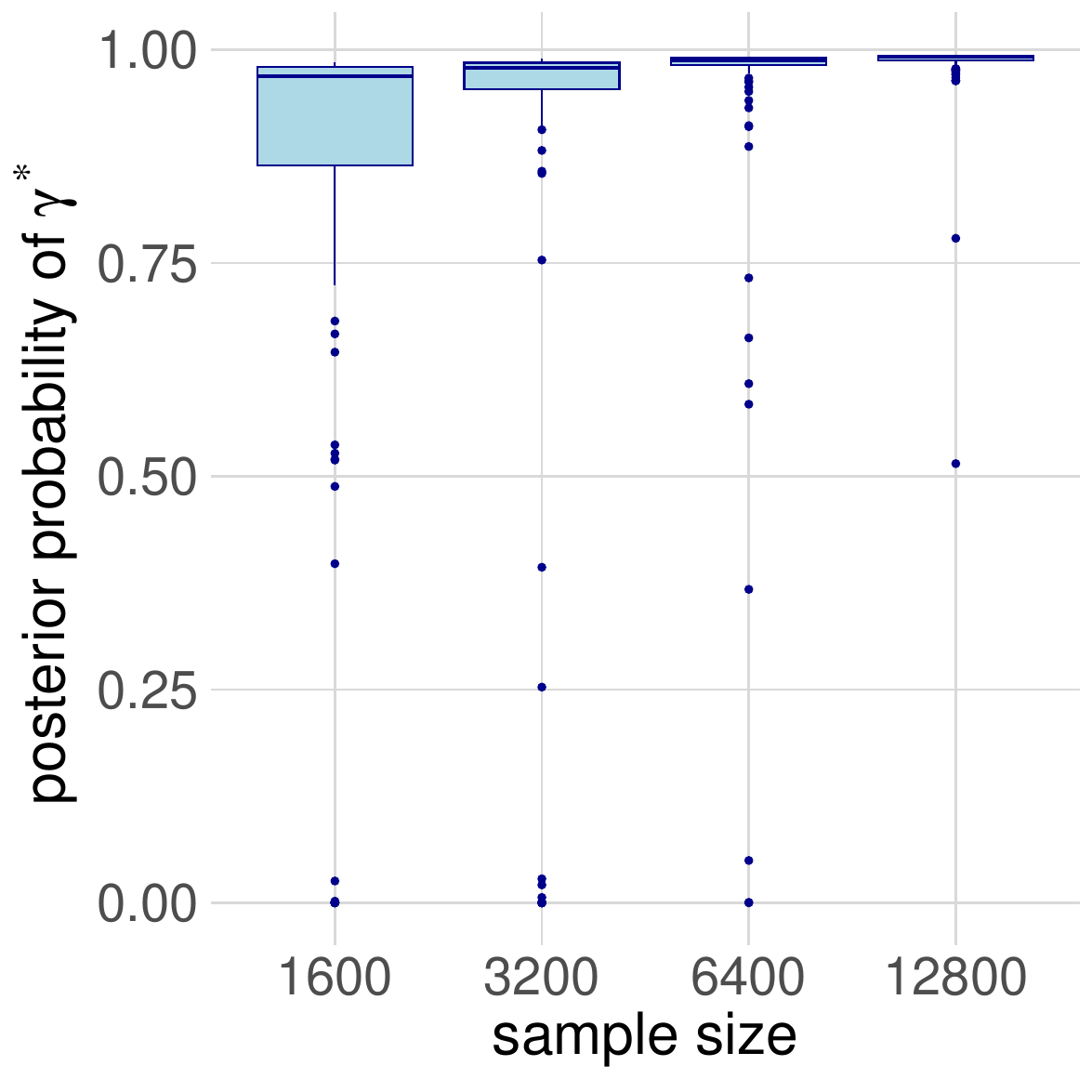}\\
    (a)
  \end{minipage}
  \hfill
  \begin{minipage}[b]{0.42\textwidth}
    \centering
    \includegraphics[width=\textwidth]{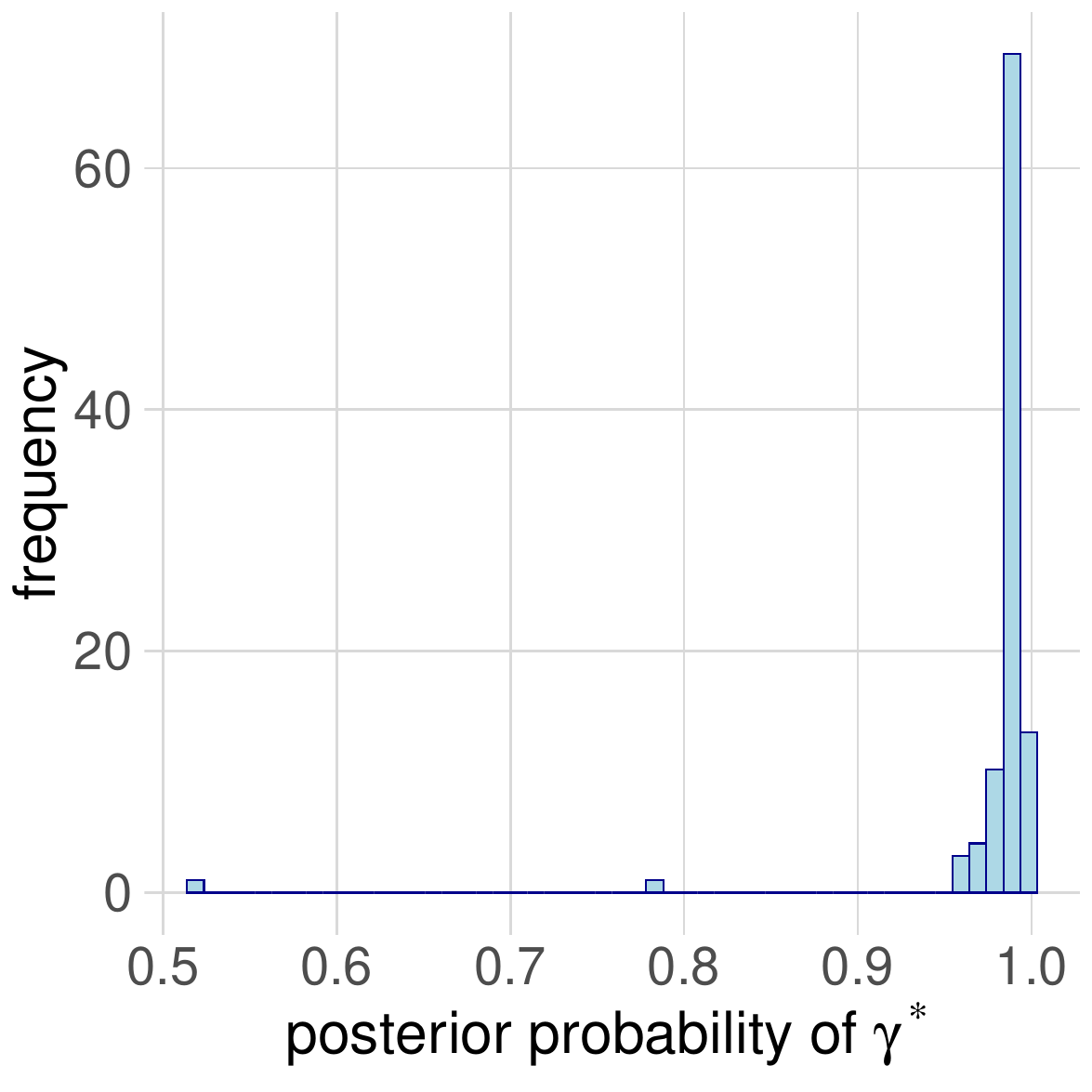}\\
    (b)
  \end{minipage}
  \caption{Results of the first study. Panel (a): boxplots of ${\rm Pr}(\gamma \in \cE(\gamma^*, n\cG^*) | D_n)$ over 100 replicates for four different sample sizes. Panel (b): histogram of ${\rm Pr}(\gamma \in \cE(\gamma^*, n\cG^*) | D_n)$ over 100 replicates for sample size $n = 100 \times2^7$.}
  \label{fig:study1}
\end{figure}

\paragraph*{Second study} 
Consider the setup of Example \ref{ex:barE-E} with the SEM given by \eqref{eq:sim_stud_SEM1}, i.e., the same as in the first study, except that $\epsilon_3$ is Gaussian as $\epsilon_3 \sim \text{N}(0, 0.16)$, and thus, $\gamma^*$ is the DAG in Figure \ref{fig:risk_equiv}(a). If we consider $\gamma$ to be the DAG in Figure \ref{fig:risk_equiv}(b), then as shown in Example \ref{ex:barE-E}, we have $\gamma \in \bar{\cE}^*_R$, and $\{\gamma^*\} = \cE(\gamma^*, n\cG^*) = \cE^* \subset \bar{\cE}^*_R$. Therefore, as indicated earlier, the uniform prior $\pi_g(\cdot) \propto 1$ fails to lead us to the desired posterior consistency, which is clear from Figure \ref{fig:study2_unif}(a), and more specifically, the histogram in Figure \ref{fig:study2_unif}(b) strongly suggests the possibility of in-distribution convergence of $\pi(\gamma^* | D_n)$ to $\text{Ber}(1/2)$. To address this issue, following Theorem \ref{thm:post_cons_gen}, we next employ the complexity prior $\pi_g(\gamma) \propto \exp\lt(-n^\alpha d_n |\gamma|\rt)$, where we choose $\alpha = 0.99$, and in light of Remark \ref{rem:d_n}, $d_n$ is considered as 
\begin{align*}
    d_n = (1/K) \; 
    \min\{\hat{\delta}_n(\gamma, \gamma') : \;{\hat{\delta}_n(\gamma, \gamma') > 0, \; \gamma, \gamma' \in \Gamma^p}\}, 
\end{align*}
where $K = {{p\choose 2}}$ and $\hat{\delta}_n(\gamma, \gamma')$ is the maximum likelihood estimate of the quantity $(\delta_\gamma - \delta_{\gamma'})$. 
Indeed, the in-probability convergence of $\pi(\gamma^* | D_n)$ to $1$ is apparent from the shrinking boxplots in Figure \ref{fig:study2_comp}(a), and the histogram in Figure \ref{fig:study2_comp}(b) concentrating at $1$.  

\begin{figure}[ht]
  \centering
  \begin{minipage}[b]{0.42\textwidth}
    \centering
    \includegraphics[width=\textwidth]{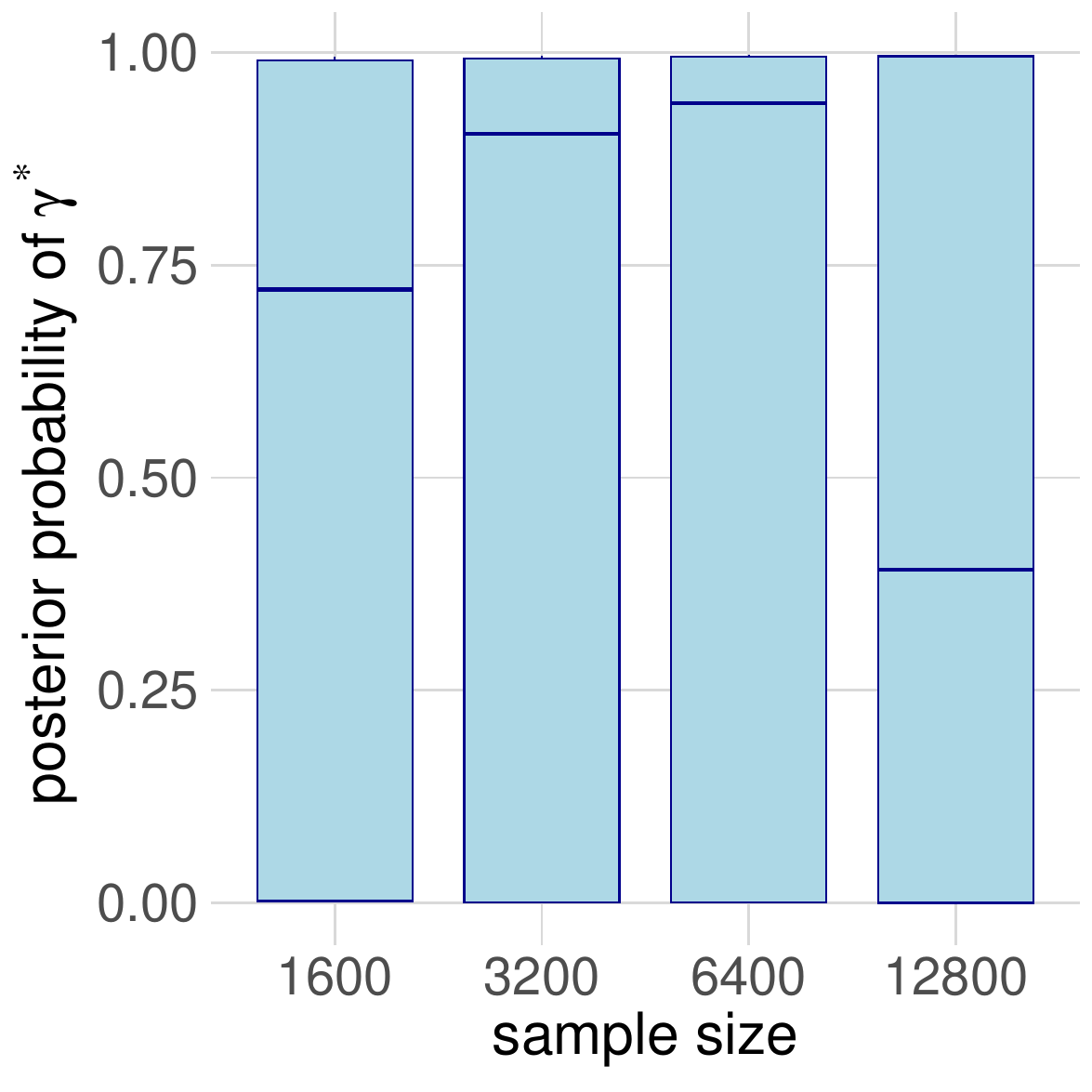}\\
    (a)
  \end{minipage}
  \hfill
  \begin{minipage}[b]{0.42\textwidth}
    \centering
    \includegraphics[width=\textwidth]{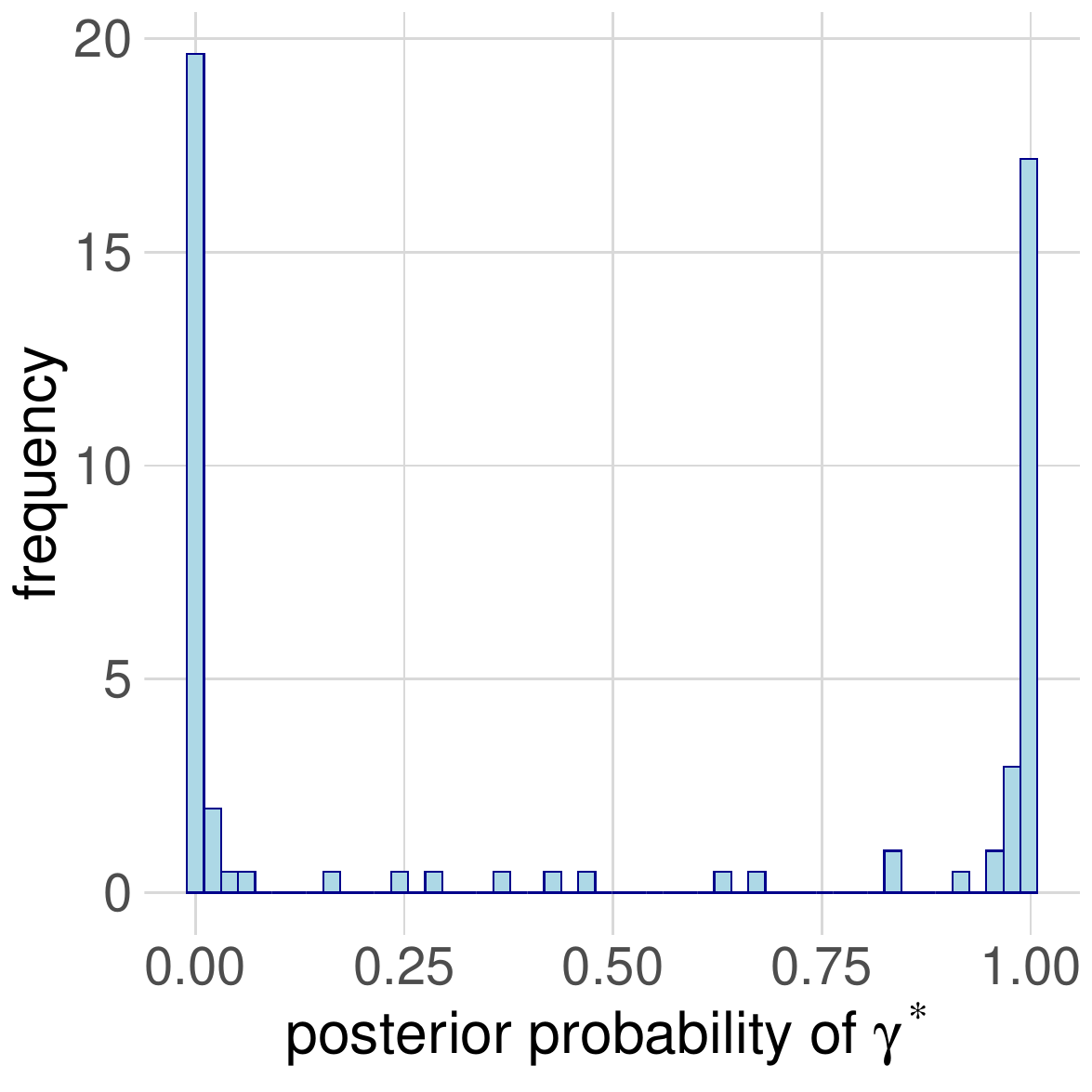}\\
    (b)
  \end{minipage}
  \caption{Same as Figure \ref{fig:study1} but for the second study with the uniform DAG prior.}
  \label{fig:study2_unif}
\end{figure}

\begin{figure}[ht]
  \centering
  \begin{minipage}[b]{0.42\textwidth}
    \centering
    \includegraphics[width=\textwidth]{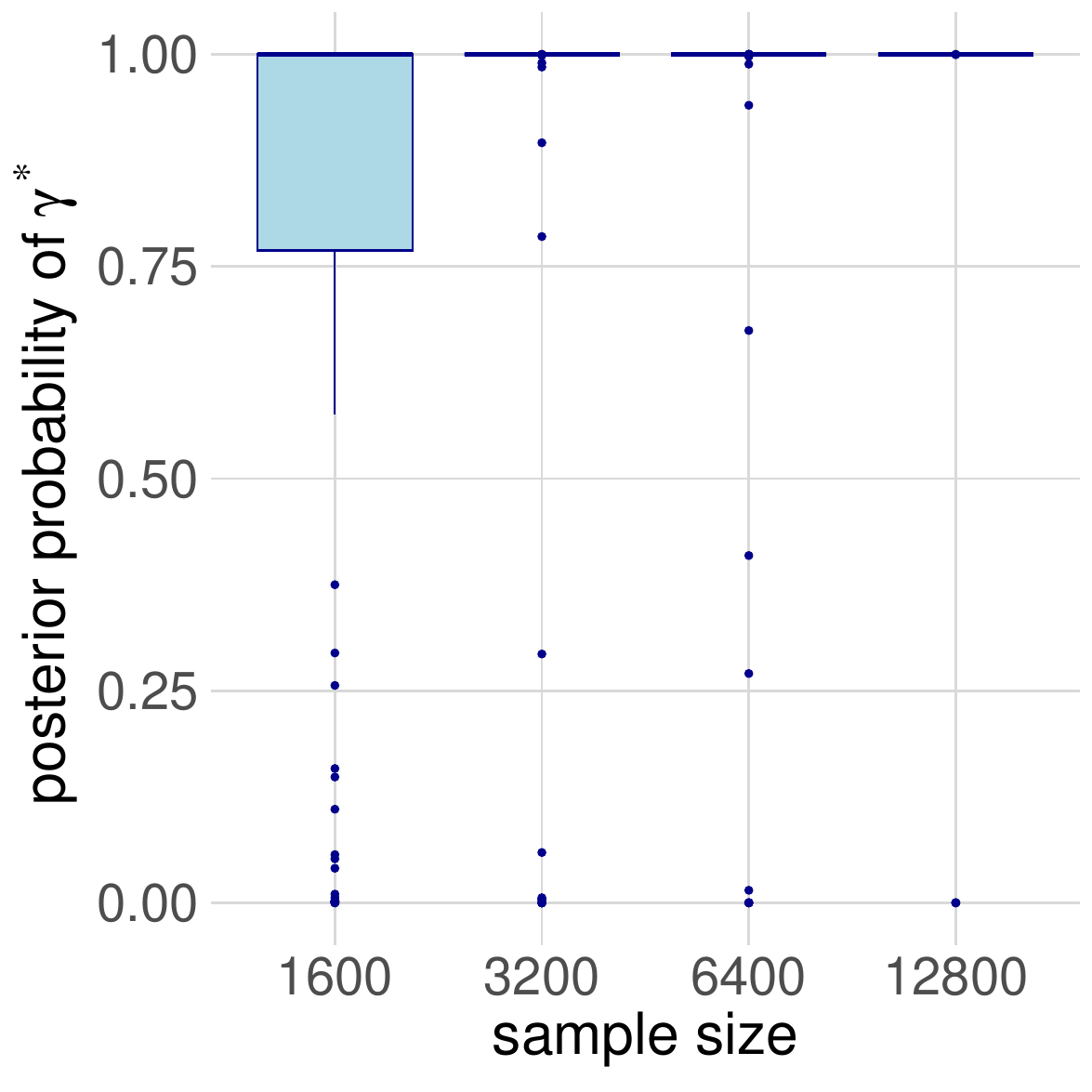}\\
    (a)
  \end{minipage}
  \hfill
  \begin{minipage}[b]{0.42\textwidth}
    \centering
    \includegraphics[width=\textwidth]{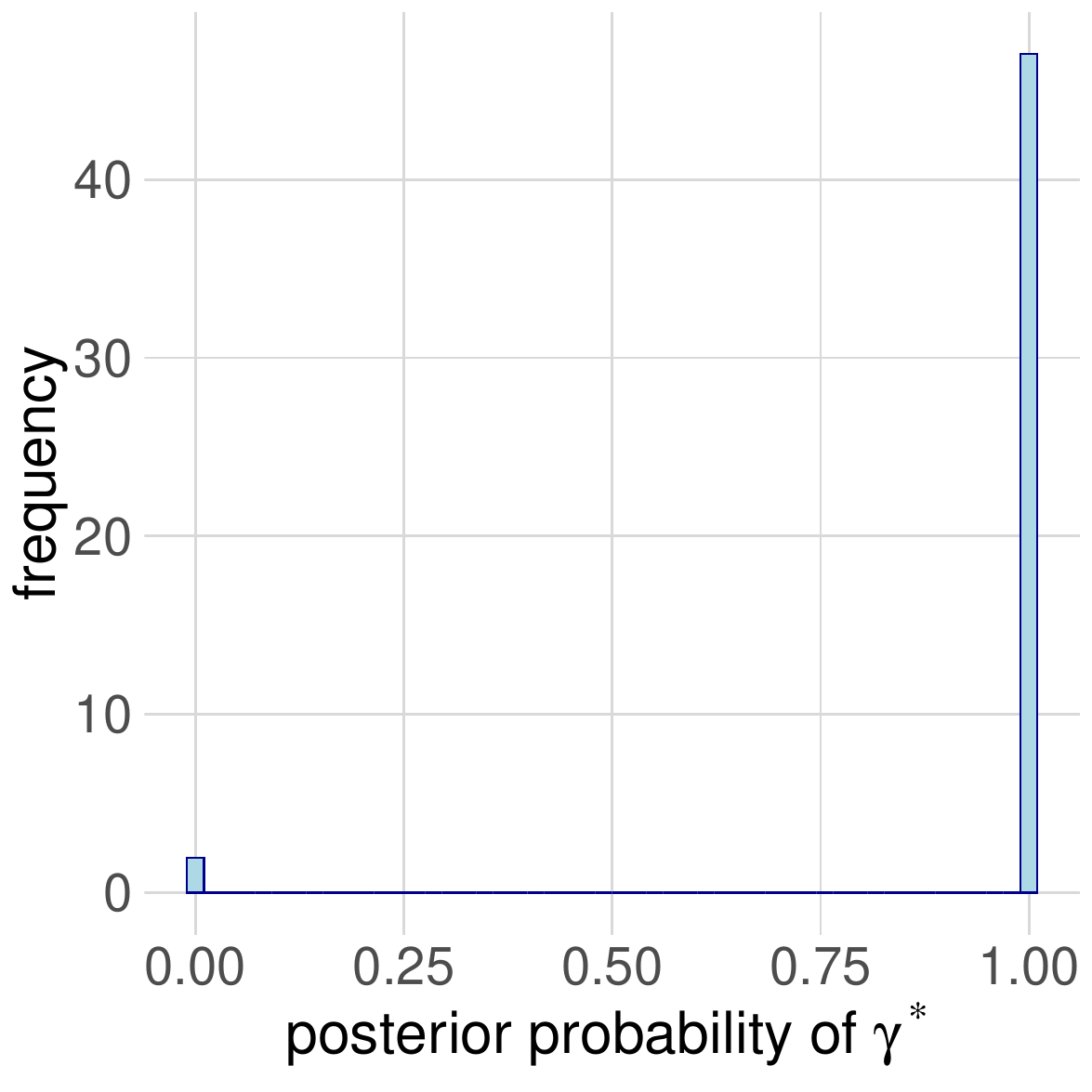}\\
    (b)
  \end{minipage}
  \caption{Same as Figure \ref{fig:study1} but for the second study with the complexity prior.}
  \label{fig:study2_comp}
\end{figure}

\paragraph*{Third study}

Consider the setup of Example \ref{ex:barE=E} with $\gamma^*$ being the DAG in Figure \ref{fig:risk_equiv_E=RE}(a) and the following SEM:
\begin{align*}
X_1 &= \epsilon_1,\\
X_2 &= 2.5X_1 + \epsilon_2,\\
X_3 &= 1.8X_2 + 2.2X_1 + \epsilon_3,
\end{align*}
where $\epsilon_1$ is the only non-Gaussian error, i.e., $n\cG^* = \{1\}$. The distribution of $\epsilon_1$ is the same as in the first study, and those of $\epsilon_2$ and $\epsilon_3$ are the same as in the second study. If we consider $\gamma$ to be the DAG in Figure \ref{fig:risk_equiv_E=RE}(b), then as shown in Example \ref{ex:barE=E}, we have $\gamma \in \bar{\cE}^*_R$, and $\cE(\gamma^*, n\cG^*) = \cE^* = \bar{\cE}^*_R = \{\gamma^*, \gamma\}$. Furthermore, following Theorem \ref{thm:post_cons_gen}, we consider the complexity prior, as outlined in the previous study, to compute ${\rm Pr}(\gamma \in \cE(\gamma^*, n\cG^*) | D_n) = \pi(\gamma^* | D_n) + \pi(\gamma | D_n)$. As expected, the posterior consistency is evident from the boxplots in Figure \ref{fig:study3_pc}(a) and the histogram in Figure \ref{fig:study3_pc}(b). 

\begin{figure}[ht]
  \centering
  \begin{minipage}[b]{0.42\textwidth}
    \centering
    \includegraphics[width=\textwidth]{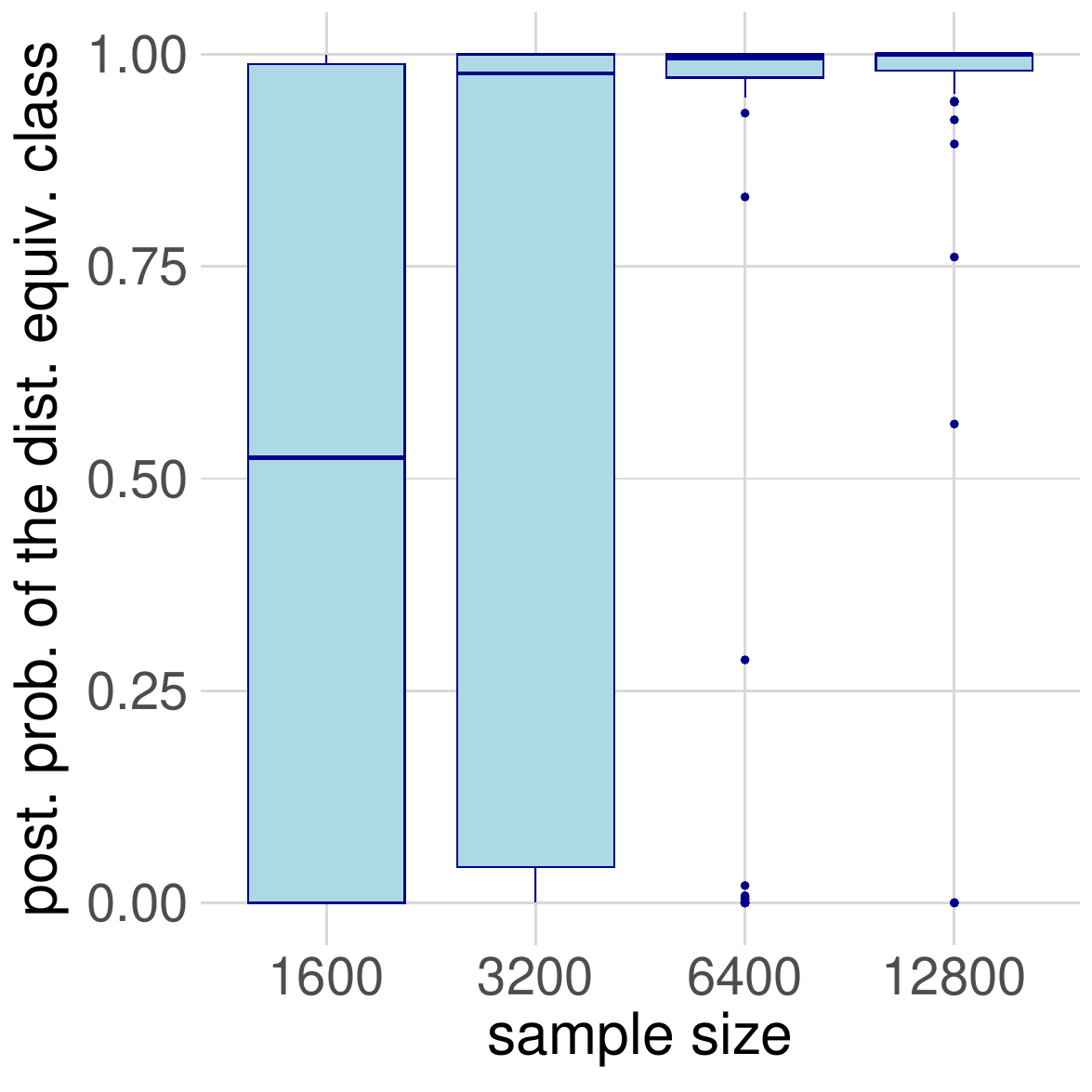}\\
    (a)
  \end{minipage}
  \hfill
  \begin{minipage}[b]{0.42\textwidth}
    \centering
    \includegraphics[width=\textwidth]{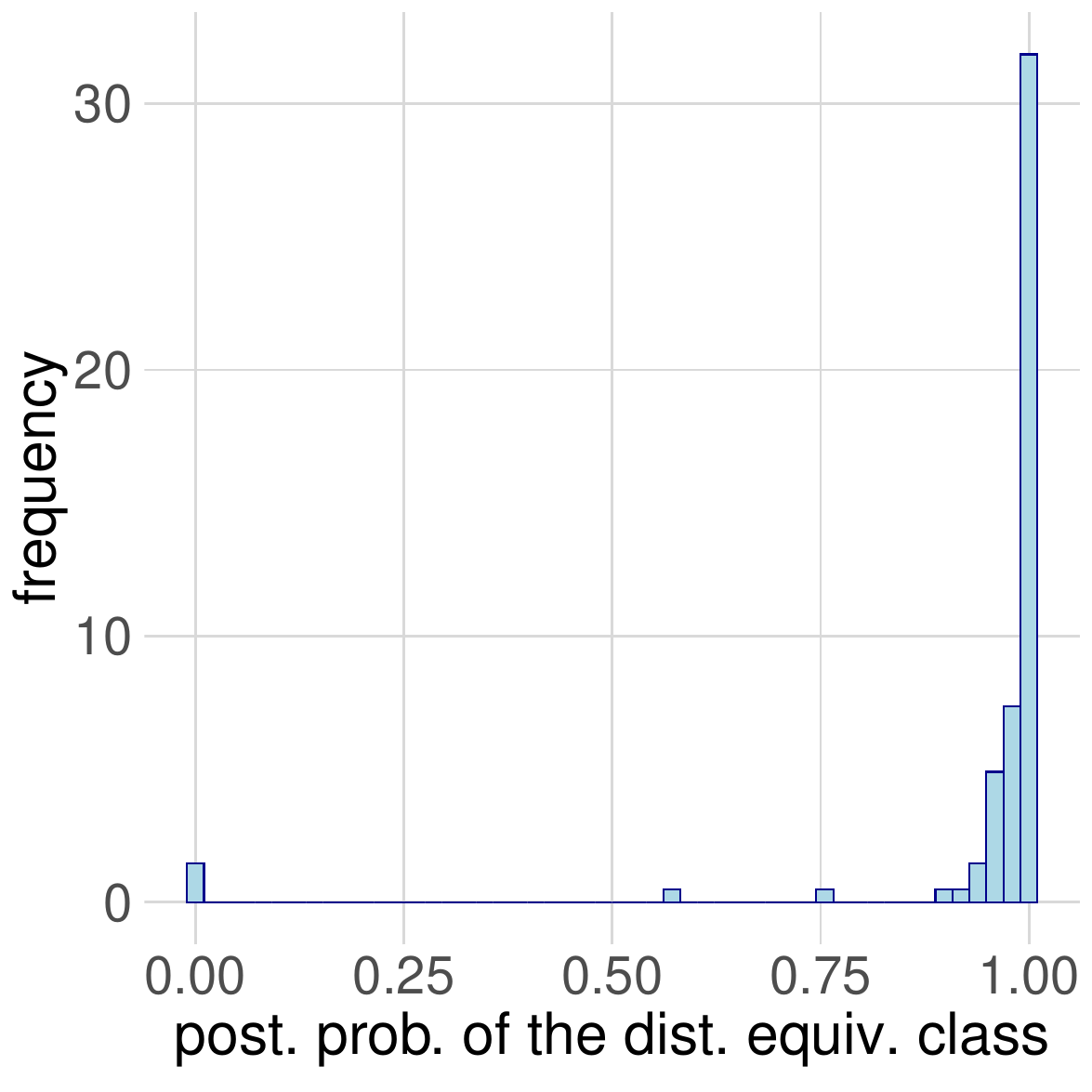}\\
    (b)
  \end{minipage}
  \caption{Same as Figure \ref{fig:study1} but for the third study with the complexity prior.}
  \label{fig:study3_pc}
\end{figure}

In this context, since $\gamma$ is an equivalent DAG model, it is also of interest, in spirit of \cite{huggins2023reproducible}, to investigate the asymptotic behavior of the \textit{posterior share} of $\gamma^*$, defined as the quantity $\pi(\gamma^*|D_n)/(\pi(\gamma^*|D_n) + \pi(\gamma | D_n))$. For this, we include the associated histograms in Figure \ref{fig:study3_ps} representing its asymptotic behavior at different sample sizes under consideration, which suggests the in-distribution convergence of the posterior share to $\text{Ber}(1/2)$.

\begin{figure}[ht]
  \centering
  \begin{minipage}[b]{0.42\textwidth}
    \centering
    \includegraphics[width=\textwidth]{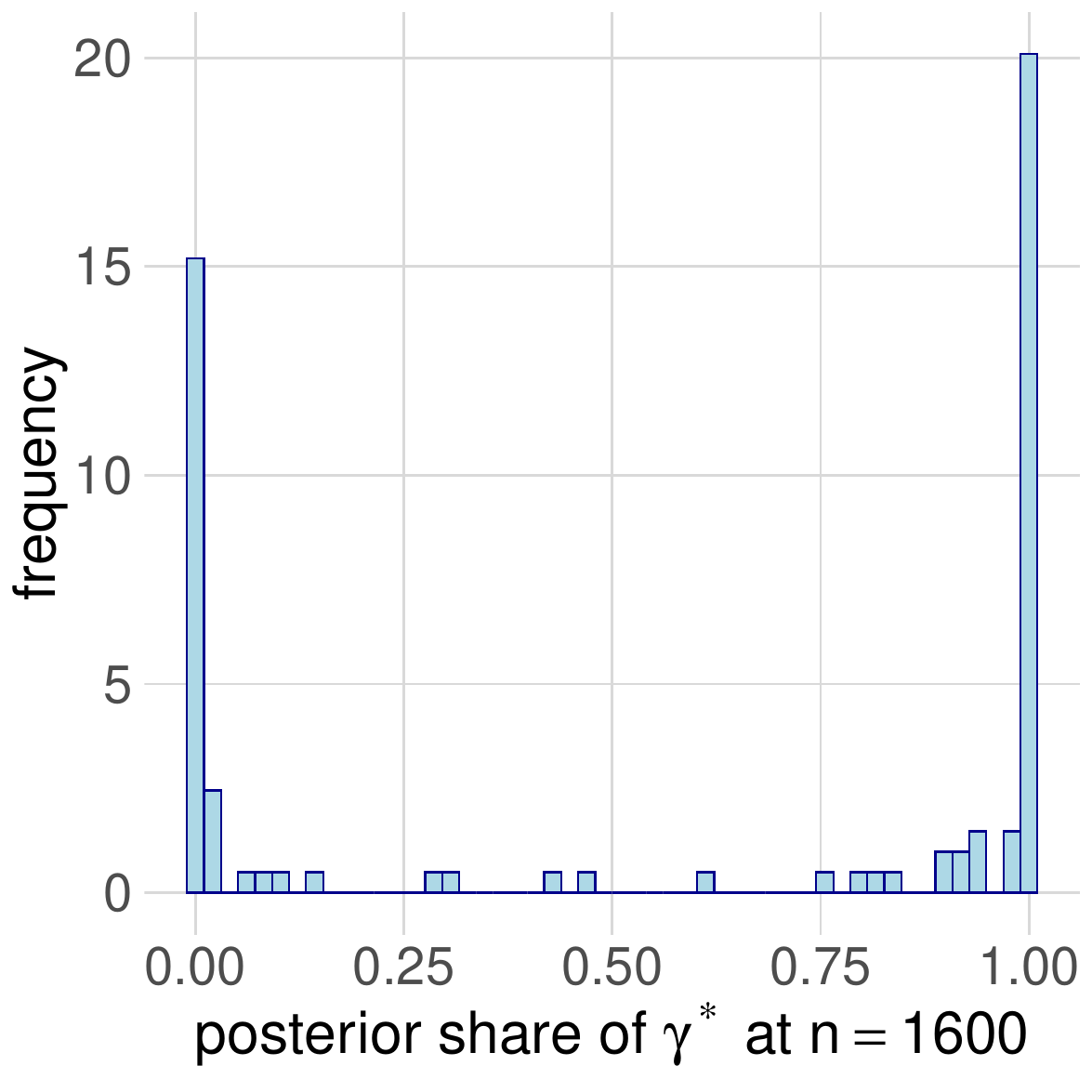}\\
    (a)
  \end{minipage}
  \hfill
  \begin{minipage}[b]{0.42\textwidth}
    \centering
    \includegraphics[width=\textwidth]{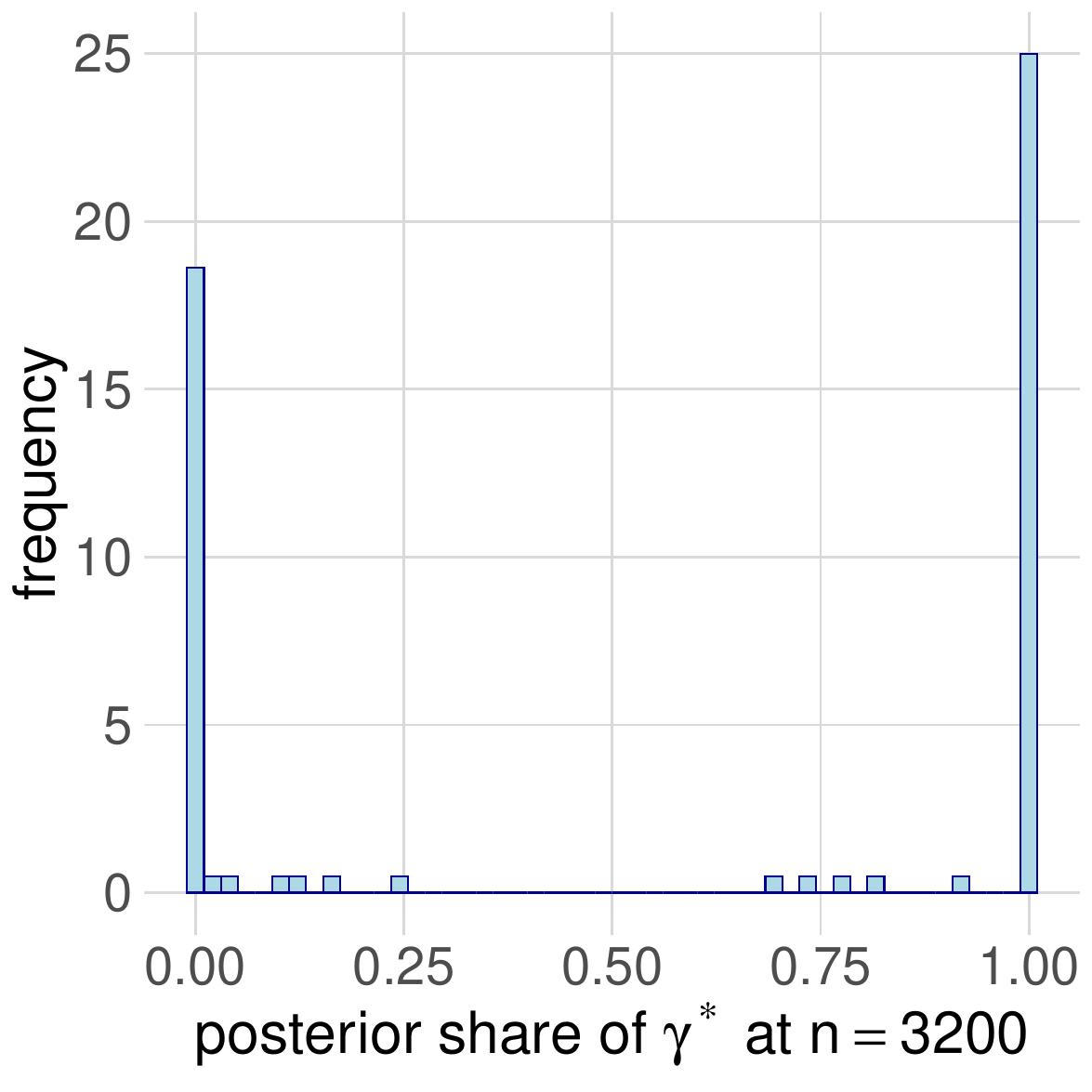}\\
    (b)
  \end{minipage}

  \vspace{0.2cm} 

  \begin{minipage}[b]{0.42\textwidth}
    \centering
    \includegraphics[width=\textwidth]{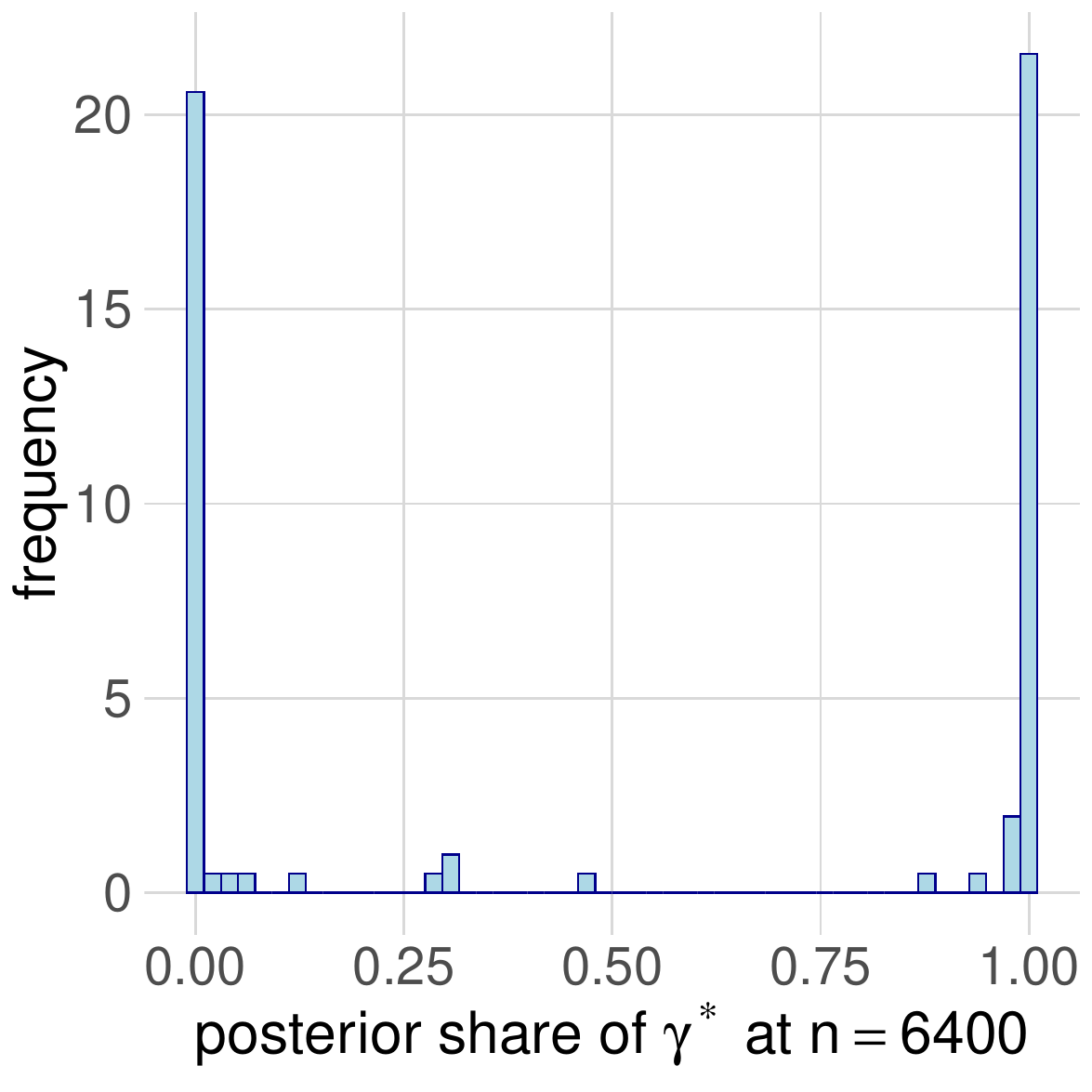}\\
    (c)
  \end{minipage}
  \hfill
  \begin{minipage}[b]{0.42\textwidth}
    \centering
    \includegraphics[width=\textwidth]{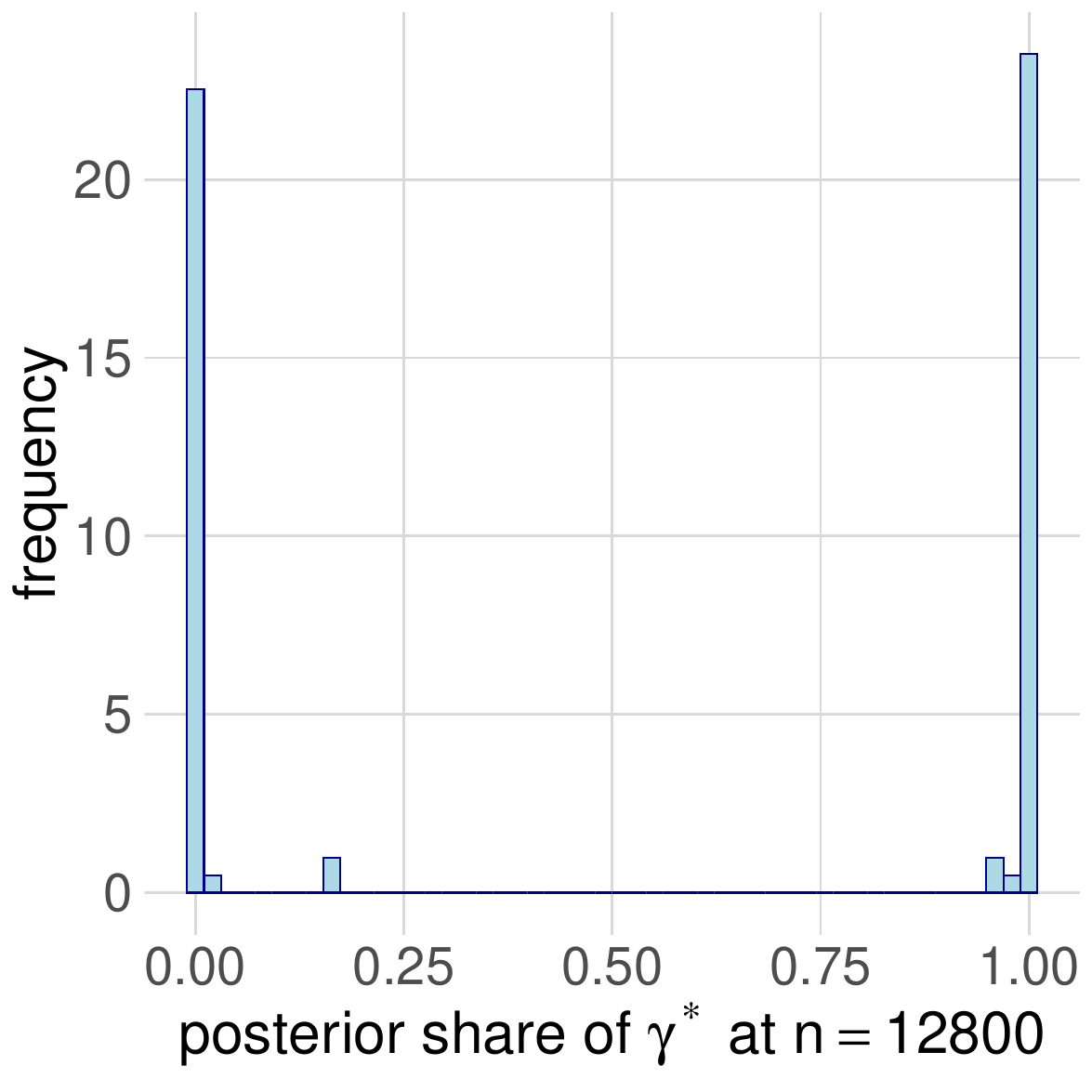}\\
    (d)
  \end{minipage}

  \caption{Histograms of the posterior share of $\gamma^*$ over 100 replicates for sample sizes $n = 100 \times 2^k, k = 4, 5, 6, 7$.}
  \label{fig:study3_ps}
\end{figure}

\section{Conclusion}\label{sec:cnclsn}

In this work, we consider the problem of learning the DAG structure of a linear recursive SEM. The associated error variables in the SEM are assumed to follow some scale mixture of Gaussian, which, unlike most existing works, provides the flexibility that we can not only incorporate non-Gaussian errors but also allow some errors to be Gaussian. In order to identify the unknown data-generating DAG, we propose a Bayesian SEM with Laplace error variables and theoretically study its property when the data-generating SEM does not necessarily have Laplace errors. We establish that our proposed method can consistently recover the true underlying DAG up to its distribution equivalence class, that is, the posterior probability of this class converges to unity as the sample size grows to infinity. Therefore, apart from consistency, our method is also shown to achieve optimality in that further refinement of the equivalence class is not possible without additional assumptions. En route to proving the consistency, we additionally characterize the distribution equivalence classes under an arbitrary combination of Gaussian and non-Gaussian errors, which can be of independent interest to the readers. Finally, our theoretical results show distinct rates of divergence of the Bayes factors depending on the structure of competing DAGs.

There are several natural generalizations of the current work. For instance, it would be interesting to consider more general non-Gaussian distributions and establish similar consistency results for the proposed method. Moreover, we can extend the results of the present work to the high-dimensional setting under additional assumptions, if needed, such as equal error variances, specific tail behaviors of the error distributions, and sparsity conditions \cite{johnson2012bayesian,castillo2015bayesian}.

Finally, there are many open questions for future research, such as designing efficient DAG selection methods for nonlinear SEM and developing similar asymptotic theory. Although it is possible to use basis expansion to accommodate nonlinearity, with the growth in sample size, we typically need to allow the number of basis functions to increase, which induces a high-dimensional scenario even when the number of variables does not grow with the sample size, and thereby appoints some fresh theoretical challenges. Another important direction is to consider directed cyclic graphs or non-recursive SEMs,  which is significantly more challenging because, unlike DAGs, their factorization and Markov equivalence characterizations are more intricate. Apart from that, the absence of conjugate priors and consequently, the intractability of marginal likelihoods poses additional challenges in theoretical analysis similar to the present work. Lastly, another avenue of interest is to consider the presence of latent confounders or correlated errors \cite{wang2023causal, salehkaleybar2020learning, drton2011global, chen2024discovery, li2024nonlinear}.

\begin{funding}
The research of A. Chaudhuri and Y. Ni were supported by NIH R01 GM148974. The research of Y. Ni was additionally supported by NSF DMS-2112943.  The research of A. Bhattacharya was supported partially by NSF DMS-2210689 and NSF DMS-1916371. 
\end{funding}

\begin{supplement}
\textbf{Supplement to "Consistent DAG selection for Bayesian causal discovery under general error distributions"} In the supplement we prove all results and present additional technical lemmas. In Appendix \ref{app:working}, we derive some essential properties of our working model which are utilized to obtain the results in Appendix \ref{app:identif} regarding the identifiability theory. In Appendix \ref{app:lap}, we obtain the Laplace approximation which plays a crucial role in establishing the posterior consistency in Appendix \ref{app:post_const}.
\end{supplement}

\newpage

\begin{appendix}

\section{Some properties of the working model}\label{app:working}

\paragraph*{Notations}

For $p \in \bN$, the family of all permutations of $[p]$ is denoted by $\cT_p$.
For any vector $x$, we denote its $\ell_1$ norm and $\ell_2$ norm by $|x|$ and $||x||$, respectively. Moreover, if $k^{\rm th}$ element of $x$ is denoted by $x_k$, then the support of $x$, denoted by ${\rm supp}(x)$, is defined to be the set of indices of its non-zero elements, i.e., ${\rm supp}(x) = \{k : x_k \neq 0\}$. For any two matrices $A$ and $B$ of the same dimension, we denote their Hadamard product by $A \circ B$.\\

Fix an arbitrary $\gamma \in \Gamma^p$ and consider the corresponding model in \eqref{eq:lap_model_alt}.
Thus, for notational simplicity, in the rest of the paper, we omit the superscript $\gamma$ from the notations $b^\gamma$, $b_{jk}^\gamma$, $\theta^\gamma$, $\theta_j^\gamma$, $e^\gamma_j$, $H^\gamma(\cdot)$, $\pa^\gamma(j)$, $\de^\gamma(j)$, $\an^\gamma(j)$, $\bar{\de}^\gamma(j)$, $\bar{\an}^\gamma(j)$, where $j \in [p]$, $k \in \pa^\gamma(j)$.\\

Now, let the causal order of $\gamma$ be denoted by $\sigma$. 
Then, we define a quantity that captures the total causal effect of an ancestor on a node in $\gamma$, as follows. 
Specifically, we define recursively over the causal order 
 \begin{align}\label{eq:defb*}
\text{for every $j \in [p]$ \, and \, $s \in \an(j)$}, \quad b_{j \gets s} := \sum_{k \in \pa(j) \cap \bar{\de}(s)} b_{jk} b_{k \gets s}, \quad \text{and} \quad b_{j \gets j} \equiv 1.
\end{align}
Moreover, when $\gamma = \gamma^*$, we use the notation $b^*_{j \gets s}$ in an analogous manner,
and in particular, if for every $j \in [p]$ and $k \in \pa^*(j)$, $b^*_{jk} = \beta^*_{jk}$, then we further adapt the notation as $\beta^*_{j \gets s}$.

\begin{lemma}\label{lem:rep_anc}
For every $j \in [p]$, we have
\begin{align*}
X_j = \sum_{\ell \in \bar{\an}(j)} b_{j \gets \ell} e_\ell.
\end{align*}
\end{lemma}
\begin{proof}
We prove this by induction over the causal order $\sigma$. Note that, the hypotheses is trivially true for $j$ such that $\sigma(j) = 1$ since $\bar{\an}(j) = \{j\}$ and $X_j = e_j$. Now, fix $j \in [p]$ for which $\sigma(j) = m$ for some $m > 1$, and suppose that the hypotheses is true for every $j \in \{\ell : 1 \leq \sigma(\ell) \leq m-1\}$. Then,
\begin{align*}
X_j &= \sum_{k \in {\rm pa}(j)} b_{jk} X_{k} + e_j\\
&= \sum_{k \in \pa(j)} b_{jk} \sum_{\ell \in \bar{\an}(k)} b_{k \gets \ell} e_\ell + e_j\\
&= \sum_{k \in \pa(j)} \sum_{\ell \in \bar{\an}(k)} b_{jk} b_{k \gets \ell} e_\ell + e_j\\
&= \sum_{\ell \in \an(j)} \sum_{k \in \pa(j) \cap \bar{\de}(\ell)} b_{jk}b_{k \gets \ell} e_\ell + e_j\\
&= \sum_{\ell \in \an(j)} b_{j \gets \ell} e_\ell + b_{j \gets j} e_j =  \sum_{\ell \in \bar{\an}(j)} b_{j \gets \ell} e_\ell,
\end{align*} 
where the first equality is from \eqref{eq:lap_model_alt} and the second one follows from the induction hypotheses as $\pa(j) \subseteq \{\ell : 1 \leq \sigma(\ell) \leq m-1\}$. The fourth equality follows by rearranging the sum using the fact that, 
\begin{align*}
\an(j) = \pa(j) \cup \bigcup_{k \in \pa(j)} {\an}(k) =  \bigcup_{k \in \pa(j)} \bar{\an}(k),
\end{align*}
i.e., for every $\ell \in \an(j)$, there exists a parent of $j$, $k \in \pa(j)$ such that $\ell \in \bar{\an}(k)$ or equivalently, $k \in \bar{\de}(\ell)$. Finally, the last equality follows by using the definition in \eqref{eq:defb*}. The proof is complete.
\end{proof}

Note that, for every $i \in [p]$, $\sigma^{-1}(i)$ determines the variable whose causal order is $i$. Now, suppose we denote by $X_{\sigma}$ the random vector whose elements are $X_j, j \in [p]$ and ordered according to $\sigma$, i.e.,
\begin{align*}
X_{\sigma} := (X_{\sigma^{-1}(1)}, X_{\sigma^{-1}(2)}, \dots, X_{\sigma^{-1}(p)}),
\end{align*}
and we define $e_{\sigma}$ in a similar way.
Then, in view of Lemma \ref{lem:rep_anc}, we can represent
\begin{align}\label{eq:X_Be}
X_{\sigma} = B e_{\sigma},
\end{align}
where $B \in \bR^{p \times p}$ is a lower triangular matrix such that for every $u, v \in [p]$, $u \geq v$, the $(u, v)^{\rm th}$ element of $B$ is $b_{j \gets k}$, i.e.,
\begin{align*}
B_{uv} = b_{j \gets k}, \qquad \text{where} \;\; \sigma(j) = u \;\;\;  \text{and} \;\;\; \sigma(k) = v,
\end{align*}
and hence, all its diagonal elements are $1$. 

Now, without loss of generality, suppose the true causal order $\sigma^*$ is such that $\sigma^*(j) = j$ for every $j \in [p]$.
Then, according to the true model in \eqref{eq:model}, we have the lower triangular matrix $\cB^*$ such that, in a similar way as above, for every $u, v \in [p]$, $u \geq v$, we have its $(u, v)^{\rm th}$ element $\cB^*_{uv} = \beta^*_{u \gets v}$
to obtain the similar representation
\begin{align}\label{eq:X_Beps}
X = \cB^* \epsilon, \qquad \text{where} \;\; \epsilon = (\epsilon_1, \epsilon_2, \dots, \epsilon_p).
\end{align}

Next, note that there exists a permutation matrix $P$ for which $X_\sigma = P X$, which leads to
\begin{align}\label{eq:defA}
\begin{split}
\qquad &e_\sigma = B^{-1}X_\sigma = B^{-1}P X = B^{-1}P\cB^* \epsilon = A \epsilon,\\ 
\quad &\text{where} \qquad A := B^{-1}P\cB^*.
\end{split}
\end{align}
For every $i \in [p]$, we denote the $i^{\rm th}$ row of $A = ((a_{ij}))$ by $a_i^T$, where $a_i \in \bR^p$.
Moreover, we define the following set of indices related to $A$, which will be useful later,
\begin{align*}
\cR_A := \{i \in [p] : \supp(a_i \circ a_j) \neq \emptyset \;\;\; \text{for some} \;\;\; j \neq i\} \quad \text{and} \quad \cC_A := \bigcup_{i, j \in [p]} {\rm supp}(a_i \circ a_j).
\end{align*}
\begin{lemma}\label{lem:detA}
We have \ $\det(A) = 1$.
\end{lemma}
\begin{proof}
Note that, since $B$ and $\cB^*$ are lower triangular with all their diagonal elements being $1$, $\det(B) = \det(\cB^*) = 1$, and also $\det(P) = 1$.   
Thus, following \eqref{eq:defA},
\begin{align*}
\det(A) = \det(B^{-1}P\cB^*) = \det(B^{-1})\det(P)\det(\cB^*) = \det(B)^{-1} = 1.
\end{align*}
\end{proof}
Now, let $\lambda := (\lambda_1, \lambda_2, \dots, \lambda_p)$, where $\lambda_j, j \in [p]$ are the mixing variables mentioned in \eqref{eq:error_dist}, and we define a random matrix $\Lambda$ whose rows are the transpose of $p$ independent random vectors $\lambda^{(i)} = (\lambda^{(i)}_1, \dots, \lambda^{(i)}_p)$, $i \in [p]$, that are identically distributed to $\lambda$, i.e., $\Lambda = ((\lambda^{(i)}_j))$. 

\begin{lemma}\label{lem:EdetM}
We have
\begin{align*}
\sfE_*[\det(A \circ \Lambda)] = \prod_{i \in [p]} \sfE_*[\lambda_i] > 0.
\end{align*}
\end{lemma}
\begin{proof}
We have $A \circ \Lambda = ((a_{ij} \lambda^{(i)}_j))$, and thus, by the definition of determinant
\begin{align*}
\det(A \circ \Lambda) = \sum_{\tau \in \cT_p} {\rm sgnt}(\tau) \; \prod_{i \in [p]} a_{i \tau(i)} \lambda^{(i)}_{\tau(i)},
\end{align*}
where ${\rm sgnt}(\tau)$ denotes the signature of a permutation $\tau \in \cT_p$.
Therefore, 
\begin{align*}
&\sfE_*[\det(A \circ \Lambda)] 
= \sum_{\tau \in \cT_p} {\rm sgnt}(\tau) \prod_{i \in [p]} a_{i \tau(i)} \sfE_*\Big[\prod_{i \in [p]} \lambda^{(i)}_{\tau(i)}\Big]\\
&= \sum_{\tau \in \cT_p} {\rm sgnt}(\tau) \prod_{i \in [p]} a_{i \tau(i)} \prod_{i \in [p]} \sfE_*\Big[\lambda^{(i)}_{\tau(i)}\Big]
= \sum_{\tau \in \cT_p} {\rm sgnt}(\tau) \prod_{i \in [p]} a_{i \tau(i)} \prod_{i \in [p]} \sfE_*\big[\lambda_{\tau(i)}\big]\\
&= \prod_{i \in [p]} \sfE_*[\lambda_i] \sum_{\tau \in \cT_p} {\rm sgnt}(\tau) \prod_{i \in [p]} a_{i \tau(i)} = \prod_{i \in [p]} \sfE_*[\lambda_i] \det (A) = \prod_{i \in [p]} \sfE_*[\lambda_i],
\end{align*} 
where the second equality follows from the independence of $\lambda^{(i)}$, $i \in [p]$, the second last one follows from the definition of determinant and the last one is due to Lemma \ref{lem:detA}. Finally, the positivity trivially follows from the definitions of $\lambda_j, j \in [p]$. 
\end{proof}

\begin{lemma}[Hadamard's Inequality \cite{hadamard1893determinants}]\label{lem:Hadamard}
If $V \in \bR^{p \times p}$ is a matrix with columns denoted by $v_i, i \in [p]$, then
\begin{align*}
|\det (V)| \leq \prod_{i \in [p]} ||v_i||.
\end{align*}
Moreover, when each column is non-zero, the equality is achieved if and only if the columns are orthogonal.
\end{lemma}

\begin{lemma}\label{lem:ineq_a_lam}
We have
\begin{align*}
\prod_{i \in [p]} \sfE_*[\lambda_i] \leq \prod_{i \in [p]} \sfE_*\big[||a_i \circ \lambda||\big],
\end{align*}
where the equality holds if and only if for every $i, j \in [p]$ either of the following conditions is satisfied:
\begin{enumerate}
\item[(1)] ${\rm supp}(a_i \circ a_j) = \emptyset$.
\item[(2)] for every $k \in {\rm supp}(a_i \circ a_j)$, $\lambda_k$ is almost surely degenerate, satisfying 
\begin{align*}
\sum_{k \in {\rm supp}(a_i \circ a_j)} a_{ik}a_{jk} \lambda^2_k \;\; \overset{\rm a.s.}{=} 0, \quad \text{i.e.,} \quad (a_i \circ \lambda)^T(a_j \circ \lambda) \overset{\rm a.s.}{=} 0,
\end{align*}
which necessarily implies that $|{\rm supp}(a_i \circ a_j)| \geq 2$.
\end{enumerate}
\end{lemma}
\begin{proof}
Clearly, the row vectors of $A \circ \Lambda$ are $(a_i \circ \lambda^{(i)})^T, i \in [p]$, as we have
\begin{align*}
A \circ \Lambda = 
\begin{bmatrix}
a_1^T\\
a_2^T\\
\vdots\\
a_p^T
\end{bmatrix}
 \circ \begin{bmatrix}
(\lambda^{(1)})^T\\
(\lambda^{(2)})^T\\
\vdots\\
(\lambda^{(p)})^T
\end{bmatrix}
=
\begin{bmatrix}
(a_1 \circ \lambda^{(1)})^T\\
(a_2 \circ \lambda^{(2)})^T\\
\vdots\\
(a_p \circ \lambda^{(p)})^T
\end{bmatrix}.
\end{align*}
Thus, by applying Lemma \ref{lem:Hadamard}, we have
\begin{align}\label{ineq:apply_Hadamard}
\det(A \circ \Lambda) \leq |\det(A \circ \Lambda)| = |\det((A \circ \Lambda)^T)| \leq \prod_{i \in [p]} ||a_i \circ \lambda^{(i)}||.
\end{align}
Now, we have
\begin{align}\notag
\prod_{i \in [p]} \sfE_*[\lambda_i] &= \sfE_*[\det(A \circ \Lambda)]\\ \label{ineq:Exp_ineq}
&\leq \sfE_*\Big[\prod_{i \in [p]} ||a_i \circ \lambda^{(i)}||\Big]\\ \notag
&= \prod_{i \in [p]} \sfE_*\big[||a_i \circ \lambda^{(i)}||\big] = \prod_{i \in [p]} \sfE_*\big[||a_i \circ \lambda||\big],
\end{align}
where the first equality follows from Lemma \ref{lem:EdetM}, the inequality \eqref{ineq:Exp_ineq} follows from \eqref{ineq:apply_Hadamard}, and the second equality follows from the independence of $\lambda^{(i)}$, $i \in [p]$. This proves the first part.

Moreover, the equality clearly holds if and only if equality holds in \eqref{ineq:Exp_ineq}. Due to \eqref{ineq:apply_Hadamard}, that in turn holds if and only if equality is achieved in \eqref{ineq:apply_Hadamard} almost surely, i.e., 
\begin{align*}
\det(A \circ \Lambda) \overset{\rm a.s.}{=} \prod_{i \in [p]} ||a_i \circ \lambda^{(i)}||.
\end{align*}
By Lemma \ref{lem:Hadamard} the above successively happens if and only if the vectors $a_i \circ \lambda^{(i)}, i \in [p]$ are orthogonal almost surely. 

Now, fix arbitrary $i, j \in [p]$, then $a_i \circ \lambda^{(i)}$ and $a_j \circ \lambda^{(j)}$ are orthogonal when
\begin{align*}
\sum_{k \in [p]} a_{ik}a_{jk} \lambda^{(i)}_k\lambda^{(j)}_k \;\; \overset{\rm a.s.}{=} 0.
\end{align*}
However, since the random variables $\lambda_k^{(i)}, \lambda_k^{(j)}, k \in [p]$ are independent and positive, the above holds if and only if either of the following two conditions is satisfied.
First is that $a_{ik}a_{jk} = 0$ for every $k \in [p]$, i.e., condition (1) holds. The second is that for every $k \in [p]$ such that $a_{ik}a_{jk} \neq 0$, i.e., $k \in {\rm supp}(a_i \circ a_j)$, both $\lambda_k^{(i)}$ and $\lambda_k^{(j)}$ are almost surely degenerate, and so is $\lambda_k$ as they are also identically distributed to $\lambda_k$, along with the relation that
\begin{align*}
\sum_{k : \; a_{ik}a_{jk} \neq 0} a_{ik}a_{jk} \lambda^2_k \;\; \overset{\rm a.s.}{=} 0,
\end{align*}
i.e., condition (2) holds. The proof is complete.
\end{proof}

Now we establish some properties of the matrix $A$ based on the following lemma.

\begin{lemma}\label{lem:Aprop}
Suppose that $M = ((m_{ij})) \in \bR^{p \times p}$ is a non-singular matrix whose $i^{\rm th}$ row is denoted by $m_i^T$, where $m_i \in \bR^p$, for every $i \in [p]$, and let $\eta = (\eta_1, \eta_2, \dots, \eta_p) \in \bR^p$ whose every element is non-zero. Moreover, we define two sets of indices, $\cR_M$ and $\cC_M$,
as follows.
\begin{align*}
\cR_M := \{i \in [p] : \supp(m_i \circ m_j) \neq \emptyset \;\; \text{for some} \;\; j \neq i\} \quad \text{and} \quad \cC_M := \bigcup_{i, j \in [p]} {\rm supp}(m_i \circ m_j).
\end{align*}
Then for every $i, j \in [p]$ either of the following conditions is satisfied:
\begin{enumerate}
\item[(1)] ${\rm supp}(m_i \circ m_j) = \emptyset$, 
\item[(2)] 
\begin{align*}
\sum_{k \in {\rm supp}(m_i \circ m_j)} m_{ik}m_{jk} \eta^2_k \;\; = 0, \quad \text{i.e.,} \quad (m_i \circ \eta)^T(m_j \circ \eta) = 0,
\end{align*}
\end{enumerate}
if and only if \; $|\cR_M| = |\cC_M|$ \; as well as both the following hold in case $\cR_M, \cC_M \neq \emptyset$,
\begin{enumerate}
\item[(a)] for every $i \in \cR_M$, 
\begin{align*}
\supp(m_i) \subseteq \cC_M \qquad \text{and} \qquad \sum_{k \in \cC_M} m_{ik}m_{jk}\eta_k^2 = 0, \;\;\; \text{for every} \;\; j \in \cR_M, j \neq i, 
\end{align*}
which necessarily implies that \; $|\cC_M| \geq |\supp(m_i)| \geq 2$ \; and \; $\cC_M = \bigcup_{i \in \cR_M} \supp(m_i)$,
\item[(b)] for every $i \notin \cR_M$, \; 
\begin{align*}
\supp(m_i) \subseteq \cC_M^c,  \;\; \text{and it is a singleton},  
\end{align*}  
which is equivalent to having \; $\cC_M^c = \bigcup_{i \notin \cR_M} \supp(m_i)$,\;\; as a disjoint union of singletons, 
\end{enumerate}
or, in other words, there exist permutation matrices $P_1$ and $P_2$ such that
\begin{align*}
P_1MP_2 \;\; = \;\;
\begin{bmatrix}
M_0 & \bold{0}\\
\bold{0} &\Delta
\end{bmatrix},
\end{align*}
with $M_0 \in \bR^{|\cR_M| \times |\cR_M|}$ corresponding to the rows and columns of $M$ with indices in $\cR_M$ and $\cC_M$, respectively,  such that the rows of $M_0 \circ \eta_0$ are orthogonal, where 
\begin{align*}
    &\eta_0^T := 
    \begin{bmatrix}
    \eta' &\eta' &\cdots &\eta'
    \end{bmatrix} \in \bR^{|\cR_M| \times |\cR_M|}, \quad \text{and}\\
    &\eta' \in \bR^{|\cC_M|} \;\; \text{is the subvector of $P_2^T\eta$ consisting of its first $|\cC_M|$ many elements},
\end{align*}
and $\Delta$ being some diagonal matrix.
\end{lemma}
\begin{proof}
We only prove here the necessity part since the sufficiency part is straightforward.
Note that every column of $M$ with index in $\cC_M$ has at least two non-zero elements, since by definition $k \in \cC_M$ if and only if there exists $i, j \in [p]$ such that $k \in \supp(m_i \circ m_j)$, i.e., $m_{ik}m_{jk} \neq 0$. Let $|\cC_M| = \ell \leq p$ and $P_2$ be some permutation matrix such that the first $\ell$ columns of $MP_2$ are the columns of $M$ with indices in $\cC_M$. Subsequently, we denote by $M' = ((m'_{ij})) \in \bR^{p \times \ell}$ and $M'' = ((m''_{ij})) \in \bR^{p \times (p - \ell)}$ the submatrices of $MP_2$ formed by its first $\ell$ columns and the rest of the columns, respectively. i.e., we write
\begin{align}\label{eq:MP2}
MP_2 = 
\begin{bmatrix}
M' & M''
\end{bmatrix},
\end{align}
where, as already indicated, $M'$ consists of the columns of $M$ with indices in $\cC_M$, and $M''$ the columns with indices not in $\cC_M$. Furthermore, for every $i \in [p]$, we denote by $m'^{T}_i$ 
the $i^{\rm th}$ row of $M'$, where $m'_i \in \bR^\ell$. 
We also denote by $\eta'^{T}$ the subvector of $\eta^T P_2$ formed by its first $\ell$ entries, i.e., $\eta'$ consists of the elements of $\eta$ with indices in $\cC_M$. As $M$ is non-singular, $M'$ is of full column rank, and so is the following matrix
\begin{align}\label{eq:M'eta'}
M' \circ
\begin{bmatrix}
\eta'^{T}\\
\eta'^T\\
\vdots\\
\eta'^T
\end{bmatrix}
=
\begin{bmatrix}
m'^T_1\\
m'^T_2\\
\vdots\\
m'^T_p
\end{bmatrix}
\circ
\begin{bmatrix}
\eta'^{T}\\
\eta'^T\\
\vdots\\
\eta'^T
\end{bmatrix}
=
\begin{bmatrix}
(m'_1 \circ \eta')^T\\
(m'_2 \circ \eta')^T\\
\vdots\\
(m'_p \circ \eta')^T
\end{bmatrix}
=
\begin{bmatrix}
m'_{11}\eta'_1 \quad & m'_{12}\eta'_2 \quad &\cdots \quad &m'_{1\ell}\eta'_\ell\\
m'_{21}\eta'_1 \quad & m'_{22}\eta'_2 \quad &\cdots \quad&m'_{2\ell}\eta'_\ell\\
\vdots \quad &\vdots \quad &\ddots \quad &\vdots\\
m'_{p1}\eta'_1 \quad & m'_{p2}\eta'_2 \quad &\cdots \quad &m'_{p\ell}\eta'_\ell
\end{bmatrix}.
\end{align}
Now it is important to note that, since for every $i, j \in [p]$, \; ${\rm supp}(m_i \circ m_j) \subseteq \cC_M$, we have, due to condition (2), 
\begin{align*}
&\sum_{k \in \cC_M} m_{ik}m_{jk}\eta_k^2 = 0, \qquad \text{which implies that}\\ &\sum_{k = 1}^\ell m'_{ik}m'_{jk}\eta'^2_k = 0, \quad \text{i.e.,} \quad (m'_i \circ \eta')^T(m'_j \circ \eta') = 0.
\end{align*}
Therefore, the rows of the matrix in \eqref{eq:M'eta'} that is of rank $\ell$ are orthogonal. This immediately implies that this matrix has exactly $\ell$ many non-zero rows, and since $\eta'_k \neq 0$ for every $k \in [\ell]$, this fact is also true for $M'$. Furthermore, $\cR_M$ is in fact the set of indices of the non-zero rows of $M'$. Indeed, this directly follows from the definitions of $\cR_M$ and $\cC_M$ that any $k \in \cC_M$ if and only if there exists $j \neq i$ such that $m_{ik}, m_{jk} \neq 0$, which holds if and only if $i \in \cR_M$.

Suppose $P'_1$ be some permutation matrix such that the first $\ell$ rows of $P'_1 M'$ are the $\ell$ non-zero rows of $M'$ that also have indices in $\cR_M$. Subsequently, we denote by $M_0 \in \bR^{\ell \times \ell}$ the submatrix of $P'_1 M'$ formed by its first $\ell$ rows, i.e.,
\begin{align}\label{eq:MP'}
P'_1 M' =
\begin{bmatrix}
M_0\\
\bold{0}
\end{bmatrix}, \qquad \text{and let} \qquad 
P'_1 M'' =
\begin{bmatrix}
M_1\\
M_2
\end{bmatrix},
\end{align} 
where $M_1 \in \bR^{\ell \times (p - \ell)}$ and $M_2 \in \bR^{(p - \ell) \times (p - \ell)}$. 
Thus, following \eqref{eq:MP2} and \eqref{eq:MP'}, we have
\begin{align}\label{eq:P1MP2}
P'_1 M P_2 = 
\begin{bmatrix}
P'_1 M' \;\; & P'_1 M''
\end{bmatrix}
= 
\begin{bmatrix}
M_0 \quad & M_1\\
\bold{0} \quad & M_2
\end{bmatrix}.
\end{align}
Now note that, by the definition of $\cC_M$ every column of $M$ with index not in $\cC_M$ has at most one non-zero element, but on the other hand, due to non-singularity of $M$, it must have at least one non-zero element. Therefore, every column in $M''$ has exactly one non-zero element, which further implies that the total number of non-zero elements in $M''$, and hence in $P'_1M''$ is exactly $(p - \ell)$. Again, due to non-singularity of $M$, the matrix $P'_1MP_2$ is also non-singular, and thus, from the representation in \eqref{eq:P1MP2} each of the $(p - \ell)$ many rows of $M_2$ must have at least one non-zero element.
This is only possible when we have $M_1 = \bold{0}$, and both every row and every column of $M_2$ has exactly one non-zero element, i.e., $M_2 = \Delta P_0$, for some diagonal matrix $\Delta$ and some permutation matrix $P_0$. Finally, pre-multiplying both sides in \eqref{eq:P1MP2} by another permutation matrix $P''_1$ such that the rows of $M_2$ is further arranged to form $\Delta$, and letting $P_1 = P''_1P'_1$ completes the proof.
\end{proof}

\begin{lemma}[Darmois-Skitovic Theorem \cite{darmois1953analyse, skitovitch1953property}]\label{lem:DarSki}
Consider the random vector $Z = (Z_1, \dots, Z_p)$, where $Z_i, \ i \in [p]$ are independent, and let $u, v \in \bR^p$. Then
\begin{align*}
u^TZ \;\; \text{and} \;\; v^TZ \;\; \text{are independent only if} \;\; \text{for every} \;\; k \in {\rm supp}(u \circ v), \;\; Z_k \;\; \text{is Gaussian}. 
\end{align*}
\end{lemma}

\begin{lemma}\label{lem:exp_mod}
For any $u \in \bR^p$, we have
\begin{align*}
\sfE_*\lt[|u^T\epsilon|\rt] = \sqrt{\frac{2}{\pi}} \; \sfE_*\lt[||u \circ \lambda||\rt].
\end{align*}
\end{lemma}
\begin{proof}
Let $u = (u_1, u_2, \dots, u_p)$. Then, $u^T\epsilon | \lambda \; \sim \; N\lt(0, \sum_{i = 1}^p u_i^2 \lambda_i^2\rt)$. Thus,
\begin{align*}
\sfE_*\lt[|u^T\epsilon|\rt] &= \sfE_*\lt[\sfE_*\lt[|u^T\epsilon| | \lambda\rt]\rt]\\
&= \sqrt{\frac{2}{\pi}} \; \sfE_*\lt[\lt(\sum_{i = 1}^p u_i^2 \lambda_i^2\rt)^{1/2}\rt] = \sqrt{\frac{2}{\pi}} \; \sfE_*\lt[||u \circ \lambda||\rt].
\end{align*}
\end{proof}

\begin{lemma}\label{lem:ineq_eps_e}
We have
\begin{align*}
\prod_{i \in [p]} \sfE_*[|\epsilon_i|] \; \leq \; \prod_{i \in [p]}\sfE_*[|e_i|], 
\end{align*}
where the equality holds if and only if \ $e_i, i \in [p]$ \ are pairwise independent, which in turn is true if and only if
\; $|\cR_A| = |\cC_A|$, and the following conditions hold:
\begin{enumerate}
\item[(i)] in case $\cC_A \neq \emptyset$, \; $|\cC_A| \geq 2$, and for every $k \in \cC_A$, \; $\epsilon_k \sim N(0, \lambda_k^2)$,
\item[(ii)] for every $i \in \cR_A$, \; $\supp(a_i) \subseteq \cC_A$, \; 
for which
\begin{align*}
e_{\sigma^{-1}(i)} = \sum_{k \in \cC_A} a_{ik} \epsilon_{k} \qquad \text{and} \qquad
\sum_{k \in \cC_A} a_{ik}a_{jk}\lambda_k^2 = 0, \;\;\; \text{for every} \;\; j \in \cR_A, j \neq i,
\end{align*}
\item[(iii)] there exists a permutation $\kappa \in \cT_p$ \, such that for every $i \notin \cR_A$, $\kappa(i) \notin \cC_A$ and \, $\supp(a_i) = \{\kappa(i)\}$ \; for which \;
$
e_{\sigma^{-1}(i)} = a_{i\kappa(i)} \epsilon_{\kappa(i)}. 
$
\end{enumerate}
\end{lemma}
\begin{proof}
Following \eqref{eq:defA}, for every $i \in [p]$, \; $e_{\sigma^{-1}(i)} = a_i^T\epsilon$, and thus, using Lemma \ref{lem:exp_mod}
\begin{align*}
&\prod_{i = 1}^p \sfE_*[|\epsilon_i|] = \lt(\sqrt{\frac{2}{\pi}}\rt)^p \; \prod_{i = 1}^p\sfE_*\lt[\lambda_i\rt] \qquad \text{and}\\ 
&\prod_{i = 1}^p \sfE_*[|e_i|] = \prod_{i = 1}^p \sfE_*[|e_{\sigma^{-1}(i)}|] = \lt(\sqrt{\frac{2}{\pi}}\rt)^p \; \sfE_*\lt[||a_i \circ \lambda||\rt].
\end{align*}
Therefore, following the above, it suffices to show that
\begin{align*}
\prod_{i \in [p]} \sfE_*[\lambda_i] \leq \prod_{i \in [p]} \sfE_*\big[||a_i \circ \lambda||\big].
\end{align*}
Indeed, the above is true due to Lemma \ref{lem:ineq_a_lam}, and the equality holds if and only if conditions (1) and (2) in Lemma \ref{lem:ineq_a_lam} hold, which in turn prove the equivalence between independence of $e_i, i \in [p]$ and conditions (i)-(iii), as shown below.

First, according to condition (1) and the definition of $\cR_A$, for every $i, j \in [p]$ such that $\supp(a_i \circ a_j) \neq \emptyset$, i.e., $i, j \in \cR_A \neq \emptyset$, $\lambda_k$ is almost surely degenerate for every $k \in {\rm supp}(a_i \circ a_j)$, i.e., $\epsilon_k \sim N(0, \lambda_k^2)$. This is clearly equivalent to condition (i) by the definition of $\cC_A$. 
Again, by Lemma \ref{lem:Aprop} these two conditions in Lemma  \ref{lem:ineq_a_lam} hold if and only if $A$ satisfies conditions (a) and (b) in Lemma \ref{lem:Aprop}, which are further equivalent to having conditions (ii) and (iii) due to the representation that \ $e_{\sigma^{-1}(i)} = a_i^T\epsilon$ for every $i \in [p]$. For every $i, j \in \cR_A$, $e_{\sigma^{-1}(i)}$ and $e_{\sigma^{-1}(i)}$ are independent since by conditions (i) and (ii) both follow Gaussian distribution with their covariance being $0$. Moreover, due to condition (iii) for every $i \notin \cR_A$, $e_{\sigma^{-1}(i)}$ is independent of 
$e_{\sigma^{-1}(j)}$ for every $j \neq i$.

Finally, when the variables $e_i, i \in [p]$, are pairwise independent, consider the pair $e_{\sigma^{-1}(i)} = a_i^T\epsilon$ and $e_{\sigma^{-1}(j)} = a_j^T\epsilon$. They are independent only if either $\supp(a_i \circ a_j)  = \emptyset$, or by the Darmois-Skitovic Theorem, stated in Lemma \ref{lem:DarSki}, for every $k \in \supp(a_i \circ a_j)$, $\epsilon_k$ is Gaussian, i.e., following $N(0, \lambda_k^2)$, along with their covariance necessarily being zero, i.e., $\sum_{k \in {\rm supp}(a_i \circ a_j)} a_{ik}a_{jk} \lambda^2_k \;\; = 0$. These are precisely conditions (1) and (2) in Lemma 
\ref{lem:ineq_a_lam}, and this completes the proof.
\end{proof}

\begin{lemma}\label{lem:sig_inv}
For every $i \in [p]$, we have
\begin{align*}
\sigma^{-1}(i) \in \{j \in \bar{\an}^*(\sigma^{-1}(i)): \beta^*_{\sigma^{-1}(i) \gets j} \neq 0\} \; \subseteq \; \bigcup_{j \in [i]} \supp(a_j).
\end{align*}
\end{lemma}
\begin{proof}
Fix any $i \in [p]$. Then following the representation in \eqref{eq:X_Beps} we have
\begin{align}\label{eq:bet*repr}
X_{\sigma^{-1}(i)} = \sum_{k \in \bar{\an}^*(\sigma^{-1}(i))} \beta^*_{\sigma^{-1}(i) \gets k} \epsilon_k,
\end{align}
where $ \beta^*_{\sigma^{-1}(i) \gets \sigma^{-1}(i)} = 1$. Now, 
from the representation in \eqref{eq:X_Be} we again have
\begin{align} \notag
X_{\sigma^{-1}(i)} &= \sum_{\ell \in \bar{\an}(\sigma^{-1}(i))} b_{\sigma^{-1}(i) \gets \ell} e_\ell\\ \notag
&= \sum_{\sigma^{-1}(j) \in \bar{\an}(\sigma^{-1}(i))} b_{\sigma^{-1}(i) \gets \sigma^{-1}(j)} e_{\sigma^{-1}(j)}\\ \notag
&= \sum_{\sigma^{-1}(j) \in \bar{\an}(\sigma^{-1}(i))} b_{\sigma^{-1}(i) \gets \sigma^{-1}(j)} a_j^T\epsilon\\ \notag
&= \sum_{\sigma^{-1}(j) \in \bar{\an}(\sigma^{-1}(i))} b_{\sigma^{-1}(i) \gets \sigma^{-1}(j)} \sum_{k \in \supp(a_j)} a_{jk}\epsilon_k\\ \label{eq:bet_repr}
&= \sum_{\sigma^{-1}(j) \in \bar{\an}(\sigma^{-1}(i))} \sum_{k \in \supp(a_j)} b_{\sigma^{-1}(i) \gets \sigma^{-1}(j)} a_{jk} \epsilon_k,
\end{align}
where the second equality follows only by replacing $\ell$ with $\sigma^{-1}(j)$, the third one is due to $e_{\sigma^{-1}(j)} = a_j^T\epsilon$ that follows from \eqref{eq:defA}. Here, note that $\sigma^{-1}(j) \in \bar{\an}(\sigma^{-1}(i))$ only if $j \leq i$, or equivalently, $j \in [i]$. Therefore, if $k \notin \bigcup_{j \in [i]} \supp(a_j)$ then the coefficient of $\epsilon_k$ in \eqref{eq:bet_repr} is $0$, and thus, comparing with \eqref{eq:bet*repr}, either $k \notin \bar{\an}^*(\sigma^{-1}(i))$ or $\beta^*_{\sigma^{-1}(i) \gets k} = 0$. This completes the proof.
\end{proof}

In condition (iii) of Lemma \ref{lem:ineq_eps_e}, we derive that under equality there exists a one-one mapping from the elements of $\cR_A^c$ to that of $\cC_A^c$ via some permutation $\kappa \in \cT_p$, i.e., for every $i \in \cR_A^c$ we have $\kappa(i) \in \cC_A^c$, satisfying the relation that
\begin{align*}
e_{\sigma^{-1}(i)} = a_{i\kappa(i)} \epsilon_{\kappa(i)}.
\end{align*}
In the following lemma, we further establish a necessary and sufficient condition for the conditions in Lemma \ref{lem:ineq_eps_e} in case $\cR_A = \cC_A = \emptyset$.

\begin{lemma}\label{lem:RCempty_ki}
We have $\cR_A = \cC_A = \emptyset$, and the conditions in Lemma \ref{lem:ineq_eps_e} hold if and only if $A = P$ for some permutation matrix $P$, i.e., for every $i \in [p]$,
\begin{align*}
a_{i\kappa(i)} = 1, \qquad \text{and \ also,} \qquad  \kappa(i) = \sigma^{-1}(i).
\end{align*}
\end{lemma}
\begin{proof}
When $A = P$ for some permutation matrix $P$, $\cR_A = \cC_A = \emptyset$ and condition (iii) in Lemma \ref{lem:ineq_eps_e} is clearly satisfied, which establishes the sufficiency part. Now we prove the necessity part by induction over $i \in [p]$. Note that, by Lemma \ref{lem:sig_inv}, we have
\begin{align*}
\sigma^{-1}(1) \in \supp(a_1) = \{\kappa(1)\}, 
\end{align*}
and thus, $\sigma^{-1}(1) = \kappa(1)$. Now, according to the representation in \eqref{eq:X_Be}, we have $X_{\sigma^{-1}(1)} =  X_{\kappa(1)} = a_{1\kappa(1)} \epsilon_{\kappa(1)}$, whereas according to \eqref{eq:X_Beps} the coefficient of $\epsilon_{\kappa(1)}$ in the expression of $X_{\kappa(1)}$ is $1$. Therefore,  $a_{1\kappa(1)} = 1$, and hence, the induction hypotheses is true for $i = 1$.
Next, fix $j \in [p-1]$ and suppose that the hypotheses is true for every $i \in [j]$. Then for $i = j+1$, again by Lemma \ref{lem:sig_inv} we have 
\begin{align*}
\sigma^{-1}(j + 1) \in \bigcup_{i \in [j+1]} \supp(a_i) = \bigcup_{i \in [j+1]} \{\kappa(i)\} = \{\kappa (i) : i \in [j+1]\}.
\end{align*}
Since $\sigma^{-1}(j + 1) \neq \sigma^{-1}(i) = \kappa(i)$ for every $i \in [j]$ by the induction hypotheses, for the above to hold, we must have $\sigma^{-1}(j+1) = \kappa(j + 1)$. Thus, the coefficient of $\epsilon_{\kappa(j+1)}$ in the expression of $X_{\kappa(j+1)}$ is $a_{j+1 \kappa(j + 1)}$, whereas according to \eqref{eq:X_Beps} it is $1$, and this completes the proof.
\end{proof}

\begin{lemma}\label{lem:uniq_rep}
Consider two sets of nodes $\cI, \cJ \subseteq [p]$, and the non-zero real numbers $\alpha_k, k \in \cI$ and $\alpha'_k, k \in \cJ$. Then
\begin{align}\label{eq:uniq_rep1}
\sum_{k \in \cI} \alpha_k X_{k} = \sum_{k \in \cJ} \alpha'_k X_{k} 
\end{align}
if and only if \; $\cI = \cJ$ \; and \; $\alpha_k = \alpha'_k$ \; for every $k \in \cI$.
\end{lemma}
\begin{proof}
The sufficiency part trivially holds, and we prove the necessity part as follows.

Let $\ell := \max \; (\cI \cup \cJ)$ and without loss of generality, suppose that $\ell \in \cJ$. Then we have
\begin{align*}
\sum_{k \in \cI} \alpha_k X_{k} &= \sum_{k \in \cJ, k \neq \ell} \alpha'_k X_{k} + \alpha'_{\ell} X_\ell\\
&= \sum_{k \in \cJ, k \neq \ell} \alpha'_k X_{k} + \alpha'_\ell \sum_{k \in \bar{\an}^*(\ell)} \beta^*_{\ell \gets k} \epsilon_k\\
&= \sum_{k \in \cJ, k \neq \ell} \alpha'_k X_{k} + \alpha'_\ell \sum_{k \in {\an}^*(\ell)} \beta^*_{\ell \gets k} \epsilon_k + \alpha'_\ell \epsilon_\ell,
\end{align*}
where the first equality follows from \eqref{eq:uniq_rep1}, and the second one is due to the representation in \eqref{eq:X_Beps}. Now by the definition of $\ell$, for every $k \in \cI$, $\ell \notin \bar{\an}^*(k)$, which implies that in order for $\epsilon_\ell$ to appear on the right hand side we must have $\ell \in \cI$, and furthermore, $\alpha_\ell = \alpha'_\ell$. This immediately leaves us with
\begin{align*}
\sum_{k \in \cI, k \neq \ell} \alpha_k X_k = \sum_{k \in \cJ, k \neq \ell} \alpha'_k X_k.
\end{align*}
Without loss of generality, suppose that $|\cI| \geq |\cJ|$. Then repeating the above argument, we have $\cJ \subseteq \cI$, and $\alpha_k = \alpha'_k$ \; for every $k \in \cJ$, which further leaves us with  
\begin{align*}
\sum_{k \in \cI \setminus \cJ} \alpha_k X_k = 0.
\end{align*}
Next, if $\cI \setminus \cJ \neq \emptyset$ and we let $m := \max (\cI \setminus \cJ)$, then using the above and the representation in \eqref{eq:X_Beps} we have, similarly as before,
\begin{align*}
\sum_{k \in \cI \setminus \cJ} \alpha_k X_{k} = \sum_{k \in \cI \setminus \cJ, k \neq m} \alpha_k X_{k} + \alpha_m \sum_{k \in {\an}^*(m)} \beta^*_{\ell \gets k} \epsilon_k + \alpha_m \epsilon_m = 0.
\end{align*}
Now by the definition of $m$, for every $k \in \cI \setminus \cJ$, $k \in \bar{\an}^*(m)$, which immediately implies that in order for the above to hold we must have $\alpha_m = 0$, which is a contradiction. Thus, $\cI \setminus \cJ = \emptyset$, and the proof is complete.
\end{proof}

\begin{lemma}\label{lem:R=C=emp}
We have $\cR_A = \cC_A = \emptyset$, and the conditions in Lemma \ref{lem:ineq_eps_e} hold if and only if $\gamma = \gamma^*$, and also $b_{ij} = \beta^*_{ij}$ for every $j \in \pa(i)$.
\end{lemma}
\begin{proof}
In view of Lemma \ref{lem:RCempty_ki}, 
it suffices to show that 
\begin{align*}
&e_{\sigma^{-1}(i)} = \epsilon_{\sigma^{-1}(i)} \qquad \text{for every} \;\; i \in [p]\\
\text{if and only if} \qquad &\pa(i) = \pa^*(i), \;\;\; b_{ij} = \beta^*_{ij} \qquad \text{for every} \;\; i \in [p], \; j \in \pa(i).
\end{align*}
From \eqref{eq:model} and \eqref{eq:lap_model_alt}, we have for every $i \in [p]$,
\begin{align}\label{eq:equiv_rep}
X_{\sigma^{-1}(i)} = \sum_{k \in {\rm pa}^*(\sigma^{-1}(i))} \beta^*_{\sigma^{-1}(i)k} X_{k} + \epsilon_{\sigma^{-1}(i)} = \sum_{k \in {\rm pa}(\sigma^{-1}(i))} b_{\sigma^{-1}(i)k} X_{k} + e_{\sigma^{-1}(i)}.
\end{align}
Thus, we have, for every $i \in [p]$,
\begin{align*}
&e_{\sigma^{-1}(i)} = \epsilon_{\sigma^{-1}(i)}\\
\text{if and only if} \qquad &\sum_{k \in {\rm pa}^*(\sigma^{-1}(i))} \beta^*_{\sigma^{-1}(i)k} X_{k} = \sum_{k \in {\rm pa}(\sigma^{-1}(i))} b_{\sigma^{-1}(i)k} X_{k}\\
\text{if and only if} \qquad &\pa^*(\sigma^{-1}(i)) = \pa(\sigma^{-1}(i)), \;\;\; \beta^*_{\sigma^{-1}(i)k} = b_{\sigma^{-1}(i)k} \qquad \forall \;\; k \in  \pa(\sigma^{-1}(i)),
\end{align*}
where the second equivalence follows from Lemma \ref{lem:uniq_rep}. This completes the proof.
\end{proof}

\begin{remark}\label{lem:par_faith}
Note that, to establish the above result we do not need faithfulness of $\sfP_X^*$ as defined in Definition \ref{def:faith}. Therefore, when there is at most one Gaussian error, clearly by Lemma \ref{lem:ineq_eps_e}, $\cR_A = \cC_A = \emptyset$, and thus the faithfulness assumption is not needed.
However, when it is not the case, we assume faithfulness to establish further results. To begin with, we show next in Lemma \ref{lem:ancest_infl}, under the assumption of faithfulness, for any node the total causal effect of any of its ancestors does not vanish, i.e., for every $\ell \in [p]$, and $k \in  \an^*(\ell)$, we have $\beta^*_{\ell \gets k} \neq 0$. As a consequence, the first set in Lemma \ref{lem:sig_inv} coincides with $\bar{\an}^*(\sigma^{-1}(i))$.
\end{remark}

\begin{lemma}\label{lem:ancest_infl}
Suppose that $\sfP_X^*$ is faithful to $\gamma^*$. Then for every $\ell \in [p]$, and $k \in  \an^*(\ell)$, we have $\beta^*_{\ell \gets k} \neq 0$.
\end{lemma}
\begin{proof}
Fix  $\ell \in [p]$ and $k \in \an^*(\ell)$. Then following \eqref{eq:defb*}, we have 
 \begin{align*}
\beta^*_{\ell \gets k} = \sum_{j \in \pa^*(\ell) \cap \bar{\de}^*(k)} \beta^*_{\ell j} \beta^*_{j \gets k}.
\end{align*}
Therefore, in case \ $\pa(\ell) \cap \bar{\de}(k) = \{k\}$, clearly $\beta^*_{\ell \gets k} = \beta^*_{\ell k} \neq 0$, as per the definition of parent.

For the other case, there exists $\ell > j > k$, such that $j \in \pa(\ell) \cap \bar{\de}(k)$. This implies there exists an unblocked path from $k$ to $\ell$, in which $j$ is a non-collider. 
Since $j \notin \an^*(k)$, this immediately implies that $k$ and $\ell$ are $d$-connected, and furthermore, when $\an^*(k) \neq \emptyset$, they are not even $d$-separated by $\an^*(k)$. Therefore, we have  
\begin{align}\label{eq:d-sep}
\text{according to} \;\; \gamma^*, \quad k \nindep \ell, \quad \text{and when} \;\; \an^*(k) \neq \emptyset, \quad k \nindep \ell \ | \ \an^*(k).
\end{align}
Now, suppose that $\beta^*_{\ell \gets k} = 0$. Then, following the representation in \eqref{eq:X_Beps}, we have
\begin{align*}
X_{k} = \sum_{j \in {\an}^*(k)} \beta^*_{k \gets j} \epsilon_j + \epsilon_k, \qquad \text{and} \qquad X_{\ell} = \sum_{j \in {\an}^*(\ell)\setminus \{k\}} \beta^*_{\ell \gets j} \epsilon_j + \epsilon_\ell.
\end{align*}
Since $\an^*(k) \subseteq \an^*(\ell)$, the above implies that 
\begin{align}\label{eq:d-sep_viol}
\text{under} \;\; \sfP_X^*, \quad \text{when} \;\; \an^*(k) = \emptyset, \quad k \indep \ell \quad \text{and when} \;\; \an^*(k) \neq \emptyset, \quad k \indep \ell \ | \ \an^*(k).
\end{align}
Thus, comparing \eqref{eq:d-sep} and \eqref{eq:d-sep_viol}, clearly $\bI(\sfP_X^*) \nsubseteq \bI(\gamma^*)$, which violates the faithfulness of $\sfP_X^*$ to $\gamma^*$. Therefore we must have $\beta^*_{\ell \gets k} \neq 0$, and the proof is complete.
\end{proof}

\begin{lemma}\label{lem:non-G_err_nodes}
Suppose that $\sfP_X^*$ is faithful to $\gamma^*$, and $\cR_A, \cC_A \neq \emptyset$. Then the conditions in Lemma \ref{lem:ineq_eps_e} hold only if for every $i \in \cR_A^c$,
\begin{align*}
a_{i\kappa(i)} = 1 \qquad \text{and} \qquad \kappa(i) = \sigma^{-1}(i).
\end{align*}
\end{lemma}
\begin{proof}
Fix any $i \in \cR_A^c$. Continuing with the representation in \eqref{eq:bet_repr} we have
\begin{align}\notag
X_{\sigma^{-1}(i)} &= \sum_{\sigma^{-1}(j) \in \bar{\an}(\sigma^{-1}(i))} \sum_{k \in \supp(a_j)} b_{\sigma^{-1}(i) \gets \sigma^{-1}(j)} a_{jk} \epsilon_k\\ \label{eq:bet_repr_sub}
&= \sum_{\sigma^{-1}(j) \in {\an}(\sigma^{-1}(i))} \sum_{k \in \supp(a_j)} b_{\sigma^{-1}(i) \gets \sigma^{-1}(j)} a_{jk} \epsilon_k + \sum_{k \in \supp(a_i)} a_{ik}\epsilon_k\\ \label{eq:bet_repr_sub2}
&= \sum_{\sigma^{-1}(j) \in {\an}(\sigma^{-1}(i))} \sum_{k \in \supp(a_j)} b_{\sigma^{-1}(i) \gets \sigma^{-1}(j)} a_{jk} \epsilon_k + a_{i\kappa(i)}\epsilon_{\kappa(i)},
\end{align}
where the third equality follows from condition (iii) in Lemma \ref{lem:ineq_eps_e} that $\supp(a_i) = \{\kappa(i)\}$. Since $a_{i\kappa(i)} \neq 0$ and $\kappa(i) \notin \supp(a_j)$ for every $j \neq i$, by comparing with \eqref{eq:bet*repr} we must have $\kappa(i) \in \bar{\an}^*(\sigma^{-1}(i))$.

Now, by Lemma \ref{lem:sig_inv}, we have
\begin{align}\label{eq:supp_cup_k}
\begin{split}
\sigma^{-1}(i) \in \bigcup_{j \in [i]} \supp(a_j) &= \Big(\bigcup_{j \in [i-1]} \supp(a_j)\Big) \cup \supp(a_i)\\
&= \Big(\bigcup_{j \in [i-1]} \supp(a_j)\Big) \cup \{\kappa(i)\}.
\end{split}
\end{align}
In addition, we claim that
\begin{align*}
\sigma^{-1}(i) \notin \bigcup_{j \in [i-1]} \supp(a_j), \quad \text{i.e.,} \quad \sigma^{-1}(i) \notin \supp(a_j) \quad \text{for every} \;\; j \in [i-1],
\end{align*}
and prove this claim by induction over $j \in [i-1]$. Note that, following the representation in \eqref{eq:bet_repr} we have
\begin{align*}
X_{\sigma^{-1}(1)} = \sum_{k \in \supp(a_1)} a_{1k} \epsilon_k.
\end{align*}
If $\sigma^{-1}(i) \in \supp(a_1)$, then comparing the above with \eqref{eq:bet*repr}, $\sigma^{-1}(i) \in \an^*(\sigma^{-1}(1))$, and since $\kappa(i) \in \bar{\an}^*(\sigma^{-1}(i))$, we must have $\kappa(i) \in \an^*(\sigma^{-1}(1))$. Therefore, due to Lemma \ref{lem:ancest_infl} and Lemma \ref{lem:sig_inv} we must have $\kappa(i) \in \supp(a_1)$, which is a contradiction again by the fact that $\kappa(i) \notin \supp(a_j)$ for every $j \neq i$. Thus, $\sigma^{-1}(i) \notin \supp(a_1)$, and the claim is true for $j = 1$. Next, fix $\ell \in [i-1]$ and suppose that the hypotheses is true for every $j \in [\ell-1]$. Then for $j = \ell$, following \eqref{eq:bet_repr_sub}, we have
\begin{align*}
X_{\sigma^{-1}(\ell)} = \sum_{\sigma^{-1}(j) \in {\an}(\sigma^{-1}(\ell))} \sum_{k \in \supp(a_j)} b_{\sigma^{-1}(\ell) \gets \sigma^{-1}(j)} a_{jk} \epsilon_k \;\; + \sum_{k \in \supp(a_\ell)} a_{\ell k}\epsilon_k
\end{align*}
Since $\sigma^{-1}(i) \notin \supp(a_j)$ for every $j \leq \ell$, if $\sigma^{-1}(i) \in \supp(a_\ell)$, then due to \eqref{eq:bet*repr}, $\sigma^{-1}(i) \in \an^*(\sigma^{-1}(\ell))$, and as $\kappa(i) \in \bar{\an}^*(\sigma^{-1}(i))$, we must have $\kappa(i) \in \an^*(\sigma^{-1}(\ell))$. Therefore, again due to Lemma \ref{lem:ancest_infl} and Lemma \ref{lem:sig_inv}
we must also have $\kappa(i) \in \cup_{j \leq \ell} \supp(a_j)$, which is again a contradiction. Thus, $\sigma^{-1}(i) \notin \supp(a_\ell)$, and the claim is true for every $j \in [i-1]$, and due to \eqref{eq:supp_cup_k}, this immediately implies that $\sigma^{-1}(i) = \kappa(i)$. Subsequently, from the representaion in \eqref{eq:bet_repr_sub2}, we have $a_{i \kappa(i)} = 1$. The proof is complete.
\end{proof}

\begin{lemma}\label{lem:R,C=non_emp}
Suppose that $\sfP_X^*$ is faithful to $\gamma^*$. Then $\cR_A, \cC_A \neq \emptyset$, and the conditions in Lemma \ref{lem:ineq_eps_e} hold if and only if there exist at least two Gaussian errors, i.e., $|n\cG^*| \leq (p-2)$,  and $\gamma, b$ satisfy the following conditions:
\begin{enumerate}
\item[(a)] for every $i \notin \cR_A$,
\begin{enumerate}
\item[(i)] $\sigma^{-1}(i) \notin \cC_A$,
\item[(ii)]
$\pa(\sigma^{-1}(i)) = \pa^*(\sigma^{-1}(i))$, and also $b_{\sigma^{-1}(i)j} = \beta^*_{\sigma^{-1}(i)j}$ for every $j \in \pa(\sigma^{-1}(i))$.
\end{enumerate}
\item[(b)] for every $i \in \cR_A$, 
\begin{enumerate}
\item[(i)] $\sigma^{-1}(i) \in \cC_A$,
\item[(ii)] $\pa(\sigma^{-1}(i))$ and $b_{\sigma^{-1}(i)j}, j \in \pa(\sigma^{-1}(i))$ are such that $e_{\sigma^{-1}(i)}$ is some linear combination of the Gaussian errors $\epsilon_k, k \in \cC_A$, and 
$e_{\sigma^{-1}(i)}, i \in \cR_A$ are pairwise independent.
\end{enumerate}
\end{enumerate}
\end{lemma}
\begin{proof}
First we prove the necessity part. Since for every $k \in \cC_A$, $\epsilon_k$ is Gaussian, we have $n\cG^* \subseteq \cC_A^c$, and thus, $|\cC_A| \geq 2$ immediately implies $|n\cG^*| \leq (p-2)$. Moreover, due to Lemma \ref{lem:non-G_err_nodes} and condition (iii) in Lemma \ref{lem:ineq_eps_e}, $\sigma^{-1}(i) = \kappa(i) \notin \cC_A$, which proves condition (a)(i), and in addition, we have $a_{i\kappa(i)} = 1$. Thus, it suffices to show that
\begin{align*}
&e_{\sigma^{-1}(i)} = \epsilon_{\sigma^{-1}(i)} \qquad \text{for every} \;\; i \in \cR_A^c\\
\text{only if} \qquad &\pa(\sigma^{-1}(i)) = \pa^*(\sigma^{-1}(i)), \;\;\; b_{\sigma^{-1}(i)j} = \beta^*_{\sigma^{-1}(i)j} \qquad \forall \;\; i \in \cR_A^c, \; j \in \pa(\sigma^{-1}(i)).
\end{align*}
From \eqref{eq:model} and \eqref{eq:lap_model_alt}, we have for every $i \in [p]$,
\begin{align}\label{eq:equiv_rep}
X_{\sigma^{-1}(i)} = \sum_{k \in {\rm pa}^*(\sigma^{-1}(i))} \beta^*_{\sigma^{-1}(i)k} X_{k} + \epsilon_{\sigma^{-1}(i)} = \sum_{k \in {\rm pa}(\sigma^{-1}(i))} b_{\sigma^{-1}(i)k} X_{k} + e_{\sigma^{-1}(i)}.
\end{align}
Thus, we have, for every $i \in \cR_A^c$,
\begin{align}\label{iff:nonG}
\begin{split}
&e_{\sigma^{-1}(i)} = \epsilon_{\sigma^{-1}(i)}\\
\text{iff} \qquad &\sum_{k \in {\rm pa}^*(\sigma^{-1}(i))} \beta^*_{\sigma^{-1}(i)k} X_{k} = \sum_{k \in {\rm pa}(\sigma^{-1}(i))} b_{\sigma^{-1}(i)k} X_{k}\\
\text{iff} \qquad &\pa^*(\sigma^{-1}(i)) = \pa(\sigma^{-1}(i)), \;\;\; \beta^*_{\sigma^{-1}(i)k} = b_{\sigma^{-1}(i)k} \qquad \forall \;\; k \in  \pa(\sigma^{-1}(i)),
\end{split}
\end{align}
where the first equivalence follows from \eqref{eq:equiv_rep} and the second one is due to Lemma \ref{lem:uniq_rep}. This proves condition (a)(ii).

Next, condition (b)(i) follows directly from (a)(i) and the facts that $\sigma$ is one-to-one and $|\cR_A| = |\cC_A|$. Furthermore, condition (i) and the first part of condition (ii) in Lemma \ref{lem:ineq_eps_e} directly suggest the existence of $\pa(\sigma^{-1}(i))$ and $\{b_{\sigma^{-1}(i)j}: j \in \pa(\sigma^{-1}(i))\}$ satisfying the relation that
\begin{align*}
e_{\sigma^{-1}(i)} = X_{\sigma^{-1}(i)} - \sum_{k \in {\rm pa}(\sigma^{-1}(i))} b_{\sigma^{-1}(i)k} X_{k} = \sum_{k \in \cC_A} a_{ik} \epsilon_{k},
\end{align*}
which proves the first part of condition (b)(ii). The second part, i.e., the pairwise independence in condition (b)(ii) immediately follows from the second part of condition (ii) in Lemma \ref{lem:ineq_eps_e}.

Finally, we prove the sufficiency part. Due to condition (a)(ii) and \eqref{iff:nonG}, we have for every $i \notin \cR_A$, $e_{\sigma^{-1}(i)} = \epsilon_{\sigma^{-1}(i)}$, also with $\sigma^{-1}(i) \notin \cC_A$ by condition (a)(i). Moreover, following condition (b) clearly $e_{\sigma^{-1}(i)}, i \in \cR_A$ are pairwise independent, and furthermore, as they are function of the errors $\epsilon_k, k \in \cC_A$, they are also independent of $e_{\sigma^{-1}(i)}, i \notin \cR_A$. Therefore, $e_i, i \in [p]$ are pairwise independent, which by Lemma \ref{lem:ineq_eps_e} implies that the conditions in Lemma \ref{lem:ineq_eps_e} hold. The proof is complete.
 \end{proof}

\section{Proofs regarding identifiability}\label{app:identif}

\subsection{Some important lemmas}

In this subsection we establish some lemmas which are critical in establishing the results in Section \ref{sec:identif}.

\begin{lemma}\label{lem:min_f}
For some $a > 0$, we have
\begin{align*}
\argmin_{x > 0} \;\; \log x + \frac{a}{x} = a.
\end{align*} 
\end{lemma}
\begin{proof}
Let $f(x) := \log x + a/x$, then the result follows from the fact that the only root of $f'(x) = 0$ is $x = a$ and $f''(a) = 1/a^2 > 0$.
\end{proof}

\begin{lemma}\label{lem:minH}
We have
\begin{align*}
\min_{(b, \theta)} \;\; H(b, \theta) \; = \; p (1 + \log 2) \; + \;  \log  \Bigg(\min_b \; \prod_{j \in [p]} \sfE_*\bigg[\bigg|X_j - \sum_{k \in {\rm pa}(j)} b_{jk} X_{k}\bigg|\bigg]\Bigg).
\end{align*}
\end{lemma}
\begin{proof}
According to \eqref{eq:pX}, we have, for any $(b, \theta)$,
\begin{align}\label{eq:H}
H(b, \theta) = p \, \log 2 + \sum_{j \in [p]} \Bigg(\log \theta_j + \frac{1}{\theta_j} \sfE_*\bigg[\bigg|X_j - \sum_{k \in {\rm pa}(j)} b_{jk} X_{k}\bigg|\bigg]\Bigg).
\end{align} 
Thus, for any fixed $b$, if we define $\tilde{\theta}(b) = (\tilde{\theta}_j(b) : j \in [p]) := \argmin_\theta H(b, \theta)$, then by Lemma \ref{lem:min_f}, for every $j \in [p]$,
\begin{align*}
\tilde{\theta}_j(b) = \sfE_*\bigg[\bigg|X_j - \sum_{k \in {\rm pa}(j)} b_{jk} X_{k}\bigg|\bigg],
\end{align*}
which further implies from \eqref{eq:H} that
\begin{align*}
\min_{\theta} \;\; H(b, \theta) \; = \; H(b, \tilde{\theta}(b)) \; = \; p (1 + \log 2) \; + \;  \log  \Bigg(\prod_{j \in [p]} \sfE_*\bigg[\bigg|X_j - \sum_{k \in {\rm pa}(j)} b_{jk} X_{k}\bigg|\bigg]\Bigg).
\end{align*}
Finally, taking minimum over $b$ on both sides finishes the proof due to log-concavity of $f^\gamma(X|b^\gamma, \theta^\gamma, \gamma)$ in $(b^\gamma, \theta^\gamma)$, see Lemma \ref{lem:log-concave}, resulting in the
convexity of $H(b, \theta)$.
\end{proof}

\begin{lemma}\label{lem:sum_sym_rvs}
If $X$ and $Y$ are two independent random variables with their distributions being symmetric with respect to $0$, then the distribution of $X + Y$ is also symmetric with respect to $0$.
\end{lemma}
\begin{proof}
If for every $x \in \bR$, we define $F(x) := \sfP(X \leq x)$ and $\bar{F}(x) := \sfP(X < x)$, then for any $t > 0$,
\begin{align*}
\sfP (X + Y \leq -t) &= \sfE[\sfP(X \leq -Y - t | Y)] = \sfE [F(-Y - t)]\\ 
&= 1 - \sfE [\bar{F}(Y + t)] = 1 - \sfE [\bar{F}(-Y + t)] \\
&= 1 - \sfE[\sfP(X < -Y + t | Y)] = 1 - \sfP(X + Y < t) = \sfP (X + Y \geq t),
\end{align*}
where the third equality follows from the fact that, due to symmetry of $X$, $F(-x) = \sfP(X \leq -x) = \sfP(X \geq x) = 1 - \bar{F}(x)$, and the fourth equality follows since $Y$ is equally distributed with $-Y$ due to symmetry of $Y$. Thus, the result follows.
\end{proof}

\begin{lemma}\label{lem:min_abs_exp}
If $U, V_1, V_2, \dots, V_k$ are independent random variables whose distributions are symmetric with respect to $0$, and $\sfE[|U|], \sfE[|V_i|] < \infty$ for every $i \in [k]$, then 
\begin{align*}
\argmin_{(t_1, \dots, t_k) \in \bR^k} \;\; \sfE \big[\big|U + \sum_{i \in [k]} t_i V_i\big|\big] = (0, 0, \dots, 0).
\end{align*}
\end{lemma}
\begin{proof} 
Suppose the underlying probability distribution of $V_k$ is denoted by $\sfP_k$.
Then
\begin{align*}
&\sfE \big[\big|U + \sum_{i \in [k]} t_i V_i\big|\big]\\ 
&= \sfE \bigg[ \sfE\bigg[\big|U + \sum_{i \in [k]} t_i V_i \big| \;\; \bigg| \;\; V_k\bigg]\bigg]\\
&= \sfE \bigg[ \sfE\bigg[\big|U + \sum_{i \in [k]} t_i V_i \big| \;\; \bigg| \;\; V_k\bigg] \mathbbm{1}\{V_k \neq 0\} \bigg] + \sfE \bigg[ \sfE\bigg[\big|U + \sum_{i \in [k]} t_i V_i \big| \;\; \bigg| \;\; V_k\bigg] \mathbbm{1}\{V_k = 0\} \bigg]\\
&= \int_{V_k \neq 0} \sfE\bigg[\big|U + \sum_{i \in [k]} t_i V_i \big| \;\; \bigg| \;\; V_k\bigg] d\sfP_k + \sfE \bigg[ \sfE\bigg[\big|U + \sum_{i \in [k-1]} t_i V_i \big| \;\; \bigg| \;\; V_k = 0\bigg] \mathbbm{1}\{V_k = 0\} \bigg]\\
&= \int_{V_k \neq 0} |V_k| \sfE\bigg[\big|\big(U + \sum_{i \in [k-1]} t_i V_i\big)/V_k + t_k\big| \;\; \bigg| \;\; V_k\bigg] d\sfP_k + \sfP_k(V_k = 0) \; \sfE \big[\big|U + \sum_{i \in [k-1]} t_i V_i\big|\big].
\end{align*}
Now, for any arbitrarily fixed $t_1, t_2, \dots, t_{k-1} \in \bR$ and $v \neq 0$, we have
\begin{align*}
\sfE\bigg[\big|\big(U + \sum_{i \in [k-1]} t_i V_i\big)/V_k + t_k\big| \;\; \bigg| \;\; V_k = v\bigg] = \sfE\bigg[\big|\big(U + \sum_{i \in [k-1]} t_i V_i\big)/v + t_k\big| \bigg] < \infty,
\end{align*}
where the equality follows since $V_k$ is independent with $U, V_1, \dots, V_{k-1}$. Note that the second quantity in the above is minimized at $t_k = 0$ because the median of the random variable $(U + \sum_{i \in [k-1]} t_i V_i\big)/v$ is $0$, which is true as its distribution is symmetric with respect to $0$ due to Lemma \ref{lem:sum_sym_rvs}. Therefore, the result follows.
\end{proof}

\begin{lemma}\label{lem:apply_rep_anc}
For every $j \in [p]$ and any set of real numbers $\{c_{jk} : k \in \pa^*(j)\}$, we have
\begin{align*}
\sum_{k \in \pa^*(j)} c_{jk}X_k = \sum_{\ell \in \an^*(j)} \sum_{k \in \pa^*(j) \cap \bar{\de^*}(\ell)} c_{jk}\beta^*_{k \gets \ell} \epsilon_\ell.
\end{align*}
\end{lemma}
\begin{proof}
We have
\begin{align*}
\sum_{k \in {\rm pa}^*(j)} c_{jk} X_{k}
&= \sum_{k \in \pa^*(j)} c_{jk} \sum_{\ell \in \bar{\an^*}(k)} \beta^*_{k \gets \ell} \epsilon_\ell\\
&= \sum_{k \in \pa^*(j)} \sum_{\ell \in \bar{\an^*}(k)} c_{jk} \beta^*_{k \gets \ell} \epsilon_\ell\\
&= \sum_{\ell \in \an^*(j)} \sum_{k \in \pa^*(j) \cap \bar{\de^*}(\ell)} c_{jk}\beta^*_{k \gets \ell} \epsilon_\ell,
\end{align*}
where the first equality follows from Lemma \ref{lem:rep_anc}, and third one follows from the similar step in Lemma \ref{lem:rep_anc}.
\end{proof}

\begin{lemma}\label{lem:minH*}
We have
\begin{align*}
\tilde{b}^*_j = \argmin_{b^*} \;\; \prod_{j \in [p]} \sfE_*\bigg[\bigg|X_j - \sum_{k \in {\rm pa^*}(j)} b^*_{jk} X_{k}\bigg|\bigg],
\end{align*}
and furthermore \, $\tilde{b}^*_{jk} = \beta_{jk}^*$ \, for every $j \in [p]$ and $k \in \pa^*(j)$.
\end{lemma}
\begin{proof}
Fix $j \in [p]$. The first part follows from the proof of Lemma \ref{lem:minH}. Now, consider the objective function 
\begin{align} \notag
&\sfE_*\bigg[\bigg|X_j - \sum_{k \in {\rm pa}^*(j)} b^*_{jk} X_{k}\bigg|\bigg]\\ \notag
&= \sfE_*\bigg[\bigg|\sum_{k \in {\rm pa}^*(j)} \beta^*_{jk} X_{k} + \epsilon_j - \sum_{k \in {\rm pa}^*(j)} b^*_{jk} X_{k}\bigg|\bigg]\\ \notag
&= \sfE_*\bigg[\bigg| \sum_{\ell \in \an^*(j)} \sum_{k \in \pa^*(j) \cap \bar{\de^*}(\ell)} \beta^*_{jk}\beta^*_{k \gets \ell} \epsilon_\ell + \epsilon_j - \sum_{\ell \in \an^*(j)} \sum_{k \in \pa^*(j) \cap \bar{\de^*}(\ell)} b^*_{jk}\beta^*_{k \gets \ell} \epsilon_\ell\bigg|\bigg]\\ \label{eq:lem:minH*1}
&= \sfE_*\bigg[\bigg| \sum_{\ell \in \an^*(j)} \sum_{k \in \pa^*(j) \cap \bar{\de^*}(\ell)} (\beta^*_{jk} - b^*_{jk})\beta^*_{k \gets \ell} \epsilon_\ell + \epsilon_j \bigg|\bigg],
\end{align}
where the first equality is due to \eqref{eq:model}, the second one follows by an application of Lemma \ref{lem:apply_rep_anc}. Furthermore, in the representation in \eqref{eq:lem:minH*1}, the variables $\{\epsilon_\ell : \ell \in \an^*(j)\}$ and $\epsilon_j$ are independent with distributions symmetric with respect to $0$ and finite first moment. Thus, by applying Lemma \ref{lem:min_abs_exp} and following the first par, we have
\begin{align}\label{eq:minH*1}
\sum_{k \in \pa^*(j) \cap \bar{\de^*}(\ell)} (\beta^*_{jk} -  \tilde{b}^*_{jk})\beta^*_{k \gets \ell} = 0 \qquad \text{for every} \;\; \ell \in \an^*(j).
\end{align}
This implies that \eqref{eq:minH*1} is true for every $\ell \in \pa^*(j) \subseteq \an^*(j)$. Fix any arbitrary $\ell \in \pa^*(j)$ and since by definitions $\ell \in \bar{\de^*}(\ell)$ and $\beta^*_{\ell \gets \ell} = 1$, we can rewrite \eqref{eq:minH*1} as 
\begin{align}\label{eq:minH*2}
(\beta^*_{j\ell} -  \tilde{b}^*_{j\ell}) + \sum_{k \in \pa^*(j) \cap {\de^*}(\ell)} (\beta^*_{jk} -  \tilde{b}^*_{jk})\beta^*_{k \gets \ell} = 0.
\end{align}

Next, we prove that $\tilde{b}^*_{j\ell} = \beta_{j\ell}^*$ for every $\ell \in \pa^*(j)$ by induction over their causal orders, from the highest to the lowest. Note that, the hypotheses is true for $\ell = \argmax_{k \in \pa^*(j)} \sigma^*(k)$ since $\pa^*(j) \cap {\de^*}(\ell) = \emptyset$. Now, fix $\ell \in \pa^*(j)$ for which $\sigma^*(\ell) = m$, where clearly $m  < \sigma^*(j) \leq p$, and suppose that the hypotheses is true for every $\ell \in \pa^*(j)$ such that $\sigma^*(\ell) \geq m+1$. Again, for every $k \in \de^*(\ell)$, $\sigma^*(k) \geq \sigma^*(\ell)+1 = m+1$, therefore, for every $k \in \pa^*(j) \cap {\de^*}(\ell)$, we have $\tilde{b}^*_{jk} = \beta^*_{jk}$. This readily implies from \eqref{eq:minH*2} that $\tilde{b}^*_{j\ell} = \beta_{j\ell}^*$.
The proof is complete.
\end{proof}

\begin{lemma}\label{lem:supset_h} 
Suppose there exists \ $v^* > 0$ \ such that for every $\gamma \in \Gamma^p$,
\begin{align*}
    \min_{b^\gamma}\; h^\gamma(b^\gamma) \, \geq \, v^*, \qquad \text{where} \;\; 
    h^\gamma(b^\gamma) \, :=
    \; \prod_{j \in [p]} \sfE_*\bigg[\bigg|X_j - \sum_{k \in {\rm pa}^\gamma(j)} b_{jk}^\gamma X_{k}\bigg|\bigg]. 
\end{align*}
Furtheremore, if for some $\gamma \in \Gamma^p$, there exists $\tilde{b}^\gamma$ such that $h^\gamma(\tilde{b}^\gamma) = v^*$, then equality holds in the above for every $\gamma' \supseteq \gamma$.
\end{lemma}
\begin{proof} 
Fix any $\gamma' \supseteq \gamma$. Then for every $j \in [p]$, $\pa^\gamma(j) \subseteq \pa^{\gamma'}(j)$. Now, for every $b^{\gamma'}$, we have $h^{\gamma'}(b^{\gamma'}) \geq v^*$. Furthermore, let $\tilde{b}^{\gamma'}$ be such that $\tilde{b}^{\gamma'}_{jk} = \tilde{b}^\gamma_{jk}$ for every $j \in [p], k \in \pa^{\gamma}(j)$. Then considering the limit as $\tilde{b}^{\gamma'}_{jk} \to 0$ for every $k \in \pa^{\gamma'}(j) \setminus \pa^\gamma(j), j \in [p]$, we have
\begin{align*}
    \lim \; h^{\gamma'}(\tilde{b}^{\gamma'})
    = \lim \,
    \prod_{j \in [p]} \sfE_*\bigg[\bigg|X_j - \sum_{k \in {\rm pa}^\gamma(j)} \tilde{b}_{jk}^\gamma X_{k} - \sum_{k \in {\rm pa}^{\gamma'}(j) \setminus \pa^\gamma(j)} \tilde{b}_{jk}^{\gamma'} X_{k}\bigg|\bigg] = h^\gamma(\tilde{b}^\gamma) = v^*.
\end{align*}
This implies that $\min_{b^{\gamma'}} h^{\gamma'}(b^{\gamma'}) = v^*$, and the proof is complete.
\end{proof}

\subsection{Proof of Theorem \ref{thm:H<H}}\label{pf:thm:H<H}

\begin{proof}

We have
\begin{align*}
h_* &= H^*(\tilde{b}^*, \tilde{\theta}^*) = \min_{(b^*, \theta^*)} \; H^*(b^*, \theta^*)\\ 
&= \; p (1 + \log 2) \; + \;  \log  \Bigg(\min_{b^*} \; \prod_{j \in [p]} \sfE_*\bigg[\bigg|X_j - \sum_{k \in {\rm pa}^*(j)} b^*_{jk} X_{k}\bigg|\bigg]\Bigg)\\
&= \; p (1 + \log 2) \; + \;  \log  \Bigg(\prod_{j \in [p]} \sfE_*\bigg[\bigg|X_j - \sum_{k \in {\rm pa}^*(j)} \beta^*_{jk} X_{k}\bigg|\bigg]\Bigg)\\
&= \; p (1 + \log 2) \; + \;  \log  \Bigg(\prod_{j \in [p]} \sfE_*[|\epsilon_j|] \Bigg),
\end{align*}
where the third equality follows from Lemma \ref{lem:minH}, the fourth one follows from Lemma \ref{lem:minH*} and the last one is due to \eqref{eq:model}.
Therefore, in view of Lemma \ref{lem:minH}, it suffices to show that
\begin{align}\label{eq:H*<H}
\prod_{j \in [p]} \sfE_*[|\epsilon_j|] \; \leq \; \min_b \; \prod_{j \in [p]} \sfE_*\bigg[\bigg|X_j - \sum_{k \in {\rm pa}(j)} b_{jk} X_{k}\bigg|\bigg] \; = \; \min_b \; \prod_{j \in [p]}\sfE_*[|e_j|].
\end{align}
Indeed, the above holds since by Lemma \ref{lem:ineq_eps_e}, for every $b$,
\begin{align*}
\prod_{j \in [p]} \sfE_*[|\epsilon_j|] \; \leq \; \prod_{j \in [p]}\sfE_*[|e_j|].
\end{align*}
The equality in the above holds if and only if the conditions in Lemma \ref{lem:ineq_eps_e} hold, and in that event, we further consider two cases: either $\cR_A = \cC_A = \emptyset$ or $\cR_A, \cC_A \neq \emptyset$.  Following Lemma \ref{lem:R=C=emp}, the first case holds if and only if $\gamma = \gamma^*$, trivially satisfying conditions (1) and (2). Following Lemma \ref{lem:R,C=non_emp}, the latter case holds if and only if the conditions in Lemma \ref{lem:R,C=non_emp} holds. Now, following condition (b)(ii) in Lemma \ref{lem:R,C=non_emp}, for every $k \in \cC_A$, $\epsilon_k$ is Gaussian, which implies that $n\cG^* \subseteq \cC_A^c$. If $j \in n\cG^*$, then there exists $i \notin \cR_A$ such that $j = \sigma^{-1}(i) \notin \cC_A$, and thus, following condition (a) in Lemma \ref{lem:R,C=non_emp}, $\pa(j) = \pa^*(j)$, satisfying condition (1). However, if $j \notin n\cG^*$, then in case $j \in \cC_A$, condition (b) in Lemma \ref{lem:R,C=non_emp} immediately implies the existence of $\beta_{jk}^\gamma, k \in \pa(j)$ such that $\eta_j^\gamma$ is some linear combination of the Gaussian errors, in particular, $\epsilon_k, k \in \cC_A$, and also, $\eta_j, j \in \cC_A$ are pairwise independent. Furthermore, in case $j \notin \cC_A$, by letting $\pa(j) = \pa^*(j)$ and $\beta_{jk}^\gamma = \beta^*_{jk}$ according to condition (a) in Lemma \ref{lem:R,C=non_emp}, we have $\eta_j^\gamma = \epsilon_j$. This implies that $\eta_j^\gamma, j \notin n\cG^*$ are pairwise independent, and thus, condition (2) is satisfied. The sufficiency part follows similarly. Finally, by using Lemma \ref{lem:supset_h}, the equality in \eqref{eq:H*<H} is extended to any superset of $\gamma$, and this completes the proof. 
\end{proof}

\subsection{Proof of Corollary \ref{cor:E=DE=gam}}\label{pf:cor:E=DE=gam}

\begin{proof}
    If $\bar{\cE}^* = \cS^*$, then following \ref{thm:H<H} there exists no $\gamma \neq \gamma^*$ such that $\sfP_X^* \in \cP(\gamma, n\cG)$ for some $n\cG \subseteq [p]$. This immediately implies from the definitions \eqref{eq:dist_equiv} and \eqref{def:epsR} that $\bar{\cE}^*_R = \cE(\gamma^*, n\cG^*) = \{\gamma^*\}$. Also, since for every $\gamma \in \cS^*, \gamma \neq \gamma^*$, clearly $\gamma \supset \gamma^*$ i.e., $|\gamma| > |\gamma^*|$, we have, following the definition \eqref{def:E*}, $\cE^* = \{\gamma^*\}$. 
\end{proof}

\begin{lemma}\label{lem:nG=nG*}
    If  $\cP(\gamma, n\cG) = \cP(\gamma^*, n\cG^*)$ for some $n\cG \subseteq[p]$, then $n\cG = n\cG^*$. Furthermore,
    \begin{align*}
    \cE(\gamma^*, n\cG^*) = \{\gamma \in \Gamma^p : \cP(\gamma, n\cG^*) = \cP(\gamma^*, n\cG^*)\}.
\end{align*}
\end{lemma}
\begin{proof}
    Consider $\sfP_X \in \cP(\gamma^*, n\cG^*)$, i.e., there exists independent random variables $\eta_j^*, j \in [p]$ such that under $\sfP_X$, $X_j, j \in [p]$ are generated by some linear acyclic SEM represented by $\gamma$ with the error corresponding to node $j$ being $\eta_j^*$, and furthermore, for every $j \in n\cG$, $\eta_j^*$ is non-Gaussian. Since $\sfP_X \in \cP(\gamma, n\cG)$, i.e., there exist independent random variables $\eta_j^\gamma, j \in [p]$ such that we have an equivalent SEM representation according to $\gamma$ with the errors $\eta_j^\gamma$, and furthermore, for every $j \in n\cG$, $\eta_j^\gamma$ is non-Gaussian. Now, it is important to emphasize here that, in Lemma \ref{lem:ineq_eps_e}, Lemma \ref{lem:R=C=emp} and Lemma \ref{lem:R,C=non_emp}, the fact that $n\cG^* \subseteq \cC_A^c$ only depends on $\epsilon_j, j \in n\cG^*$ being non-Gaussian and it does not depend on any other distributional assumptions. Therefore, the same result holds for any $\sfP_X \in \cP(\gamma^*, n\cG^*)$, and subsequently following these lemmas, we have for every $j \in n\cG^*$, $\eta_j^\gamma = \eta_j^*$, i.e., $\eta_j^\gamma$ is non-Gaussian. This implies that $n\cG^* \subseteq n\cG$, and again by the same steps above we can show that $n\cG \subseteq n\cG^*$, which establishes that $n\cG = n\cG^*$. Thus, following the definition in \eqref{eq:dist_equiv}, the second result holds immediately, and the proof is complete.
\end{proof}

\subsection{Proof of Theorem \ref{thm:char_E}}\label{pf:thm:char_E}

\begin{proof}
Fix $\gamma \in \cE(\gamma^*, n\cG^*)$. Then due to Lemma \ref{lem:nG=nG*}, $\sfP_X^* \in \cP(\gamma, n\cG^*)$, which in turn by Theorem \ref{thm:H<H} implies that $\pa^\gamma(j) = \pa^*(j)$ for every $j \in n\cG^*$. Furthermore, since $\sfP_X^*$ is faithful to $\gamma^*$, we have $\bI(\gamma) \subseteq \bI(\sfP_X^*) = \bI(\gamma^*)$. Now, consider a probability distribution $\sfP_X \in \cP(\gamma, n\cG^*)$ which is faithful to $\gamma$. Then, as $\sfP_X \in \cP(\gamma^*, n\cG^*)$, we must have $\bI(\gamma^*) \subseteq \bI(\sfP_X) = \bI(\gamma)$, and therefore, $\bI(\gamma) = \bI(\gamma^*)$. This shows that
\begin{align*}
    \cE(\gamma^*, n\cG^*) \subseteq \{\gamma \in \Gamma^p :  \pa^\gamma(j) = \pa^*(j) \quad \text{for every} \;\; j \in n\cG^* \quad \text{and} \quad \bI(\gamma) = \bI(\gamma^*)\}.
\end{align*}
Next, fix $\gamma \in \Gamma^p$ such that $\pa^\gamma(j) = \pa^*(j)$ for every $j \in n\cG^*$ and $\bI(\gamma) = \bI(\gamma^*)$. Consider $\sfP_X \in \cP(\gamma^*, n\cG^*)$, i.e., there exists independent random variables $\eta_j^*, j \in [p]$ such that under $\sfP_X$, $X_j, j \in [p]$ are generated by some linear acyclic SEM represented by $\gamma^*$ with the error corresponding to node $j$ being $\eta_j^*$, and furthermore, for every $j \in n\cG^*$, $\eta_j^*$ is non-Gaussian. If for every $j \in n\cG^*$, $\eta_j^*$ were Gaussian, then $\sfP_X$ would be some multivariate Gaussian distribution. Now, it is well known \cite{geiger2002parameter} that for Gaussian DAG models, Markov equivalence is equivalent to distribution equivalence, and thus, as $\bI(\gamma) = \bI(\gamma^*)$, i.e., $\gamma$ is Markov equivalent to $\gamma^*$, $\sfP_X$ can be represented by some linear acyclic SEM according to $\gamma$ with the errors denoted by the independent random variables $\eta_j^\gamma, j \in [p]$. Furthermore, since $\pa^\gamma(j) = \pa^*(j)$ for every $j \in n\cG^*$, we must also have $\eta_j^\gamma = \eta_j^*$ for every $j \in n\cG^*$. Otherwise, this leads to
\begin{align*}
    \eta_j^\gamma - \eta_j^* = \sum_{k \in \pa^*(j)}c_{jk}X_k,
\end{align*}
where $c_{jk}$ must be non-zero for at least one $k \in \pa^*(j)$, in turn contradicting the fact that  $X_k, k \in \pa^*(j)$ are independent of both $\eta_j^\gamma$ and $\eta_j^*$. Moreover, in order for $\eta_j^\gamma, j \in [p]$ being pairwise independent, for every $j \notin n\cG^*$, 
$\eta_j^\gamma$ must be some linear combination of the errors $\eta^*_j, j \notin n\cG^*$. Therefore, the expressions of $\eta_j^\gamma, j \in [p]$ imply that the SEM representation still holds even when $\eta_j^*, j \in n\cG^*$ were non-Gaussian. Thus, $\sfP_X \in \cP(\gamma, n\cG^*)$, implying that $\cP(\gamma^*, n\cG^*) \subseteq \cP(\gamma, n\cG^*)$. We can similarly show that $\cP(\gamma, n\cG^*) \subseteq \cP(\gamma^*, n\cG^*)$, leading us to $\cP(\gamma^*, n\cG^*) = \cP(\gamma, n\cG^*)$. As a result, we have
\begin{align*}
      \{\gamma \in \Gamma^p :  \pa^\gamma(j) = \pa^*(j) \quad \text{for every} \;\; j \in n\cG^* \quad \text{and} \quad \bI(\gamma) = \bI(\gamma^*)\} \subseteq \cE(\gamma^*, n\cG^*),
\end{align*}
which completes the proof.
\end{proof}

\begin{lemma}\label{lem:IgamsubIgam*}
If $\sfP_X^*$ is faithful to $\gamma^*$, and $H^*(\tilde{b}^*, \tilde{\theta}^*) = H(\tilde{b}, \tilde{\theta})$, then $\bI(\gamma) \subseteq \bI(\gamma^*)$.
\end{lemma}
\begin{proof}
Since $\sfP_X^*$ is Markov with respect to $\gamma^*$, we already have $\bI(\gamma^*) \subseteq \bI(\sfP_X^*)$, and thus, due to faithfulness, we further have $\bI(\gamma^*) = \bI(\sfP^*_X)$. Now, following Theorem \ref{thm:H<H}, we have $\sfP^*_X \in \cP(\gamma, n\cG^*)$, i.e., $\sfP_X^*$ factorizes with respect to $\gamma$, which implies that $\sfP^*_X$ is Markov with respect to $\gamma$. Thus, $\bI(\gamma) \subseteq \bI(\sfP^*_X) = \bI(\gamma^*)$, and the proof is complete.
\end{proof}

\begin{lemma}\label{lem:|gam|and|gam*|}
Suppose that $\bI(\gamma) \subseteq \bI(\gamma^*)$. Then the following hold.
\begin{enumerate}
\item[(i)] $|\gamma^*| \leq |\gamma|$.
\item[(ii)] $|\gamma^*| = |\gamma|$ \;\; if and only if \;\; $\bI(\gamma) = \bI(\gamma^*)$, i.e., $\gamma$ and $\gamma^*$ are Markov equivalent.
\end{enumerate}
\end{lemma}
\begin{proof}
Suppose that $|\gamma| < |\gamma^*|$. Then there exist $i, j \in [p]$ such that $(i \to j) \in \gamma^*$ but $i$ and $j$ are not adjacent in $\gamma$, i.e., both $(i \to j), (j \to i) \notin \gamma$. This implies for every $\cV \subseteq [p] \setminus \{i, j\}$ we have \; $i \nindep j | \cV$ \; under $\gamma^*$, however, there exists $\cV \subseteq [p] \setminus \{i, j\}$ such that \; $i \indep j | \cV$ \; under $\gamma$. Thus, $\bI(\gamma) \nsubseteq \bI(\gamma^*)$, leading to contradiction. This proves (i).

Furthermore, if $\bI(\gamma) = \bI(\gamma^*)$ then by \cite{andersson1997characterization} they must have the same skeleton, which implies $|\gamma| = |\gamma^*|$. Now, suppose that $|\gamma| = |\gamma^*|$ but $\bI(\gamma) \neq \bI(\gamma^*)$. Then, again by \cite{andersson1997characterization} they either have different skeleton or have different v-structure. In the first case, since $|\gamma| = |\gamma^*|$ there exist $i, j \in [p]$ such that $(i \to j) \in \gamma^*$ but $i$ and $j$ are not adjacent in $\gamma$, i.e., both $(i \to j), (j \to i) \notin \gamma$. By the same argument provided above in the proof of (i), this again implies $\bI(\gamma) \nsubseteq \bI(\gamma^*)$, leading to contradiction. Therefore, they must have the same skeleton but different v-structure, which again immediately implies that $\bI(\gamma) \nsubseteq \bI(\gamma^*)$. Thus, $\bI(\gamma) = \bI(\gamma^*)$, and this proves (ii).
\end{proof}

\subsection{Proof of Theorem \ref{thm:E=DE}}\label{pf:thm:E=DE}

\begin{proof}
We recall from \eqref{def:E*} that
\begin{align*}
\cE^* := \Big\{\gamma \in \Gamma^p : H^\gamma(\tilde{b}^\gamma, \tilde{\theta}^\gamma) = H^*(\tilde{b}^*, \tilde{\theta}^*) \quad \text{and} \quad |\gamma| = |\gamma^*|\Big\}.
\end{align*}
Therefore, the first result follows immediately from Lemma \ref{lem:IgamsubIgam*} and Lemma \ref{lem:|gam|and|gam*|}(ii).

Now, fix $\gamma \in \cE^*$. Then following the first condition for equality in Theorem \ref{thm:H<H}, we must have, for every $j \in n\cG^*$, $\pa^\gamma(j) = \pa^*(j)$. Thus, using the first result, we establish that $\cE^* \subseteq \cE^*(\gamma^*, n\cG^*)$.    

Next, fix $\gamma \in \cE^*(\gamma^*, n\cG^*)$. Then due to Lemma \ref{lem:nG=nG*}, $\sfP_X^* \in \cP(\gamma, n\cG^*)$, which in turn by Theorem \ref{thm:H<H} implies that $H^\gamma(\tilde{b}^\gamma, \tilde{\theta}^\gamma) = H^*(\tilde{b}^*, \tilde{\theta}^*)$. Furthermore, from the first part of the proof in Theorem \ref{thm:char_E}, we have $\bI(\gamma) = \bI(\gamma^*)$. Therefore, using the first result we obtain $\cE^*(\gamma^*, n\cG^*) \subseteq \cE^* $, and this completes the proof.
\end{proof}

\subsection{Proof of Corollary \ref{cor:ineq_risk_and_param}}\label{pf:cor:ineq_risk_and_param}
\begin{proof}
Fix $\gamma \notin \cE^*$. Then by Theorem \ref{thm:E=DE}, we either have $H^\gamma(\tilde{b}^\gamma, \tilde{\theta}^\gamma) \neq H^*(\tilde{b}^*, \tilde{\theta}^*)$, or $\bI(\gamma) \neq \bI(\gamma^*)$. If the former happens, then due to Theorem \ref{thm:H<H}, we must have $H^\gamma(\tilde{b}^\gamma, \tilde{\theta}^\gamma) < H^*(\tilde{b}^*, \tilde{\theta}^*)$. When that is not the case, i.e., $H^\gamma(\tilde{b}^\gamma, \tilde{\theta}^\gamma) = H^*(\tilde{b}^*, \tilde{\theta}^*)$, then following Lemma \ref{lem:IgamsubIgam*} we have $\bI(\gamma) \subseteq \bI(\gamma^*)$. However, in that case, it is necessary that $\bI(\gamma) \neq \bI(\gamma^*)$, which further yields $\bI(\gamma) \subset \bI(\gamma^*)$. Therefore, following Lemma \ref{lem:|gam|and|gam*|} we must have $|\gamma| > |\gamma^*|$, which completes the proof.
\end{proof}

In what follows, we denote by $\Var_*[\cdot]$ the associated variance.

\subsection{Proof of Proposition \ref{prop:spl_cases}}\label{pf:prop:spl_cases}

\begin{proof}
Fix $\gamma \in \bar{\cE}^*$, i.e., $H^\gamma(\tilde{b}^\gamma, \tilde{\theta}^\gamma) = H^*(\tilde{b}^*, \tilde{\theta}^*)$. Then, following the same steps in the proof of Theorem \ref{thm:H<H}, this is equivalent to having
\begin{align*}
\prod_{j \in [p]} \sfE_*[|\epsilon_j|] \; = \; \prod_{j \in [p]}\sfE_*[|e_j|].
\end{align*}
Again, due to Lemma \ref{lem:ineq_eps_e}, the above holds if and only if the conditions in Lemma \ref{lem:ineq_eps_e} hold. 

First, suppose that (a) holds. In that case, if $\cC_A \neq \emptyset$, then due to condition (i) in Lemma \ref{lem:ineq_eps_e}, it is necessary that $|n\cG^*| \leq (p-2)$, which leads us to contradiction. Thus, $\cR_A = \cC_A = \emptyset$, and then by following Lemma \ref{lem:R=C=emp}, we must have $\gamma = \gamma^*$.

Next, suppose that (b) holds, and let $\Var_*[\epsilon_j] = V^*$ for every $j \in [p]$. Then, following \eqref{eq:defA}, for every $i \in [p]$, we have $\Var_*[e_{\sigma^{-1}(i)}] = \Var_*[a_i^T\epsilon] = ||a_i||^2V^*$. Thus, due to the equality of error variances, we must have $||a_i||, i \in [p]$ all equal, say to $a$. Furthermore, since by Lemma \ref{lem:ineq_eps_e}, $e_i, i \in [p]$ are pairwise independent, we have $a_i, a_j$ being orthogonal for every $i, j \in [p], i \neq j$. Thus, applying the equality condition in Hadamard's inequality, see Lemma \ref{lem:Hadamard}, we have
\begin{align*}
    |\det(A^T)| = \prod_{i \in [p]} ||a_i|| = a^p.
\end{align*}
Since by Lemma \ref{lem:detA}, $\det(A) = 1$, we have $a = 1$, i.e., for every $i \in [p]$, $||a_i|| = 1$. Now, suppose that $\cR_A \neq \emptyset$, and let $\ell = \min \cR_A$. Then following the same steps from the proof of Lemma \ref{lem:RCempty_ki}, it is not difficult to show that $a_{j\kappa(j)} = 1$ and $\kappa(j) = \sigma^{-1}(j)$ for every $j < \ell$. Thus, from the representation in \eqref{eq:X_Be} and condition (ii) in Lemma \ref{lem:ineq_eps_e} we have
\begin{align*}
X_{\sigma^{-1}(\ell)} &= \sum_{k \in \bar{\an}(\sigma^{-1}(\ell))}b_{\sigma^{-1}(\ell) \gets k}e_k\\
&= \sum_{\sigma^{-1}(j) \in \an(\sigma^{-1}(\ell))}b_{\sigma^{-1}(\ell) \gets \sigma^{-1}(j)}e_{\sigma^{-1}(j)} + e_{\sigma^{-1}(\ell)}\\
&= \sum_{\sigma^{-1}(j) \in \an(\sigma^{-1}(\ell))}b_{\sigma^{-1}(\ell) \gets \sigma^{-1}(j)}\epsilon_{\sigma^{-1}(j)} + a_\ell^T\epsilon,
\end{align*} where $|\supp(a_\ell)| \geq 2$. 
 Comparing the above with the representation obtained from \eqref{eq:X_Beps}, which is
\begin{align}
X_{\sigma^{-1}(\ell)} = \sum_{k \in \an^*(\sigma^{-1}(\ell))} \beta^*_{\sigma^{-1}(\ell) \gets k} \epsilon_k + \epsilon_{\sigma^{-1}(\ell)},
\end{align}
we must have $a_{\ell \sigma^{-1}(\ell)} = 1$. However, since $|\supp(a_\ell)| \geq 2$, it further implies that $||a_\ell|| > 1$, leading to a contradiction. Thus, it is necessary that $\cR_A = \cC_A = \emptyset$, which by Lemma \ref{lem:R=C=emp} leads us to $\gamma = \gamma^*$.

Finally, suppose that (c) holds. The first part immediately follows from Corollary \ref{cor:barE*}, and the second part is shown as follows. Since as per condition (i) in Lemma \ref{lem:ineq_eps_e}, $\epsilon_j, j \in \cC_A$ are Gaussian, let $\Var_*[\epsilon_j] = V^*$ for every $j \in \cC_A$. Then as per condition (ii) in Lemma \ref{lem:ineq_eps_e}, $e_{\sigma^{-1}(i)}, i \in \cR_A$,
are Gaussian, and for every $i \in \cR_A$, we have $\Var_*[e_{\sigma^{-1}(i)}] = \Var_*[a_i^T\epsilon] = ||a_i||^2V^*$. Thus, due to the equality of variances for the Gaussian errors, we must have $||a_i||, i \in \cR_A$ all equal, say to $a$. Moreover, due to condition (iii) in Lemma \ref{lem:ineq_eps_e} and Lemma \ref{lem:non-G_err_nodes}, we have $||a_i|| = 1$, for every $i \notin \cR_A$. Again, since by Lemma \ref{lem:ineq_eps_e}, $e_i, i \in [p]$ are pairwise independent, we have $a_i, a_j$ being orthogonal for every $i, j \in [p], i \neq j$. Thus, applying the equality condition in Hadamard's inequality, see Lemma \ref{lem:Hadamard}, we have
\begin{align*}
    |\det(A^T)| = \prod_{i \in [p]} ||a_i|| = \prod_{i \in \cR_A} ||a_i|| = a^{|\cR_A|}.
\end{align*}
Since by Lemma \ref{lem:detA}, $\det(A) = 1$, we have $a = 1$, i.e., for every $i \in \cR_A$, $||a_i|| = 1$. Now, following the same steps from the previous part, we can show that it eventually leads us to $\gamma = \gamma^*$. The proof is complete.
\end{proof}

\subsection{Proof of Corollary \ref{cor:anc_res}}\label{pf:cor:anc_res}

\begin{proof}
Since $k \in \an^*(j)$, suppose that $k' \in \pa^*(j)$ is such that there is a directed path from node $k$ to node $k'$ in $\gamma^*$, and similarly, since $\ell \in \de^*(j)$, suppose that $\ell'$ is such that $j \in \pa^*(\ell)$ and there is a directed path from node $\ell'$ to node $\ell$ in $\gamma^*$. We assume that $\ell \in \an^\gamma(k)$, i.e., there is a directed path from node $\ell$ to node $k$ in $\gamma$. More specifically, since $\gamma$ has the same skeleton as $\gamma^*$ due to Markov equivalence, the above implies that there exist node $k''$ on the path between node $k$ and node $k'$ in $\gamma^*$ (including both nodes) and node $\ell''$ on the path between node $\ell'$ and node $\ell$ in $\gamma^*$ (including both nodes) such that the directed path from node $\ell$ to node $k$ in $\gamma$ passes through node $\ell''$ and node $k''$. Therefore, in $\gamma$, there exists no directed path from node $k''$ to node $j$, or no directed path from node $j$ to node $\ell''$ because otherwise, it would create a cycle in $\gamma$. 

Now, note that, since $\gamma$ is Markov equivalent to $\gamma^*$, the path between node $k''$ and node $\ell''$ in $\gamma^*$ also exists in $\gamma$, and let it be denoted by $\ell'' \sim_\gamma j$. However, it must no longer be directed from $k''$ to $\ell''$ in $\gamma$ to maintain the acyclicity of $\gamma$, as stated above.
Moreover, since $j \in n\cG^*$ due to parental preservation both $(j \to \ell'),(k' \to j) \in \gamma$. Furthermore, since there exists no directed path from node $j$ to node $\ell''$ there exist nodes $j = \ell_0, \ell_1, \ell_2, \dots, \ell_k$ on $\ell'' \sim_\gamma j$ such that $(\ell_{i-1} \to \ell_i) \in \gamma^*$ for every $i \in [k]$, but $(\ell_k \to \ell_{k-1}) \in \gamma$. Therefore, it creates a new v-structure unless $(\ell_k \to \ell_{k-2}) \in \gamma$ or $(\ell_{k-1} \to \ell_{k-2}) \in \gamma$. Furthermore, if the first case happens, since $(\ell_{k-3} \to \ell_{k-2}) \in \gamma^*$, in order to avoid the creation of a new v-structure, we must have $(\ell_k \to \ell_{k-3}) \in \gamma$ or $(\ell_{k-2} \to \ell_{k-3}) \in \gamma$, and similarly under the second case, $(\ell_{k-1} \to \ell_{k-3}) \in \gamma$ or $(\ell_{k-2} \to \ell_{k-3}) \in \gamma$. That is, to sum up, we must have $(\ell_{k} \to \ell_{k-3}) \in \gamma$, $(\ell_{k-1} \to \ell_{k-3}) \in \gamma$, or $(\ell_{k-2} \to \ell_{k-3}) \in \gamma$.
Therefore, a successive application of the above argument implies that there exists $i \in [k]$ such that $(\ell_i \to \ell_0) \in \gamma$, i.e., $i \in \pa^\gamma(\ell_0) = \pa^\gamma(j)$, and as $\pa^*(j) = \pa^\gamma(j)$, we also have $(\ell_i \to j) \in \gamma^*$. However, since there is a directed path from node $j$ to node $\ell_i$, presence of $(\ell_i \to j)$ creates a cycle in $\gamma^*$. Therefore, our assumption was wrong and we must have $\ell \notin \an^\gamma(k)$. The proof is complete. 
\end{proof}

\subsection{Proof of Corollary \ref{cor:enc_DAG}}\label{pf:cor:enc_DAG}

\begin{proof}
Clearly, condition (1) ensures that the skeleton of $\gamma$ and $\gamma^*$ are the same and they also have the same v-structures, i.e., $\bI(\gamma) = \bI(\gamma^*)$, and condition (2) satisfies the parental preservation, that is, $\pa^*(j) = \pa^\gamma(j)$ for every $j \in n\cG^*$. Furthermore, condition (3) incorporates the edges which are necessary to remain undirected due to the combined effect of Markov equivalence and parental preservation, for example, to satisfy the ancestral restrictions.  
\end{proof}

\section{Establishing the Laplace approximation}\label{app:lap}

In this section, we establish the Laplace approximation under model misspecification in Theorem \ref{thm:lap_app}. 
First, since $\gamma$ is fixed in the beginning, as done previously,  we omit the superscript from the notation $f^\gamma(x | b^\gamma, \theta^\gamma, \gamma)$ and rewrite it as $f(x | b, \theta)$, where $x \in \bR^p$, for notational simplicity. Next, we let $b_j := (b_{jk} : k \in \pa(j))$ and $x_{\pa(j)} := (x_k : k \in \pa(j))$ having the same order for their corresponding elements, and consider the reparameterization of having, for every $j \in [p]$, $\eta_j = 1/\theta_j \in \bR^+$ and $w_j = b_j/\theta_j \in \bR^{\pa(j)}$ that transforms $f(x | b, \theta)$ into the following equivalent density:
\begin{align}\label{eq:eq_den}
g(x, (\eta, w)) = \prod_{j \in [p]} \frac{\eta_j}{2} \exp \lt(-|\eta_j x_j - x^T_{\pa(j)}w_j|\rt),\qquad x = (x_1, x_2, \dots, x_p) \in \bR^p,
\end{align}
where we let $(\eta, w) := (\eta_j, w_j : j \in [p]) \in \times_{j \in [p]} (\bR^+ \times \bR^{\pa(j)})$.

\subsection{Some important lemmas}

\begin{lemma}\label{lem:log-concave}
$\log g(x , (\eta, w))$ is concave in $(\eta, w)$ for every $x \in \bR^p$.
\end{lemma}
\begin{proof}
Following \eqref{eq:eq_den}, we have 
\begin{align*}
- \log g(x, (\eta, w)) = p\log 2 - \sum_{j \in [p]} \log \eta_j + \sum_{j \in [p]} |\eta_j x_j - x^T_{\pa(j)}w_j|.
\end{align*}
Note that, for every $j \in [p]$, $- \log \eta_j$ is convex in $\eta_j$, and since the function
\begin{align*}
|\eta_j x_j - x^T_{\pa(j)}w_j| &= \max\{\eta_j x_j - x^T_{\pa(j)}w_j, -\eta_j x_j + x^T_{\pa(j)}w_j\}\\
&= \max\lt\{
\begin{bmatrix}
x_j\\
-x_{\pa(j)}
\end{bmatrix}^T
\begin{bmatrix}
\eta_j\\
w_j
\end{bmatrix}, 
\begin{bmatrix}
- x_j\\
 x_{\pa(j)}
\end{bmatrix}^T
\begin{bmatrix}
\eta_j\\
w_j
\end{bmatrix}
\rt\}
\end{align*}
is maximum of two affine (hence, convex) functions, it is also convex in $(\eta_j, w_j)$. Thus, the result follows.
\end{proof}
In what follows, ${\rm sgn}(\cdot)$ denotes the signum or sign function. Moreover, we denote by $\Cov_*(\cdot)$ the associated covariance.
\begin{lemma}\label{lem:sol_exi}
For every $j \in [p]$, the following system of equations of $(\eta,w)$ has a solution:
\begin{align}\label{eq:score_eq}
\begin{split}
\frac{1}{{\eta}_j} - \sfE_*[X_j{\rm sgn}({\eta}_j X_j - X^T_{\pa(j)}{w}_j)] &= 0,\\
- \sfE_*[X_{\pa(j)} {\rm sgn}({\eta}_j X_j - X^T_{\pa(j)}{w}_j)] &= 0.
\end{split}
\end{align}
\end{lemma}
\begin{proof}
Following Lemma \ref{lem:log-concave}, it is clear that $\sfE_*[\log g(X, (\eta, w))]$ is concave in $(\eta, w)$. Therefore, there exists a solution for 
\begin{align*}
\nabla_{(\eta, w)} \sfE_*[\log g(X, (\eta, w))] = 0,
\end{align*}
where $\nabla_{(\eta, w)}$ represents the differential with respect to $(\eta, w)$. Now, following \eqref{eq:eq_den}, we have
\begin{align*}
\sfE_*[\log g(X, (\eta, w))] = - p\log 2 + \sum_{j \in [p]} \log \eta_j - \sum_{j \in [p]} \sfE_*[|\eta_j X_j - X^T_{\pa(j)}w_j|],
\end{align*}
which yields, for every $j \in [p]$,
\begin{align*}
\frac{\partial}{\partial \eta_j} \sfE_*[\log g(X, (\eta, w))] &= 
\frac{1}{{\eta}_j} - \sfE_*[X_j{\rm sgn}({\eta}_j X_j - X^T_{\pa(j)}{w}_j)],\\
\frac{\partial}{\partial w_j} \sfE_*[\log g(X, (\eta, w))] &= - \sfE_*[X_{\pa(j)} {\rm sgn}({\eta}_j X_j - X^T_{\pa(j)}{w}_j)],
\end{align*}
and the result follows.
\end{proof}

Now, fix any arbitrary $j \in [p]$. First, following Lemma \ref{lem:sol_exi}, we define the quantities $\tilde{\eta}_j \in \bR^+$ and $\tilde{w}_j \in \bR^{\pa(j)}$ as the solution of the system of equations in \eqref{eq:score_eq}.
Moreover, for any $x = (x_1, x_2, \dots, x_p) \in \bR^p$, we define the function $\cD_j : \bR^p \to \bR^{1 + \pa(j)}$ as
\begin{align}
\label{eq:cdj}
\cD_j(x) := 
\begin{bmatrix}
\frac{1}{\tilde{\eta}_j} - x_j{\rm sgn}(\tilde{\eta}_j x_j - x^T_{\pa(j)}\tilde{w}_j)\\
{\rm sgn}(\tilde{\eta}_j x_j - x^T_{\pa(j)}\tilde{w}_j) x_{\pa(j)}
\end{bmatrix},
\end{align}
and for any $t_{\eta, j} \in \bR$
and $t_{w, j} \in \bR^{\pa(j)}$, if we let $t_j = (t_{\eta, j}, t_{w, j}) \in \bR^{1 + \pa(j)}$, and $t = (t_j : j \in [p]) \in \bR^{p + \sum_{j \in [p]} \pa(j)}$, then the function $u_j : \bR^p \times \bR^{p + \sum_{j \in [p]} \pa(j)} \to \bR$ is defined as
\begin{align}
\label{eq:ujxt}
u_j(x, t) &:= 
\begin{cases}
&2(x^T_{\pa(j)}(\tilde{w_j} + t_{w, j}) - (\tilde{\eta}_j + t_{\eta, j})x_j) \\
&\qquad \times \mathbbm{1}\{x^T_{\pa(j)}\tilde{w_j} \leq \tilde{\eta}_j x_j \leq x^T_{\pa(j)}(\tilde{w_j} + t_{w, j}) - t_{\eta, j}x_j\}\\
&\qquad \qquad \qquad \qquad \qquad \qquad \qquad \qquad \qquad \text{if} \;\; x^T_{\pa(j)}t_{w, j} - t_{\eta, j}x_j \geq 0,\\
&2( - x^T_{\pa(j)}(\tilde{w_j} + t_{w, j}) + (\tilde{\eta}_j + t_{\eta, j})x_j) \\
&\qquad \times \mathbbm{1}\{x^T_{\pa(j)}\tilde{w_j} \geq \tilde{\eta}_j x_j \geq x^T_{\pa(j)}(\tilde{w_j} + t_{w, j}) - t_{\eta, j}x_j\}\\
&\qquad \qquad \qquad \qquad \qquad \qquad \qquad \qquad \qquad  \text{if} \;\; x^T_{\pa(j)}t_{w, j} - t_{\eta, j}x_j < 0.
\end{cases}
\end{align}
\begin{lemma}\label{lem:HP_decomp}
For every $y, \mu, t \in \bR$, the following decomposition holds:
\begin{align*}
|y - (\mu + t)| - |y - \mu| = D(y)t + U(y, t), 
\end{align*}
where the functions $D : \bR \to \bR$ and $U : \bR \times \bR \to \bR$ are defined as
\begin{align*}
D(y) &:= {\rm sgn}(\mu - y), \quad \text{and}\\
U(y, t) &:= 
\begin{cases}
2(t - (y - \mu)) \mathbbm{1}\{\mu \leq y \leq \mu + t\} \quad \text{if} \;\; t \geq 0,\\
2((y - \mu) - t) \mathbbm{1}\{\mu + t \leq y \leq \mu\} \quad \text{if} \;\; t < 0.
\end{cases}
\end{align*}
\end{lemma}
\begin{proof}
See Section 3A in Hjort and Pollard \cite{hjort2011asymptotics}.
\end{proof}

\begin{lemma}\label{lem:decomp}
For every $x \in \bR^p$, $t \in \bR^{p + \sum_{j \in [p]} \pa(j)}$ such that $(\tilde{\eta}, \tilde{w}) + t \in \times_{j \in [p]} (\bR^+ \times \bR^{\pa(j)})$,
the following decomposition holds:
\begin{align*}
\log g(x , (\tilde{\eta}, \tilde{w}) + t) &- \log g(x, (\tilde{\eta}, \tilde{w})) = \cD(x)^T t + \cU(x, t),
\end{align*}
with the functions $\cD : \bR^p \to \bR^{p + \sum_{j \in [p]} \pa(j)}$ and $\cU : \bR^p \times \bR^{p + \sum_{j \in [p]} \pa(j)} \to \bR$ defined as 
\begin{align*}
\cD(x) &:= (\cD_j(x) : j \in [p]), \quad \text{and}\\
\cU(x, t) &:= \sum_{j \in [p]} \cU_j(x, t) \quad \text{with} \quad \cU_j(x, t) := -u_j(x, t) - \frac{1}{2}\frac{t_{\eta, j}^2}{\tilde{\eta}_j^2} + o(t_{\eta, j}^2),
\end{align*} 
for some quantities $t_{\eta, j}, j \in [p]$ where $\cD_j$ and $u_j$ are defined in \eqref{eq:cdj} and \eqref{eq:ujxt}, respectively.
\end{lemma}
\begin{proof}
We have
\begin{align*}
&\log g(x , (\tilde{\eta}, \tilde{w}) + t) - \log g(x, (\tilde{\eta}, \tilde{w}))\\
&= \sum_{j \in [p]} \log (\tilde{\eta}_j + t_{\eta, j}) - \sum_{j \in [p]} |(\tilde{\eta}_j + t_{\eta, j}) x_j - x^T_{\pa(j)} (\tilde{w}_j + t_{w, j})|\\
&\qquad \qquad \qquad \qquad \qquad - \sum_{j \in [p]} \log \tilde{\eta}_j + \sum_{j \in [p]} |\tilde{\eta}_j x_j - x^T_{\pa(j)}\tilde{w}_j|\\
&=\sum_{j \in [p]} \lt\{\log (\tilde{\eta}_j + t_{\eta, j}) - \log \tilde{\eta}_j\rt\}\\
&\qquad \qquad \qquad \qquad \qquad - \sum_{j \in [p]} \lt\{|(\tilde{\eta}_j + t_{\eta, j}) x_j - x^T_{\pa(j)} (\tilde{w}_j + t_{w, j})| - |\tilde{\eta}_j x_j - x^T_{\pa(j)}\tilde{w}_j|\rt\}\\
&= \sum_{j \in [p]} \frac{t_{\eta, j}}{\tilde{\eta}_j} - \frac{1}{2}\frac{t_{\eta, j}^2}{\tilde{\eta}_j^2} + o(t_{\eta, j}^2)\\
&\qquad \qquad \qquad \qquad \qquad - \sum_{j \in [p]} \lt\{{\rm sgn}(\tilde{\eta}_j x_j - x^T_{\pa(j)}\tilde{w}_j) (t_{\eta, j} x_j - x_{\pa(j)}^T t_{w, j}) + u_j(x, t)\rt\}\\
&= \sum_{j \in [p]}\lt\{\lt(\frac{1}{\tilde{\eta}_j} - {\rm sgn}(\tilde{\eta}_j x_j - x^T_{\pa(j)}\tilde{w}_j)x_j\rt)t_{\eta, j} + {\rm sgn}(\tilde{\eta}_j x_j - x^T_{\pa(j)}\tilde{w}_j) x_{\pa(j)}^T t_{w, j}\rt\}\\
&\qquad \qquad \qquad \qquad \qquad + \sum_{j \in [p]} -u_j(x, t) - \frac{1}{2}\frac{t_{\eta, j}^2}{\tilde{\eta}_j^2} + o(t_{\eta, j}^2)\\
&= \sum_{j \in [p]} \cD_j(x)^Tt_j + \sum_{j \in [p]} \cU_j(x, t) = \cD(x)^Tt + \cU(x, t),
\end{align*}
where in the third equality, the first part is due to Taylor expansion and the second part follows from Lemma \ref{lem:HP_decomp}.
\end{proof}

In the next two lemmas we establish some properties of the random variables $\cD(X)$ and $\cU(X, t)$, where $t$, $\cD(\cdot)$, and $\cU(\cdot)$ are as appeared in Lemma \ref{lem:decomp}.

\begin{lemma}\label{lem:Ecov_D}
Under the assumption that $\sfE_*[\lambda_j^2] < \infty$ for every $j \in [p]$, $\cD(X)$ has zero mean and finite covariance matrix. 
\end{lemma}
\begin{proof}
From the definition of $(\tilde{\eta}_j, \tilde{w}_j)$, and due to \eqref{eq:score_eq}, $\cD_j(X)$ has zero mean for every $j \in [p]$, and thus, from the definition of $\cD(X)$ the first part is immediately proved.

Now, note that, following the representation in \eqref{eq:X_Beps}, we have, for every $k \in [p]$,
\begin{align*}
X_{k} &= \sum_{j \in {\an}^*(k)} \beta^*_{k \gets j} \epsilon_j + \epsilon_k, \qquad \text{which implies}\\\
\sfE_*[X_k^2] &= \sum_{j \in {\an}^*(k)} (\beta^*_{k \gets j})^2 \sfE_*[\epsilon_j^2] + \sfE_*[\epsilon_k^2]\\
&= \sum_{j \in {\an}^*(k)} (\beta^*_{k \gets j})^2 \sfE_*[\lambda_j^2] + \sfE_*[\lambda_k^2] < \infty.
\end{align*}
Thus, the second part follows from the fact that, for every $j, k, \ell, m \in [p]$, we have, by the Cauchy-Schwarz inequality,
\begin{align*}
&|{\sf Cov}_*(X_k {\rm sgn}(\tilde{\eta}_j X_j - X^T_{\pa(j)}\tilde{w}_j), X_\ell {\rm sgn}(\tilde{\eta}_m X_m - X^T_{\pa(m)}\tilde{w}_m))|^2\\
&\leq \Var_*[X_k {\rm sgn}(\tilde{\eta}_j X_j - X^T_{\pa(j)}\tilde{w}_j)] \; \Var_*[X_\ell {\rm sgn}(\tilde{\eta}_m X_m - X^T_{\pa(m)}\tilde{w}_m)]\\
&\leq \sfE_*[X_k^2]\sfE_*[X_\ell^2] < \infty.
\end{align*}
The proof is complete.
\end{proof}

\begin{lemma}\label{lem:E_U}
For every $t \in \bR^{p + \sum_{j \in [p]} \pa(j)}$, such that $(\tilde{\eta}, \tilde{w}) + t \in \times_{j \in [p]} (\bR^+ \times \bR^{\pa(j)})$,
\begin{align*}
\sfE_*[\cU(X, t)] = -\frac{1}{2}t^T J t + o(||t||^2),
\end{align*}
for some positive definite matrix $J$.
\end{lemma}
\begin{proof}
Following the definition, we have 
\begin{align}\label{eq:E_U}
\sfE_*[\cU(X, t)] = \sum_{j \in [p]} \sfE_*[\cU_j(X, t)] = \sum_{j \in [p]} - \sfE_*[u_j(X, t)] - \frac{1}{2}\frac{t_{\eta, j}^2}{\tilde{\eta}_j^2} + o(t_{\eta, j}^2).
\end{align}
Fix $j \in [p]$, and let $\mu_j := \sfE_*[u_j(X, t)]$. Then, using the fact that $\tilde{\eta}_j + t_{\eta, j} > 0$, it is not difficult to derive from the definition of $u_j(X, t)$ that
\begin{align*}
\mu_j &= \sfE_*[\sfE_*[u_j(X, t) | X_{\pa(j)}]]\\
&= \sfE_*\lt[\int_{X_{\pa(j)}^T\tilde{w}_j/\tilde{\eta}_j}^{X_{\pa(j)}^T(\tilde{w}_j + t_{w, j})/(\tilde{\eta}_j + t_{\eta, j})} 2(X^T_{\pa(j)}(\tilde{w_j} + t_{w, j}) - (\tilde{\eta}_j + t_{\eta, j})x) \sfp^*_{j|\pa(j)}(x|X_{\pa(j)})dx\rt],
\end{align*}
where $\sfp^*_{j|\pa(j)}$ denotes the conditional density of $X_j$ given $X_{\pa(j)}$. 

Now, by applying Leibniz integral rule we obtain that
\begin{align*}
\frac{\partial \mu_j}{\partial t_j} &= 2 \; \sfE_*\lt[\int_{X_{\pa(j)}^T\tilde{w}_j/\tilde{\eta}_j}^{X_{\pa(j)}^T(\tilde{w}_j + t_{w, j})/(\tilde{\eta}_j + t_{\eta, j})} 
\begin{bmatrix}
-x\\
X_{\pa(j)}
\end{bmatrix} 
\sfp^*_{j|\pa(j)}(x|X_{\pa(j)})dx\rt],\\
\frac{\partial^2 \mu_j}{\partial t_j^2} &= 2 \; \sfE_*\lt[
\begin{bmatrix}
- X^T_{\pa(j)} \frac{\tilde{w}_j + t_{w, j}}{\tilde{\eta}_j + t_{\eta, j}}\\
X_{\pa(j)}
\end{bmatrix} 
\begin{bmatrix}
- X^T_{\pa(j)} \frac{\tilde{w}_j + t_{w, j}}{(\tilde{\eta}_j + t_{\eta, j})^2}\\
X_{\pa(j)}\frac{1}{\tilde{\eta}_j + t_{\eta, j}}
\end{bmatrix}^T 
\rt].
\end{align*}
Thus, we have
\begin{align*}
\frac{\partial \mu_j}{\partial t_j}\bigg|_{t_j = 0} &= 0 \qquad \text{and}\\
\frac{\partial^2 \mu_j}{\partial t_j^2} \bigg|_{t_j = 0} &= W_j :=
2 \; \begin{bmatrix}
\frac{1}{\tilde{\eta}_j^3}(X^T_{\pa(j)}\tilde{w}_j)^2 & - \frac{1}{\tilde{\eta}_j^2}(X^T_{\pa(j)}\tilde{w}_j) X^T_{\pa(j)}\\
- \frac{1}{\tilde{\eta}_j^2}(X^T_{\pa(j)}\tilde{w}_j) X_{\pa(j)} & \frac{1}{\tilde{\eta}_j}X_{\pa(j)} X_{\pa(j)}^T
\end{bmatrix}.
\end{align*}
Therefore, by Taylor expansion we further have
\begin{align*}
\mu_j = \frac{1}{2}t_j^T W_j t_j + o(||t_j||^2),
\end{align*}
which, following \eqref{eq:E_U}, yields
\begin{align*}
\sfE_*[\cU(X, t)] &= \sum_{j \in [p]} - \frac{1}{2}t_j^T W_j t_j - \frac{1}{2}\frac{t_{\eta, j}^2}{\tilde{\eta}_j^2} + o(t_{\eta, j}^2) + o(||t_j||^2)\\
&= \sum_{j \in [p]} - \frac{1}{2}t_j^T J_j t_j + o(||t_j||^2)\\
&= - \frac{1}{2}t^T J t + o(||t||^2),
\end{align*}
where, for every $j \in [p]$, 
\begin{align*}
J_j :=
2 \; \begin{bmatrix}
\frac{1}{\tilde{\eta}_j^3}(X^T_{\pa(j)}\tilde{w}_j)^2 + \frac{1}{2 \tilde{\eta}_j^2} & - \frac{1}{\tilde{\eta}_j^2}(X^T_{\pa(j)}\tilde{w}_j) X^T_{\pa(j)}\\
- \frac{1}{\tilde{\eta}_j^2}(X^T_{\pa(j)}\tilde{w}_j) X_{\pa(j)} & \frac{1}{\tilde{\eta}_j}X_{\pa(j)} X_{\pa(j)}^T
\end{bmatrix},
\;\; \text{and} \;\; 
J =  
\begin{bmatrix}
J_1 &\bold{0} & \cdots & \bold{0}\\
\bold{0} &J_2 & \cdots & \bold{0}\\
\vdots & \vdots & \ddots & \vdots\\
\bold{0} & \bold{0} & \cdots & J_p
\end{bmatrix}.
\end{align*}
The proof is complete.
\end{proof}

\begin{lemma}\label{lem:Var_U}
For every $t \in \bR^{p + \sum_{j \in [p]} \pa(j)}$, such that $(\tilde{\eta}, \tilde{w}) + t \in \times_{j \in [p]} (\bR^+ \times \bR^{\pa(j)})$, 
\begin{align*}
{\sf Var}_*[\cU(X, t)] = o(||t||^2).
\end{align*}
\end{lemma}
\begin{proof}
Applying Cauchy-Schwarz inequality, we have
\begin{align}\label{eq:var_U}
\Var_*[\cU(X, t)] \leq p \sum_{j \in [p]} \Var_*[\cU_j(X, t)] = p \sum_{j \in [p]} \Var_*[u_j(X, t)] \leq p \sum_{j \in [p]} \sfE_*[u_j^2(X, t)].
\end{align}
Fix $j \in [p]$, and let $\sigma_j := \sfE_*[u_j^2(X, t)]$. Then, using the fact that $\tilde{\eta}_j + t_{\eta, j} > 0$, it is not difficult to derive from the definition of $u_j(X, t)$ that
\begin{align*}
\sigma_j &= \sfE_*[\sfE_*[u_j^2(X, t) | X_{\pa(j)}]]\\
&= \sfE_*\lt[\int_{X_{\pa(j)}^T\tilde{w}_j/\tilde{\eta}_j}^{X_{\pa(j)}^T(\tilde{w}_j + t_{w, j})/(\tilde{\eta}_j + t_{\eta, j})} 4(X^T_{\pa(j)}(\tilde{w_j} + t_{w, j}) - (\tilde{\eta}_j + t_{\eta, j})x)^2 \sfp^*_{j|\pa(j)}(x|X_{\pa(j)})dx\rt].
\end{align*}
Now, by applying Leibniz integral rule we obtain that
\begin{align*}
\frac{\partial \sigma_j}{\partial t_j} &= 8 \; \sfE_*\Bigg[\int_{X_{\pa(j)}^T\tilde{w}_j/\tilde{\eta}_j}^{X_{\pa(j)}^T(\tilde{w}_j + t_{w, j})/(\tilde{\eta}_j + t_{\eta, j})} 
(X^T_{\pa(j)}(\tilde{w_j} + t_{w, j}) - (\tilde{\eta}_j + t_{\eta, j})x)\\
&\qquad \qquad \qquad \qquad \qquad \qquad \qquad \qquad 
\times \begin{bmatrix}
-x\\
X_{\pa(j)}
\end{bmatrix} 
\sfp^*_{j|\pa(j)}(x|X_{\pa(j)})dx\Bigg],\\
\frac{\partial^2 \sigma_j}{\partial t_j^2} &= 8 \; \sfE_*\Bigg[\int_{X_{\pa(j)}^T\tilde{w}_j/\tilde{\eta}_j}^{X_{\pa(j)}^T(\tilde{w}_j + t_{w, j})/(\tilde{\eta}_j + t_{\eta, j})} 
\begin{bmatrix}
-x\\
X_{\pa(j)}
\end{bmatrix} 
\begin{bmatrix}
-x\\
X_{\pa(j)}
\end{bmatrix}^T 
\sfp^*_{j|\pa(j)}(x|X_{\pa(j)})dx\Bigg].
\end{align*}
Thus, we have both
\begin{align*}
\frac{\partial \sigma_j}{\partial t_j}\bigg|_{t_j = 0} = 0 \qquad \text{and} \qquad
\frac{\partial^2 \sigma_j}{\partial t_j^2} \bigg|_{t_j = 0} = 0,
\end{align*}
which by applying Taylor expansion yields that
\begin{align*}
\sigma_j = o(||t_j||^2).
\end{align*}
Therefore, following \eqref{eq:var_U}, we have
\begin{align*}
\Var_*[\cU(X, t)] \leq p \sum_{j \in [p]} \sigma_j = p \sum_{j \in [p]} o(||t_j||^2) = o(||t||^2),
\end{align*}
which completes the proof.
\end{proof}

Now, since $\gamma$ is fixed in the beginning, as done previously,  we omit the superscript from the notation $\cL(D_n | b^\gamma, \theta^\gamma, \gamma)$ and rewrite it as $\cL(D_n | b, \theta)$, for notational simplicity. Furthermore, we define $(\hat{b}_n, \hat{\theta}_n)$ to be the maximum likelihood estimator (MLE) of $(b, \theta)$, i.e.,
\begin{align*}
(\hat{b}_n, \hat{\theta}_n) := \argmax_{(b, \theta)} \; \cL(D_n | b, \theta).
\end{align*}
Next, after the reparameterization from $(b, \theta)$ to $(\eta, w)$, the likelihood function can be equivalently expressed as $\cL(D_n | b, \theta) = \prod_{i \in [n]}  g(X^{(i)}, (\eta, w))$, and the log-likelihood function and the MLE are denoted as $\ell_n(\eta, w)$, and $(\hat{\eta}_n, \hat{w}_n)$, respectively, i.e.,
\begin{align}\label{eq:likli_etaw}
\ell_n(\eta, w) := \sum_{i \in [n]}  \log g(X^{(i)}, (\eta, w)), \quad \text{and} \quad (\hat{\eta}_n, \hat{w}_n) = \argmax_{(\eta, w)} \; \ell_n(\eta, w).
\end{align}
Furthermore, for every $t \in \bR^{p + \sum_{j \in [p]} \pa(j)}$, we define the following function:
\begin{align}\label{eq:def_An}
A_n(t) := \ell_n(\hat{\eta}_n, \hat{w}_n) - \ell_n((\hat{\eta}_n, \hat{w}_n) + t/\sqrt{n}).
\end{align}
\begin{lemma}\label{lem:prop_A}
$A_n(\cdot)$ satisfies the following properties:
\begin{itemize}
\item[(i)] $A_n(0) = 0$,
\item[(ii)] $A_n(\cdot)$ is convex, and
\item[(iii)] for every compact set $K \subset \bR^{p + \sum_{j \in [p]} \pa(j)}$, we have, in $\sfP^*$-probability,
\begin{align*}
\sup_{t \in K} \lt|A_n(t) -  \frac{1}{2}t^T J t\rt| \to 0,
\end{align*}
where $J$ is defined in the proof of Lemma \ref{lem:E_U}.
\end{itemize}
\end{lemma}
\begin{proof}
From the definition of $A_n(t)$, (i) is immediate, and (ii) follows due to log-concavity of $g(x, (\eta, w))$ as proved in Lemma \ref{lem:log-concave}. In order to prove (iii), we use \cite[Theorem 4.1, Theorem 4.2]{hjort2011asymptotics}.
To be specific, all conditions of \cite[Theorem 4.1]{hjort2011asymptotics} are satisfied due to Lemma \ref{lem:log-concave}, \ref{lem:decomp}, \ref{lem:Ecov_D}, \ref{lem:E_U} and \ref{lem:Var_U}. Thus, following the proof techniques of \cite[Theorem 4.2]{hjort2011asymptotics}, property (iii) holds.
\end{proof}

\begin{lemma}\label{lem:An_bound}
Let $\xi_0 := \inf_{||t|| = 1} \frac{1}{2}t^T J t$. Then
\begin{align*}
\sfP^*\lt(A_n(t) \ \mathbbm{1}\{||t|| > 1\} \geq \frac{1}{2} \xi_0||t|| \ \mathbbm{1}\{||t|| > 1\}\rt) \to 1.
\end{align*}
\end{lemma}
\begin{proof}
Fix $t \in \bR^{p + \sum_{j \in [p]} \pa(j)}$ such that $||t|| > 1$. Then we can write $t = a \times u$, where $a = ||t|| > 1$ and $u = t/||t||$, and by using convexity of $A_n(\cdot)$, as proved in Lemma \ref{lem:prop_A}(ii), we have
\begin{align*}
\lt(1 - \frac{1}{a}\rt)A_n(0) + \frac{1}{a}A_n(t) \geq A_n\lt(\frac{t}{a}\rt) = A_n(u),
\end{align*}
which, by the fact that $A_n(0) = 0$, as proved in Lemma \ref{lem:prop_A}(i), further yields that
\begin{align*}
\frac{1}{||t||}A_n(t) \geq A_n(u) &= \frac{1}{2}u^t J u + \lt(A_n(u) - \frac{1}{2}u^t J u\rt)\\
&\geq \inf_{||v|| = 1} \frac{1}{2}v^t J v - \lt|A_n(u) - \frac{1}{2}u^t J u\rt|\\
&\geq \xi_0 - \sup_{||v|| = 1} \lt|A_n(v) - \frac{1}{2}v^t J v\rt|,
\end{align*}
where the second and third inequalities follow since $||u|| = 1$.
Now, by Lemma \ref{lem:prop_A}(iii), we have, in $\sfP^*$-probability,
\begin{align*}
\sup_{||v|| = 1} \lt|A_n(v) - \frac{1}{2}v^t J v\rt| \to 0,
\end{align*}
and thus, the result follows immediately.
\end{proof}

\begin{lemma}\label{lem:conv_Ln}
We have, in $\sfP^*$-probability,
\begin{align*}
\frac{1}{n} \log \cL(D_n | \hat{b}_n, \hat{\theta}_n) = - \min_{(b, \theta)} \; H(b, \theta) + O_p(1/\sqrt{n}).
\end{align*}
\end{lemma}
\begin{proof}
Following \eqref{eq:likeli_L}, we have 
\begin{align*}
\log \cL\lt(D_n\rt | b, \theta) &= -np \log 2 - n \sum_{j \in [p]} \log \theta_j - \sum_{j \in [p]} \frac{1}{\theta_j} \sum_{i \in [n]} \lt|X_j^{(i)} - b_j^TX_{\pa(j)}^{(i)}\rt|
\end{align*}
which yields, by letting $\hat{b}_{j, n}$ and $\hat{\theta}_{j, n}$ be the MLE for $b_j$ and $\theta_j$, respectively for every $j \in [p]$, and applying Lemma \ref{lem:min_f}, that
\begin{align*}
\hat{b}_{j, n} &= \min_{b_j} \sum_{i \in [n]} \lt|X_j^{(i)} - b_j^TX_{\pa(j)}^{(i)}\rt|,\\
\hat{\theta}_{j, n} &= \frac{1}{n}\sum_{i \in [n]} \lt|X_j^{(i)} - \hat{b}_{j, n}^TX_{\pa(j)}^{(i)}\rt|.
\end{align*}
Thus, plugging the above values we have
\begin{align}\label{eq:avg_L}
\frac{1}{n} \log \cL\lt(D_n\rt | \hat{b}_n, \hat{\theta}_n) = -p(1 + \log 2) - \sum_{j \in [p]} \log \hat{\theta}_{j, n},
\end{align}
and also following from Lemma \ref{lem:minH}, we have
\begin{align}\label{eq:targetH}
\min_{(b, \theta)} \; H(b, \theta) \; = \; p (1 + \log 2) \; + \;  \sum_{j \in [p]} \log  \lt(\min_{b_j} \; \sfE_*\lt[\lt|X_j - b_j^TX_{\pa(j)}\rt|\rt]\rt).
\end{align}
Now, from the consistency of MLE due to \cite[Theorem 2.1]{hjort2011asymptotics}, it follows that, for every $j \in [p]$, in $\sfP^*$-probability, we have
\begin{align*}
\hat{\theta}_{j, n} \; \to \; \min_{b_j} \; \sfE_*\lt[\lt|X_j - b_j^TX_{\pa(j)}\rt|\rt],
\end{align*}
and in fact, the above holds with $\sqrt{n}$-consistency further leading to
\begin{align*}
\log \hat{\theta}_{j, n} \; = \; \log \lt(\min_{b_j} \; \sfE_*\lt[\lt|X_j - b_j^TX_{\pa(j)}\rt|\rt]\rt) + O_p(1/\sqrt{n}).
\end{align*}
Therefore, by comparing \eqref{eq:avg_L} and \eqref{eq:targetH} the proof is complete.
\end{proof}

\begin{lemma}\label{lem:wilks_missp}
If $\gamma \in \cS^*$ then we have
\begin{align*}
    \log \cL(D_n | \hat{b}^*_n, \hat{\theta}^*_n, \gamma^*) - \log \cL(D_n | \hat{b}_n, \hat{\theta}_n) = O_p(1).
\end{align*}
\end{lemma}
\begin{proof}
Since $\gamma \supseteq \gamma^*$, by leting $b^{-*} := (b_{jk} : j \in [p], k \in \pa(j) \setminus \pa^*(j))$, and $b^{+*} := (b_{jk} : j \in [p], k \in \pa^*(j))$  we write $b = (b^{+*}, b^{-*})$. Then clearly,
\begin{align*}
    \cL(D_n | \hat{b}^*_n, \hat{\theta}^*_n, \gamma^*) = \max_{(b^*, \theta^*)} \cL(D_n | b^*, \theta^*, \gamma^*) = \max_{(b, \theta) :\, b^{-*} = 0} \cL(D_n | b, \theta) = \cL(D_n | (\hat{b}^*_n, 0), \hat{\theta}^*_n).
\end{align*}
Thus, using the reparameterization, if we let
    \begin{align*}
        (\hat{\eta}_n^*,\hat{w}^*_n) := \argmax_{(\eta^*, w^*)} \; \ell_n^*(\eta^*, w^*), 
    \end{align*}
    then we analogously have
    \begin{align*}
        (\hat{\eta}_n^*, (\hat{w}^*_n, 0)) = \argmax_{(\eta, w): \; w^{-*} = 0} \; \ell_n(\eta, w), \quad \text{that is}, \quad  \ell_n(\hat{\eta}_n^*, (\hat{w}^*_n, 0)) = \log \cL(D_n | (\hat{b}^*_n, 0), \hat{\theta}^*_n).
    \end{align*}
    Now, letting $t_n = -\sqrt{n}(\hat{\eta}_n - \hat{\eta}_n^*, \hat{w}_n - (\hat{w}_n^*, 0))$, we have
    \begin{align} \notag
        &\log \cL(D_n | \hat{b}_n, \hat{\theta}_n) - \log \cL(D_n | \hat{b}^*_n, \hat{\theta}^*_n, \gamma^*)\\ \notag
        &= \ell_n(\hat{\eta}_n, \hat{w}_n) - \ell_n(\hat{\eta}_n^*, (\hat{w}^*_n, 0)) = A_n(t_n)\\ \notag
        &\leq \frac{1}{2}t_n^TJt_n + \lt|A_n(t_n) - \frac{1}{2}t_n^TJt_n\rt|\\ \label{ineq:wilks_decomp}
        &\leq \frac{1}{2}t_n^TJt_n + \sup_{t \in K}\lt|A_n(t) - \frac{1}{2}t^TJt\rt| + \lt|A_n(t_n) - \frac{1}{2}t_n^TJt_n\rt|\mathbbm{1}\{t_n \notin K\},
    \end{align}
    where the first equality follows from the definition in \eqref{eq:def_An}, and $K$ is some compact set. 
    Furthermore, since $\gamma \in \cS^*$, we have $\tilde{\eta} = \tilde{\eta}^*$, $\tilde{w} = (\tilde{w}^*, 0)$, and thus,
    \begin{align*}
        t_n &= -\sqrt{n}(\hat{\eta}_n - \hat{\eta}_n^*, \hat{w}_n - (\hat{w}_n^*, 0))\\
        &= -\sqrt{n}(\hat{\eta}_n - \tilde{\eta}^*, \hat{w}_n - (\tilde{w}^*, 0)) + \sqrt{n}(\hat{\eta}^*_n - \tilde{\eta}^*, (\hat{w}^*_n - \tilde{w}^*, 0))\\
        &= -\sqrt{n}((\hat{\eta}_n , \hat{w}_n) - (\tilde{\eta}, \tilde{w})) + \sqrt{n}(\hat{\eta}^*_n - \tilde{\eta}^*, (\hat{w}^*_n - \tilde{w}^*, 0)).
    \end{align*}
    Therefore, $t_n = O_p(1)$ due to $\sqrt{n}$-consistency of the quantities $(\hat{\eta}_n^* - \tilde{\eta}^*)$ and $(\hat{w}_n^* - \tilde{w}^*)$ and  $((\hat{\eta}_n , \hat{w}_n) - (\tilde{\eta}, \tilde{w}))$ which again follows from \cite[Theorem 2.1]{hjort2011asymptotics}. As a consequence, the first and third terms in \eqref{ineq:wilks_decomp} are $O_p(1)$, and the second term is $o_p(1)$ due to Lemma \ref{lem:prop_A}(iii). The proof is complete.
\end{proof}

\subsection{Proof of Theorem \ref{thm:lap_app}}\label{pf:thm:lap_app}
\begin{proof}
Following \eqref{eq:marg_m}, we have 
\begin{align*}
m\lt(D_n | \gamma\rt)
&= \int \cL\lt(D_n\rt | b, \theta) \prod_{j \in [p]} \Big(\pi_\theta(\theta_j)d\theta_j \prod_{k \in \pa(j)}\pi_b(b_{jk}) db_{jk}\Big)\\
&= \int \exp\lt(\log \cL\lt(D_n\rt | b, \theta)\rt) \prod_{j \in [p]} \Big(\pi_\theta(\theta_j)d\theta_j \prod_{k \in \pa(j)}\pi_b(b_{jk}) db_{jk}\Big)\\
&= \int \exp (\ell_n(\eta, w)) \pi(\eta, w) d\eta dw\\
&= \exp (\ell_n(\hat{\eta}_n, \hat{w}_n)) \int \exp (\ell_n(\eta, w) - \ell_n(\hat{\eta}_n, \hat{w}_n)) \pi(\eta, w) d\eta dw\\
&= \exp (\ell_n(\hat{\eta}_n, \hat{w}_n)) \; n^{-\frac{p + |\gamma|}{2}}\int \exp (\ell_n((\hat{\eta}_n, \hat{w}_n) + t/\sqrt{n}) - \ell_n(\hat{\eta}_n, \hat{w}_n))\\ 
&\qquad \qquad \qquad \qquad \qquad \qquad \qquad \qquad \qquad \times \pi((\hat{\eta}_n, \hat{w}_n) + t/\sqrt{n}) dt\\
&= \exp (\ell_n(\hat{\eta}_n, \hat{w}_n)) \; n^{-\frac{p + |\gamma|}{2}} \int \exp(-A_n(t)) \pi((\hat{\eta}_n, \hat{w}_n) + t/\sqrt{n}) dt,
\end{align*}
where the third equality follows from the definition in \eqref{eq:likli_etaw} with $\pi(\cdot)$ being the equivalent prior distribution of $(\eta, w)$, the fifth one follows from the change of variable $(\eta, w) \to (\hat{\eta}_n, \hat{w}_n) + t/\sqrt{n}$, and the last one follows from the definition in \eqref{eq:def_An}. 

Thus, following the above, we have
\begin{align*}
\log m\lt(D_n | \gamma\rt) = \ell_n(\hat{\eta}_n, \hat{w}_n) - \frac{p + |\gamma|}{2} \log n + \log  \int \exp(-A_n(t)) \pi((\hat{\eta}_n, \hat{w}_n) + t/\sqrt{n}) dt.
\end{align*}
Now, since the reparameterization $(b, \theta) \to (\eta, w)$ is one-one, we have
\begin{align*}
\ell_n(\hat{\eta}_n, \hat{w}_n) = \log \cL(D_n | \hat{b}_n, \hat{\theta}_n) = - n \, \min_{(b, \theta)} \; H(b, \theta)(1 + O_p(1/\sqrt{n})),
\end{align*}
where the second equality follows from Lemma \ref{lem:conv_Ln}.

Next, let $C_\pi := \sup_{(\eta, w)}\pi (\eta, w)$. Then, by using the fact that $A_n(t) \geq 0$, and following Lemma \ref{lem:An_bound}, we obtain the following result regarding the integrand above,
\begin{align*}
\sfP^*\lt(\exp(-A_n(t)) \pi((\hat{\eta}_n, \hat{w}_n) + t/\sqrt{n}) \leq \bar{A}(t)\rt) \to 1,
\end{align*}
where 
\begin{align*}
\bar{A}(t) := 
\begin{cases}
C_\pi \quad &\text{if} \;\; ||t|| \leq 1,\\
\frac{1}{2}C_\pi\xi_0 ||t|| \quad &\text{if} \;\; ||t|| > 1.
\end{cases}
\end{align*}
Clearly, $\bar{A}(\cdot)$ is an integrable function, i.e., $\int \bar{A}(t) dt < \infty$, and also, by Lemma \ref{lem:prop_A}(iii) and consistency of MLE,
we have, in $\sfP^*$-probability,
\begin{align*}
\exp(-A_n(t)) \pi((\hat{\eta}_n, \hat{w}_n) + t/\sqrt{n}) \to \pi(\tilde{\eta}, \tilde{w}) \exp(-(1/2)t^T J t),
\end{align*}
where we recall that $(\tilde{\eta}, \tilde{w})$ is the solution of \eqref{eq:score_eq}.
Thus, by applying the dominated convergence theorem, we have, in $\sfP^*$-probability,
\begin{align*}
&\int \exp(-A_n(t)) \pi((\hat{\eta}_n, \hat{w}_n) + t/\sqrt{n}) dt\\
&\to \; \pi(\tilde{\eta}, \tilde{w}) \int \exp\lt(-\frac{1}{2}t^T J t\rt) dt = 
\pi(\tilde{\eta}, \tilde{w}) (2\pi)^{\frac{p + |\gamma|}{2}}  \sqrt{\det(J)}. 
\end{align*}
Therefore, defining $c_\gamma$ as the logarithm of the above limiting value, the proof is complete.
\end{proof}

\section{Proofs Regarding Posterior Consistency}\label{app:post_const}

\subsection{Some important lemmas}

In this subsection, we establish some lemmas that are critical in establishing the posterior consistency results in Section \ref{sec:asym_prop}.

\begin{lemma}\label{lem:BF_conv}
    Suppose that \eqref{assum:fin_sec} holds. Then for every $\gamma \in \Gamma^p$, we have
    \begin{align*}
        \log \sfBF_n(\gamma^*, \gamma) = 
        \begin{cases}
            \frac{\psi_\gamma}{2} \, \log n + O_p(1) \qquad &\text{if} \;\; \gamma \in \cS^*\\
            n \, \delta_\gamma + \frac{\psi_\gamma}{2} \, \log n + O_p(\sqrt{n}) \qquad &\text{otherwise}
        \end{cases}.
    \end{align*}
\end{lemma}
\begin{proof}
First, fix $\gamma \in \cS^*$. Following the definition in \eqref{eq:def_BF} and by applying the first Laplace approximation from Theorem \ref{thm:lap_app}, we obtain that, by letting $c^*_\gamma := c_{\gamma^*} - c_\gamma$,
    \begin{align*}
        &\log \sfBF_n(\gamma^*, \gamma)\\
        &= \log m\lt(D_n | \gamma^*\rt) - \log m\lt(D_n | \gamma\rt)\\
        &= \max_{(b^*, \theta^*)} \log \cL(D_n | b^*, \theta^*, \gamma^*) - \max_{(b^\gamma, \theta^\gamma)} \log \cL(D_n | b^\gamma, \theta^\gamma, \gamma) + \frac{|\gamma| - |\gamma^*|}{2} \log n + c^*_\gamma + o_p(1)\\
        &=O_p(1)\; + \; \frac{\psi_\gamma}{2}\, \log n \; + \; c^*_\gamma \; + \; o_p(1) = \frac{\psi_\gamma}{2}\, \log n \; + \; O_p(1),
    \end{align*}
    where the last equality follows from Lemma \ref{lem:wilks_missp} 
    since $\gamma \supseteq \gamma^*$.
    
    Next, fix $\gamma \notin \cS^*$. Again, following the definition in \eqref{eq:def_BF} and by applying the second Laplace approximation from Theorem \ref{thm:lap_app}, we obtain that, for some $R_n^*, R_n^\gamma = O_p\lt({1}/{\sqrt{n}}\rt)$,
    \begin{align*}
        &\log \sfBF_n(\gamma^*, \gamma)\\
        &= \log m\lt(D_n | \gamma^*\rt) - \log m\lt(D_n | \gamma\rt)\\
        &= - \, n \, H^*(\tilde{b}^*, \tilde{\theta}^*)\lt(1 + R^*_n\rt) \; + \, n \, H^\gamma(\tilde{b}^\gamma, \tilde{\theta}^\gamma)\lt(1 + R^\gamma_n\rt) \; + \; \frac{|\gamma| - |\gamma^*|}{2} \log n \; + \; c^*_\gamma \; + \; o_p(1)\\
        &=n \, \delta_\gamma + n \, (H^\gamma(\tilde{b}^\gamma, \tilde{\theta}^\gamma)R_n^\gamma - H^*(\tilde{b}^*, \tilde{\theta}^*)R_n^*)\; + \; \frac{\psi_\gamma}{2} \log n \; + \; c^*_\gamma \; + \; o_p(1).
    \end{align*}
    Clearly, the second term in the above is $O_p(\sqrt{n})$, and 
    the result follows. 
\end{proof}

\begin{lemma}\label{lem:PrE}
If for every $\gamma \notin \cE^*$, we have \, $\Pi_n(\gamma^*, \gamma) \to \infty$ \, in $\sfP^*$-probability, then
\begin{align*}
    {\rm Pr}(\gamma \in \cE^* | D_n) \; {\to} \; 1,\quad \text{in} \;\; \sfP^*\text{-probability}.
    \end{align*}
\end{lemma}
\begin{proof}
For every $\gamma \notin \cE^*$, we have
\begin{align*}
\pi(\gamma | D_n) &= \frac{m\lt(D_n \rt | \gamma) \times \pi_g(\gamma)}{\sum_{\gamma' \in \Gamma^p}m\lt(D_n \rt | \gamma') \times \pi_g(\gamma')}\\
&= \frac{1}{\sum_{\gamma' \in \Gamma^p} \Pi_n(\gamma', \gamma)}\\
&= \frac{1}{\Pi_n(\gamma^*, \gamma) + \sum_{\gamma' \neq \gamma^*} \Pi_n(\gamma', \gamma)} \quad \to \; 0, \quad \text{in} \;\; \sfP^*\text{-probability},
\end{align*}
where the convergence holds since $\Pi_n(\gamma^*, \gamma) \to \infty$ in $\sfP^*$-probability, and $\Pi_n(\gamma', \gamma) \geq 0$ for every $\gamma' \in \Gamma^p$. Therefore, using the above result we have
\begin{align*}
    1 - {\rm Pr}(\gamma \in \cE^* | D_n) =  \sum_{\gamma \notin \cE^*} \pi(\gamma | D_n)\quad \to \; 0, \quad \text{in} \;\; \sfP^*\text{-probability}.
\end{align*}
The proof is complete.
\end{proof}

\subsection{Proof of Theorem \ref{thm:post_cons_exac}}\label{pf:thm:post_cons_exac}

\begin{proof}
In view of Lemma \ref{lem:PrE}, it suffices to have $\Pi_n(\gamma^*, \gamma) \to \infty$ in $\sfP^*$-probability for every $\gamma \neq \gamma^*$, as shown below. 

Following the definition in \eqref{eq:def_BF} we have, for every $\gamma \in \Gamma^p$,
\begin{align*}
\log \Pi_n(\gamma^*, \gamma) &= \log \sfBF_n(\gamma^*, \gamma) + \log (\pi_g(\gamma^*)/\pi_g(\gamma))\\
&= 
\begin{cases}
        \frac{\psi_\gamma}{2} \, \log n + O_p(1) \qquad &\text{if} \;\; \gamma \in \cS^*\\
            n \, \delta_\gamma + \frac{\psi_\gamma}{2} \, \log n + O_p(\sqrt{n}) \qquad &\text{otherwise}
        \end{cases},
\end{align*}
where the second equality follows from Lemma \ref{lem:BF_conv}, and the fact that $|\log(\pi_g(\gamma^*)/\pi_g(\gamma))| \leq \log C$.
Now, $\psi_\gamma > 0$ for every $\gamma \in \cS^*$, $\gamma \neq \gamma^*$, and since $\bar{\cE}^* = \cS^*$, we have $\delta_\gamma > 0$ for every $\gamma \notin \cS^*$. Thus, from the above we have $\log \Pi_n(\gamma^*, \gamma) \to \infty$, in $\sfP^*$-probability for every $\gamma \neq \gamma^*$. The proof is complete.
\end{proof}

\subsection{Proof of Theorem \ref{thm:post_cons_gen}}\label{pf:thm:post_cons_gen}

\begin{proof}
Following Theorem \ref{thm:E=DE}, we have $\cE^* = \cE(\gamma^*, n\cG^*)$. Thus, in view of Lemma \ref{lem:PrE}, it suffices to have $\Pi_n(\gamma^*, \gamma) \to \infty$ in $\sfP^*$-probability for every $\gamma \notin \cE^*$, as shown below. 

Following the definition in \eqref{eq:def_BF} we have, for every $\gamma \in \Gamma^p$,
\begin{align*}
\log \Pi_n(\gamma^*, \gamma) &= \log \sfBF_n(\gamma^*, \gamma) + \log (\pi_g(\gamma^*)/\pi_g(\gamma))\\
&= n \, \delta_\gamma +  \frac{\psi_\gamma}{2} \, \log n + O_p(\sqrt{n}) + n^\alpha d_n (|\gamma| - |\gamma^*|)\\
&= n\, \delta_\gamma  +  n^\alpha\,d_n\psi_\gamma +  \frac{\psi_\gamma}{2}\, \log n + O_p(\sqrt{n})\\
&= n\, \delta_\gamma  + n^\alpha(d_n\psi_\gamma + O_p(n^{1/2-\alpha})) + \frac{\psi_\gamma}{2}\, \log n,
\end{align*}
where the second equality follows from Lemma \ref{lem:BF_conv}. 
Now, when $\gamma \notin \bar{\cE}^*$, we have $\delta_\gamma >0$, and since $\alpha < 1$, the above suggests $\log \Pi_n(\gamma^*, \gamma) \to \infty$ in $\sfP^*$-probability regardless of the sign of $\psi_\gamma$. Furthermore, when $\gamma \in \bar{\cE}^* \setminus \cE^*$, we have $\delta_\gamma = 0$ but $\psi_\gamma > 0$, and since $\alpha > 1/2$ and $0 < d_n = O_p(1)$, again from the above $\log \Pi_n(\gamma^*, \gamma) \to \infty$ in $\sfP^*$-probability. The proof is complete.
\end{proof}

\end{appendix}

\bibliographystyle{imsart-number} 
\bibliography{BCDnG.bib}       


\end{document}